\documentclass{amsart}
\usepackage{eurosym}
\usepackage{amssymb}
\usepackage{lscape}
\usepackage{amsmath}
\usepackage{amsmath}
\usepackage{amsfonts}
\usepackage{amsthm}
\usepackage{amssymb}
\usepackage{amscd}
\usepackage{mathrsfs}
\usepackage{pb-diagram,pb-xy}
\usepackage[cmtip,arrow]{xy}
\usepackage{graphicx}

\xyoption{all}
\newtheorem{theorem}{Theorem}
\theoremstyle{plain}

\newtheorem{corollary}{Corollary}

\newtheorem{definition}{Definition}

\newtheorem{lemma}{Lemma}

\newtheorem{proposition}{Proposition}
\newtheorem{remark}{Remark}

\numberwithin{equation}{section}

\begin{document}
\title[Classical Tilting]{Tilting theory and functor categories I.\\
Classical Tilting}
\author{R. Mart\'{\i}nez-Villa}
\address{Instituto de Matem\'{a}ticas UNAM, Unidad Morelia, Mexico}
\email{mvilla@matmor.unam.mx}
\thanks{}
\author{M. Ortiz-Morales}
\address{Instituto de Matem\'{a}ticas UNAM, Unidad Morelia, Mexico}
\email{mortiz@matmor.unam.mx}
\urladdr{http://www.matmor.unam.mx}
\thanks{The second author thanks CONACYT for giving him financial support
during his graduate studies. }
\date{December 15, 2010}
\dedicatory{}
\thanks{This paper is in final form and no version of it will be submitted
for publication elsewhere.}

\begin{abstract}
Tilting theory has been a very important tool in the classification of
finite dimensional algebras of finite and tame representation type, as well
as, in many other branches of mathematics. Happel [Ha] proved that
generalized tilting induces derived equivalences between module categories,
and tilting complexes were used by Rickard [Ri] to develop a general
Morita theory of derived categories.

In the other hand, functor categories were introduced in representation
theory by M. Auslander [A], [AQM] and used in his proof of the first
Brauer-Thrall conjecture [A2] and later on, used systematically in his
joint work with I. Reiten on stable equivalence [AR], [AR2] and many other
applications.

Recently, functor categories were used in [MVS3] to study the
Auslander-Reiten components of finite dimensional algebras.

The aim of the paper is to extend tilting theory to arbitrary functor
categories, having in mind applications to the functor category $\mathrm{Mod}%
(\mathrm{mod}_{\Lambda })$, \ with \ $\Lambda $ a finite dimensional algebra.
\end{abstract}

\maketitle

\section{\protect\bigskip Introduction and basic results}

Tilting theory traces back his history to the article by Bernstein, Gelfand
and Ponomarev ([BGP] 1973), where they defined partial Coxeter functors.
These functors were generalized by Auslander Platzeck and Reiten ([APR]
1979), but a major brake through was the article by Brenner and Butler (%
[BB]  1979). The notion of tilting was further generalized by Miyashita (%
[Mi] 1986) and Happel ([Ha] 1987). Happel showed that generalized tilting
induces derived equivalence between the corresponding module categories.
These results inspired Rickard [Ri] to develop his Morita theory of
derived categories.

This paper is the first one in a series of articles in which, having in mind
applications to functor categories from subcategories of modules over a
finite dimensional algebra to the category of abelian groups, we generalize
tilting theory from modules to functor categories.

The first paper is dedicated to classical tilting and it consists of four
sections:

In the first section we fix the notation and recall some notions from
functor categories that will be used through the paper. In the second
section, we generalize Bongartz\'{}s proof [B] of Brenner-Butler%
\'{}%
s theorem [BB] to arbitrary functor categories. In the third section we
deal with two special cases: first we extend tilting from subcategories to
full categories, and second we apply the theory so far developed, to study
tilting for infinite quivers with no relations, and apply it to compute the
Auslander-Reiten components of infinite Dynkin quivers. In the last section
we recall results from [MVS1], [MVS2], [MVS3] on graded and Koszul
categories and we apply our results from the third section, to compute the
Auslander-Reiten components of the categories of Koszul functors over a
regular Auslander-Reiten component of a finite dimensional algebra. These
results generalize previous work by [MV] on the preprojective algebra.

\subsection{The Category $\mathrm{Mod}(\mathcal{C})$}

In this section $\mathcal{C}$ will be an arbitrary skeletally small pre
additive category, and $\mathrm{Mod}(\mathcal{C})$ will denote the category
of contravariant functors from $\mathcal{C}$ to the category of abelian
groups. Following the approach by Mitchel [M], we can think of $\mathcal{C}$
as a ring "with several objects" and $\mathrm{Mod}(\mathcal{C})$ as a
category of $\mathcal{C}$-modules. The aim of the paper is to show that the
notions of tilting theory can be extended to$\mathcal{\ }\mathrm{Mod}(%
\mathcal{C)},$ to obtain generalizations of the main theorems on tilting for
rings. To obtain generalizations of some tilting theorems for finite
dimensional algebras, we need to add restrictions on our category $\mathcal{C%
}$, like: existence of pseudo kernels, Krull-Schmidt, $\mathrm{Hom}$-finite,
dualizing, etc.. To fix the notation, we recall known results on functors
and categories that we use through the paper, referring for the proofs to
the papers by Auslander and Reiten [A], [AQM], [AR].

\subsection{Functor Categories}

Let $\mathcal{C}$ be a pre additive skeletally small category. By $\mathrm{%
Mod}(\mathcal{C})$ we denote the category of additive contravariant functors
from $\mathcal{C}$ to the category of abelian groups. Then, $\mathrm{Mod}(%
\mathcal{C})$ is an abelian category with arbitrary sums and products, in
fact it has arbitrary limits and colimits, and filtered limits are exact
(Ab5 in Grothendiek terminology). It has enough projective and injective
objects. For any object $C\in \mathcal{C}$, the representable functor $%
\mathcal{C}(\;,C)$ is projective, arbitrary sums of representable functors
are projective, and any object $M\in \mathrm{Mod}(\mathcal{C})$ is covered
by an epimorpism $\coprod_{i}\mathcal{C}(\;,C_{i})\rightarrow M\rightarrow 0$%
. We say that an object $M$ in $\mathrm{Mod}(C)$ is finitely generated if
there exists an epimorphism $\coprod_{i\in I}\mathcal{C}(\;,C_{i})%
\rightarrow M\rightarrow 0$, with $I$ a finite set.

Given a finitely generated functor $M$ and and an arbitrary sum of functors $%
\coprod_{j}N_{j},$ there is a natural isomorphism $\mathrm{Hom}_{\mathcal{C}%
}(M,\coprod_{j}N_{j})\cong \coprod_{j}\mathrm{Hom}_{\mathcal{C}}(M,N_{j})$
(finitely generated are compact).

An object $P$ in $\mathrm{Mod}(\mathcal{C})$ is projective (finitely
generated projective) if and only if $P$ is summand of $\coprod_{i\in I}%
\mathcal{C}(\;,C_{i})$ for a family (finite) $\{C_{i}\}_{i\in I}$ of objects
in $\mathcal{C}$. The subcategory $\mathfrak{p}(\mathcal{C})$ of $\mathrm{Mod%
}(\mathcal{C})$, of all finitely generated projective objects, is a
skeletally small additive category in which idempotents split, the functor $%
P:\mathcal{C}\rightarrow \mathfrak{p}(\mathcal{C})$, $P(C)=\mathcal{C}(\;,C)$%
, is fully faithful and induces by restriction $\mathrm{res}:\mathrm{Mod}(%
\mathfrak{p}(\mathcal{C}))\rightarrow \mathrm{Mod}(\mathcal{C})$, an
equivalence of categories [A]. For this reason we may assume that our
categories are skeletally small additive categories such that idempotents
split, they were called annuli varieties in [A].

\begin{definition}
\emph{[AQM] } Given a preadditive skeletally small category $\mathcal{C}$, we
say $\mathcal{C}$ has pseudokernels, if given a map $f:C_{1}\rightarrow
C_{0} $ there exists a map $g:C_{2}\rightarrow C_{1}$ such that the
sequence: $\mathcal{C}(\;C_{2})\overset{(\;,g)}{\rightarrow }\mathcal{C}%
(\;C_{1})\overset{(\;,f)}{\rightarrow }\mathcal{C}(\;,C_{0})$ is exact.
\end{definition}

A functor $M$ is finitely presented if there exists an exact sequence $%
\mathcal{C}(\;C_{1})\rightarrow \mathcal{C}(\;,C_{0})\rightarrow
M\rightarrow 0$. We denote by $\mathrm{mod}(\mathcal{C})$ the full
subcategory of $\mathrm{Mod}(\mathcal{C})$ consisting of finitely presented
functors. It was proved in [AQM] $\mathrm{mod}(\mathcal{C})$ is abelian if
and only if $\mathcal{C}$ has pseudokernels.

We will say indistinctly that $M$ is an object of $\mathrm{Mod}(\mathcal{C})$
or that $M$ is a $\mathcal{C}$-module. A representable functor $\mathcal{C}%
(\;,C)$ will be sometimes denoted by $(\;,C).$

\subsection{Change of Categories}

According to [A], there exists a unique up to isomorphism functor $-\otimes
_{\mathcal{C}}-:\mathrm{Mod}(\mathcal{C}^{op})\times \mathrm{Mod}(\mathcal{C}%
)\rightarrow \mathbf{Ab,}$ called the \textbf{tensor product}, with the
following properties:

\begin{itemize}
\item[(a)]

\begin{itemize}
\item[(i)] For each $\mathcal{C}$-module $N$, the functor $\otimes _{%
\mathcal{C}}N:\mathrm{Mod}(\mathcal{C}^{op})\rightarrow \mathbf{Ab}$, given
by $(\otimes _{\mathcal{C}}N)(M)=M\otimes _{\mathcal{C}}N$ is right exact.

\item[(ii)] For each $\mathcal{C}^{op}$-module $M$, the functor $M\otimes _{%
\mathcal{C}}:\mathrm{Mod}(\mathcal{C})\rightarrow \mathbf{Ab}$, given by $%
(M\otimes _{\mathcal{C}})(N)=M\otimes _{\mathcal{C}}N$ is right exact.
\end{itemize}

\item[(b)] The functors $M\otimes _{\mathcal{C}}-$ and $-\otimes _{\mathcal{C%
}}N$ preserve arbitrary sums.

\item[(c)] For each object $C$ in $\mathcal{C}$ $M\otimes _{\mathcal{C}%
}(\;,C)=M(C)$ and $(C,\;\;)\otimes _{\mathcal{C}}N=N(C)$.
\end{itemize}

Given a full subcategory $\mathcal{C}^{\prime }$ of $\mathcal{C}$. The
restriction $\mathrm{res}:\mathrm{Mod}(\mathcal{C})\rightarrow \mathrm{Mod}(%
\mathcal{C}^{\prime })$ has a right adjoint, called also the tensor product,
and denoted by $\mathcal{C}\otimes _{\mathcal{C}^{\prime }}:\mathrm{Mod}(%
\mathcal{C}^{\prime })\rightarrow \mathrm{Mod}(\mathcal{C}).$ This functor
is defined by: $(\mathcal{C}\otimes _{\mathcal{C}^{\prime }}M)(C)=(C,\;\;)|_{%
\mathcal{C}^{\prime }}\otimes _{\mathcal{C}^{\prime }}M$ , for any $M$ in $%
\mathrm{Mod}(\mathcal{C}^{\prime })$ and $C$ in $\mathcal{C}$. The following
proposition is proved in [A Prop. 3.1].

\begin{proposition}
\label{cap1.9} Let $\mathcal{C}^{\prime }$ be a full subcategory of $%
\mathcal{C}$. The functor $\mathcal{C}\otimes _{\mathcal{C}^{\prime }}:%
\mathrm{Mod}(\mathcal{C}^{\prime })\rightarrow \mathrm{Mod}(\mathcal{C})$
satisfies the following conditions:

\begin{itemize}
\item[(a)] $\mathcal{C}\otimes _{\mathcal{C}^{\prime }}$ is right exact and
preserves arbitrary sums.

\item[(b)] The composition $\mathrm{Mod}(\mathcal{C}^{\prime })%
\xrightarrow{\mathcal
C\otimes_{\mathcal C'}}\mathrm{Mod}(\mathcal{C})\xrightarrow{\mathrm{res}}%
\mathrm{Mod}(C^{\prime })$ is the identity in $\mathrm{Mod}(\mathcal{C}%
^{\prime }).$

\item[(c)] For each object $C^{\prime }$ in $\mathcal{C}^{\prime }$, $%
\mathcal{C}\otimes _{\mathcal{C}^{\prime }}\mathcal{C}^{\prime
}(\;\;,C^{\prime })=\mathcal{C}(\;\;,C^{\prime })$.

\item[(d)] $\mathcal{C}\otimes _{\mathcal{C}^{\prime }}$ is fully faithful.

\item[(e)] $\mathcal{C}\otimes _{\mathcal{C}^{\prime }}$ preserves
projective objects.
\end{itemize}
\end{proposition}

The functor $M$ in $\mathrm{Mod}(\mathcal{C})$ is called projectively
presented over $\mathcal{C}^{\prime },$ if there exists an exact sequence $%
\coprod_{i\in I}\mathcal{C}(\;\;,C_{i}^{\prime })\rightarrow \coprod_{j\in J}%
\mathcal{C}(\;\;,C_{j}^{\prime })\rightarrow M\rightarrow 0$, with $%
C_{i}^{\prime },C_{j}^{\prime }\in \mathcal{C}^{\prime }$. The category $%
\mathcal{C}\otimes _{\mathcal{C}^{\prime }}\mathrm{Mod}(\mathcal{C}^{\prime
})$ is the subcategory of $\mathrm{Mod}(\mathcal{C})$ whose objects are the
functors projectively presented over $\mathcal{C}^{\prime }$. The functor $%
\mathrm{res}$ and $\mathcal{C}\otimes _{\mathcal{C}^{\prime }}$ induces an
equivalence between $\mathrm{Mod}(\mathcal{C}^{\prime })$ and $\mathcal{C}%
\otimes _{\mathcal{C}^{\prime }}\mathrm{Mod}(\mathcal{C}^{\prime })$.

We say that an exact sequence $P_{1}\xrightarrow{\alpha}P_{0}%
\xrightarrow{\beta}M\rightarrow 0$ is a \textbf{minimal} \textbf{projective
presentation} of $M$ if and only if the epimorphisms $P_{0}%
\xrightarrow{\beta} M$ and $P_{1}\rightarrow \mathrm{Im}\alpha$, are minimal
projective covers (in the sense of [AF]).

It is of interest to know under what conditions minimal projective
presentations exist.

\begin{theorem}[A Theor. 4.12]
\label{cap1.13} The following conditions are equivalent:

\begin{itemize}
\item[(a)] Every object in $\mathrm{mod}(\mathcal{C})$ has a minimal
projective presentation.

\item[(b)] For each object $C$ in $\mathcal{C}$, every finitely presented $%
\mathrm{End}(C)^{op}$-module has a minimal projective presentation.
\end{itemize}
\end{theorem}

A ring $R$ is \textbf{semiperfect}, if every finitely generated $R$-module
has a projective cover.

\begin{corollary}
\label{cap1.14} If $\mathcal{C}$ is a category, such that for every $C$ in $%
\mathcal{C}$, $\mathrm{End}(C)^{op}$ is semiperfect, then every object in $%
\mathrm{mod}(\mathcal{C})$ has a minimal projective presentation.
\end{corollary}

\subsection{ Krull-Schmidt Categories}

We start giving some definitions from [AR].

\begin{definition}
Let $R$ be a commutative artin ring, an $R$-category $\mathcal{C}$, is a pre
additive category such that $\mathcal{C}(C_{1},C_{2})$ is an $R$-module and
composition is $R$-bilinear. Under these conditions $\mathrm{\mathrm{Mod}}(%
\mathcal{C})$ is a $R$-category which we identify with the category of
functors $(\mathcal{C}^{op},\mathrm{Mod}(R))$.

An $R$-category $\mathcal{C}$ is $\mathrm{Hom}$-\textbf{finite}, if for each
pair of objects $C_{1},C_{2}$ in $\mathcal{C}$ the $R$-module $\mathcal{C}%
(C_{1},C_{2})$ is finitely generated. We denote by $(\mathcal{C}^{op},%
\mathrm{mod}(R))$, the full subcategory of $(\mathcal{C}^{op},\mathrm{%
\mathrm{Mod}}(R))$ consisting of the $\mathcal{C}$-modules such that; for
every $C$ in $\mathcal{C}$ the $R$-module $M(C)$ is finitely generated. The
category $(\mathcal{C}^{op},\mathrm{mod}(R))$ is abelian and the inclusion $(%
\mathcal{C}^{op},\mathrm{mod}(R))\rightarrow (\mathcal{C}^{op},\mathrm{%
\mathrm{Mod}}(R)) $ is exact.
\end{definition}

The category $\mathrm{mod}(C)$ is a full subcategory of $(\mathcal{C}^{op},%
\mathrm{mod}(R))$. The functors $D:(\mathcal{C}^{op},\mathrm{mod}%
(R))\rightarrow (\mathcal{C},\mathrm{mod}(R))$, and $D:(\mathcal{C},\mathrm{%
mod}(R))\rightarrow (\mathcal{C}^{op},\mathrm{mod}(R))$, are defined as
follows: for any $C$ in $\mathcal{C}$, $D(M)(C)=\mathrm{Hom}%
_{R}(M(C),I(R/r)) $, with $r$ the Jacobson radical of $R,$ and $I(R/r)$ is
the injective envelope of $R/r.$ The functor $D$ defines a duality between $(%
\mathcal{C},\mathrm{mod}(R))$ and $(\mathcal{C}^{op},\mathrm{mod}(R))$. If $%
\mathcal{C}$ is an $\mathrm{Hom}$-finite $R$-category and $M$ is in $\mathrm{%
mod}(\mathcal{C})$, then $M(C)$ is a finitely generated $R$-module and is
therefore in $\mathrm{mod}(R)$.

\begin{definition}
An $\mathrm{Hom}$-finite $R$-category\ $\mathcal{C}$ is\ \textbf{dualizing,}
\ if\ the functor\ $D:(\mathcal{C}^{op},\mathrm{mod}(R))\rightarrow (%
\mathcal{C},\mathrm{mod}(R))$ induces a duality between the categories $%
\mathrm{mod}(\mathcal{C})$ and $\mathrm{mod}(\mathcal{C}^{op}).$
\end{definition}

It is clear from the definition that for dualizing categories $\mathrm{mod}(%
\mathcal{C})$ has enough injectives.

To finish, we recall the following definition:

\begin{definition}
An additive category $\mathcal{C}$ is \textbf{Krull-Schmidt}, if every
object in $\mathcal{C}$ decomposes in a finite sum of objects whose
endomorphism ring is local.
\end{definition}

\subsection{The radical of a category}

The notion of the Jacobson radical of a category was introduced in [M] and
[A], it is defined in the following way:

\begin{definition}
The Jacobson radical of $\mathcal{C}$, $\mathrm{rad}_{\mathcal{C}}(\;,\;)$,
is a subbifunctor of $\mathrm{Hom}_{\mathcal{C}}(\;,\;)$ defined in objects
as: $\mathrm{rad}_{\mathcal{C}}(X,Y)=\{f\in \mathrm{Hom}_{\mathcal{C}%
}(X,Y)\mid $for any map $g:Y\rightarrow X$, $1-gf$ is invertible \}.

If $M$ is a $\mathcal{C}$-module , then we denote by $\mathrm{rad}M$ the
intersection of all maximal subfunctors of $M$.
\end{definition}

\begin{proposition}
\label{cap1.21}\emph{[A], [BR], [M]} Let $\mathcal{C}$ be an additive
category and $\mathrm{rad}_{\mathcal{C}}(\;,\;)$ the Jacobson radical of $%
\mathcal{C}$. Then:

\begin{itemize}
\item[(a)] For every object $C$ in $\mathcal{C}$ $\mathrm{rad}_{\mathcal{C}%
}(C,C)$ is just the Jacobson radical of $\mathrm{End}_{\mathcal{C}}(C).$

\item[(b)] If $C$ and $C^{\prime }$ are indecomposable objects in $\mathcal{C%
}$, then the radical $\mathrm{rad}_{\mathcal{C}}(C,C^{\prime })$ consists of
all non isomorphisms from $C$ to $C^{\prime }$.

\item[(c)] For every object $C$ in $\mathcal{C}$, $\mathrm{rad}_{\mathcal{C}%
}(C,\;\;)=\mathrm{rad}\mathcal{C}(C,\;\;)$ and $\mathrm{rad}_{\mathcal{C}%
}(\;\;,C)=\mathrm{rad}\mathcal{C}(\;\;,C)$.

\item[(d)] For every pair of objects $C$ and $C^{\prime }$ in $\mathcal{C}$,
$\mathrm{rad}_{\mathcal{C}}(C^{\prime },\;\;)(C)=\mathrm{rad}_{\mathcal{C}%
}(\;\;,C)(C^{\prime })$.
\end{itemize}
\end{proposition}

\begin{definition}
By an ideal of the additive category $\mathcal{C}$ we understand a sub
bifunctor of $\mathrm{Hom}_{\mathcal{C}}(\;,\;).$
\end{definition}

Given two ideals of $I_{1}$ and $I_{2}$ of $\mathcal{C}$ we define $%
I_{1}I_{2}$ as follows: $f\in I_{1}I_{2}(C_{1},C_{3}),$ if and only if, $f$
is a finite sum of morphisms $C_{1}\xrightarrow{h}C_{2}\xrightarrow{g}C_{3}$%
, with $h\in I_{1}(C_{1},C_{2})$ and $g\in I_{2}(C_{2},C_{3})$.

Given an ideal $I$ of $\mathcal{C}$ and a $\mathcal{C}$-module $M,$ we
define a $\mathcal{C}$-sub module $IM$ of $M$ by
\begin{equation*}
IM(C)=\Sigma _{f\in I(C,C^{\prime })}\mathrm{Im}M(f)\text{,}
\end{equation*}%
with $C$ in $\mathcal{C}$. We say a $\mathcal{C}$-module $S$ is \textbf{%
simple}, if it does not have proper $\mathcal{C}$-sub modules.

\begin{lemma}
\label{cap1.22} \emph{[MVS1 Lemma 2.5], [BR]} Let $\mathcal{C}$ be a
Krull-Schmidt $K$-category. Then the following statements are true:

\begin{itemize}
\item[(a)] Any simple functor in $\mathrm{Mod}(\mathcal{C})$ is of the form $%
\mathcal{C}(\;\;,C)/\mathrm{rad}_{\mathcal{C}}(\;\;,C)$, for some
indecomposable object $C$ in $\mathcal{C}$.

\item[(b)] For all finitely generated functors $F$ in $\mathrm{Mod}(\mathcal{%
C})$, the radical of $F$ is isomorphic to $\mathrm{rad}_{\mathcal{C}}F.$

\item[(c)] All finitely generated functors $F$ in $\mathrm{Mod}(\mathcal{C})$
have a projective cover.
\end{itemize}
\end{lemma}

\subsection{A Pair of Adjoint Functors}

Let $\mathcal{C}$ be a skeletally small additive category, and $\mathcal{T}$
a skeletally small full subcategory of $\mathrm{Mod}(\mathcal{C})$. Let's
define the following functor
\begin{equation*}
\phi :\mathrm{Mod}(\mathcal{C})\rightarrow \mathrm{Mod}(\mathcal{T}%
),\;\;\phi (M)=\mathrm{Hom}(\;\;,M)_{\mathcal{T}}.
\end{equation*}

Our aim in this subsection is to prove $\phi $ has a left adjoint. Since $%
\mathrm{Mod}(\mathcal{C})$ is abelian, it has equalizers, and it was
remarked above it is Ab5. In this way $\mathrm{\mathrm{Mod}}(\mathcal{C})$
is complete [SM V.2 Theo.1 ]. Since the functor $\mathrm{Hom}_\mathcal{C}%
(T,-)$ preserve limits, for every $T $ in $\mathcal{T}$ it follows that the
functor $\phi $ preserve limits.

\begin{lemma}
The category $\mathrm{Mod}(\mathcal{C})$ has an injective cogenerator.
\end{lemma}

\begin{proof}
Denote by $D$ the functor $D:\mathrm{Mod}(\mathcal{C})\rightarrow \mathrm{Mod%
}(\mathcal{C}^{op})$, $D(F)(C)=(F(C),\mathbb{Q/Z})$, and let $F$ be a
functor in $\mathrm{Mod}(\mathcal{C}^{op})$. Then we have an epimorphism $%
\coprod_{i}\mathcal{C}(C_{i},\;)\rightarrow DF$ and hence, a monomorphism
\begin{equation*}
F\xrightarrow{\eta}D^{2}F\xrightarrow{Df}\prod_{i\in I}D\mathcal{C}%
(C_{i},\;\;)\text{,}
\end{equation*}%
where $\eta _{C}(x)(f)=f(x)$. Moreover, $D(C_{i},\;)$ is injective. Hence $%
\{D\mathcal{C}(C,\;)\}_{C\in \mathcal{C}}$ is a small cogenerator set
consisting of injective objects in $\mathrm{Mod}(\mathcal{C})$.
\end{proof}

We need the following

\begin{definition}
Let $\mathcal{A}$ be an arbitray category and $\mathbf{X}\rightarrow
Y=\{X_{i}\xrightarrow{f_i}Y\}_{i\in I}$ a family of morphisms. A \textbf{%
fibered product} of $\mathbf{X}\rightarrow Y$ is a pair $(\{P%
\xrightarrow{p_i}X_{i}\}_{i\in I},L)$, where $L:P\rightarrow Y$ is a
morphism such that $f_{i}p_{i}=L$, and satisfies the following universal
property:

If $(\{Q\xrightarrow{q_i}X_{i}\}_{i\in I},L^{\prime })$ is a pair such that
for each $i\in I$, $f_{i}q_{i}=L^{\prime }$, then there exists a unique
morphism $\eta :Q\rightarrow P$, such that $L\eta =L^{\prime}$, and for each
$i\in I$, $p_{i}\eta =q_{i}$.
\end{definition}

Using the fact $\mathrm{\mathrm{Mod}}(\mathcal{C})$ is Ab5 and $\phi $ is
left exact, the reader can verify the following

\begin{proposition}
In $\mathrm{Mod}(\mathcal{C})$ any family of monomorphisms $\mathbf{X}%
\rightarrow Y=\{X_{i}\xrightarrow{f_i}Y\}_{i\in I}$ has a fibered product.
Moreover, $\phi $ preserves fibered products of monomorphisms.
\end{proposition}

Freyd's special adjoint functor theorem [MS V.7 Theo. 2] justifies the
following assertion.

\begin{theorem}
\label{AD13} Let $\mathcal{C}$ be a skeletally small pre additive category.
Then the following is true:

\begin{itemize}
\item[(a)] The category $\mathrm{Mod}(\mathcal{C})$ is complete-small, has a
small cogenerator, and any set of subobjects of a functor $M$ has a fibered
product.

\item[(b)] Let $\mathcal{T}$ be a small full subcategory of $\mathrm{Mod}(%
\mathcal{C}) $. Then the functor $\phi :\mathrm{Mod}(\mathcal{C})\rightarrow
\mathrm{Mod}(\mathcal{T})$ preserves small limits and fibered products of
families of monomorphisms.

\item[(c)] The functor $\phi :\mathrm{Mod}(\mathcal{C})\rightarrow \mathrm{%
Mod}(\mathcal{T})$ has a left adjoint.
\end{itemize}
\end{theorem}

\begin{remark}
Denote by $-\otimes \mathcal{T}$ the left adjoint of $\phi $, it follows
from Yoneda's Lemma, that for any pair of objects $T$ in $\mathcal{T}$ and $%
M $ in $\mathrm{Mod}(\mathcal{C})$, there are natural isomorphisms:
\begin{equation*}
\mathrm{Hom}_{\mathcal{C}}((\;\;,T)_{\mathcal{T}}\otimes \mathcal{T},M)\cong
\mathrm{Hom}_{\mathcal{T}}((\;\;,T)_{\mathcal{T}},\phi (M))\cong \mathrm{Hom}%
_{\mathcal{C}}(T,M)\text{.}
\end{equation*}%
By Yoneda's Lemma again, $(\;,T)_{\mathcal{T}}\otimes \mathcal{T}=T.$
\end{remark}

Since there are enough projective and enough injective objects in $\mathrm{%
Mod}(\mathcal{C})$, and $\mathrm{Mod}(\mathcal{T})$, we can define, for any
integer $n$, the nth right derived functors of the functors $\mathrm{Hom}_{%
\mathcal{C}}(M,\;\;)$, $\mathrm{Hom}_{\mathcal{C}}(\;\;,N)$, which will be
denoted by $\mathrm{Ext}_{\mathcal{C}}^{n}(M,\;\;)$ and $\mathrm{Ext}_{%
\mathcal{C}}^{n}(\;\;,N)$ respectively. Analogously, the nth left derived
functors of $M\otimes _{\mathcal{C}}-$ and $-\otimes N$, can be defined.
They will be denoted by $\mathrm{Tor}_{n}^{\mathcal{C}}(M,\;\;)$ and $%
\mathrm{Tor}_{n}^{\mathcal{C}}(\;\;,N),$ respectively.

In the same way, the right derived functors of the functor $\phi $ can be
defined, they will be denoted by $\mathrm{Ext}_{\mathcal{C}}^{n}(\;\;,-)_%
\mathcal{T}:\mathrm{Mod}(\mathcal{C})\rightarrow \mathrm{Mod}(\mathcal{T})$,
and they are defined as $\mathrm{Ext}_{\mathcal{C}}^{n}(\;\;,-)_\mathcal{T}%
(M)=\mathrm{Ext}_{\mathcal{C}}^{n}(\;\;,M)_\mathcal{T}$.

Of course, we can also define the left derived functors of the functor $%
-\otimes \mathcal{T}$, they will be denoted by $\mathrm{Tor}_{n}^{\mathcal{T}%
}(\;\;,\mathcal{T}):\mathrm{Mod}(\mathcal{T})\rightarrow \mathrm{Mod}(%
\mathcal{C})$.

We will see below relations among these functors.

\begin{proposition}
\label{TC} If $M$ is finitely presented, then the functors $\mathrm{Ext}_{%
\mathcal{C}}^{1}(M,\;\;)$ commute with arbitrary sums.
\end{proposition}

\begin{proof}
There is an exact sequence
\begin{equation}
0\rightarrow \mathrm{Ker}(\alpha )\rightarrow \mathcal{C}(\;\;,C)%
\xrightarrow{\alpha}M\rightarrow 0  \label{TA}
\end{equation}%
with $\mathrm{Ker}(\alpha )$ finitely generated and $C$ an object in $%
\mathcal{C}$. Let $\{N_{i}\}_{i\in I}$ be a family of objects in $\mathrm{Mod%
}(\mathcal{C})$. After applying $\mathrm{Hom}_\mathcal{C}(\;\;,\coprod_{i\in
I}N_{i})$ to (\ref{TA}), it follows the existence of a isomorphism $\eta $,
such that the following diagram commutes
\begin{equation*}
\divide\dgARROWLENGTH by2 \begin{diagram} \node{(\mathcal
C(\;\;,C),\coprod_{i\in I}N_i)}\arrow{e,t}{}\arrow{s,l}{\cong}
\node{(\mathrm{Ker}(\alpha),\coprod_{i\in J
}N_i)}\arrow{e,t}{}\arrow{s,l}{\cong}
\node{\mathrm{Ext}^1_\mathcal{C}(M,\coprod_{i\in
I}N_i)}\arrow{e,t}{}\arrow{s,l}{\eta } \node{0}\\ \node{\coprod_{i\in
I}(\mathcal C(\;\;,C),N_i)}\arrow{e,t}{} \node{\coprod_{i\in
I}(\mathrm{Ker}(\alpha),N_i)}\arrow{e,t}{} \node{\coprod_{i\in
I}\mathrm{Ext}^1_\mathcal{C}(M,N_i)}\arrow{e,t}{} \node{0} \end{diagram}
\end{equation*}
\end{proof}

\begin{corollary}
If $\mathcal{T}$ consists of finitely presented functors, then the functors $%
\phi $ and $\mathrm{Ext}_{\mathcal{C}}^{1}(\;\;,-)_\mathcal{T}$ commute with
arbitrary sums.
\end{corollary}

\section{Brenner-Butler%
\'{}%
s Theorem}

\subsection{The main theorem.}

In this section we introduce the notions of tilting categories and we show,
that with slights modifications, Bongartz's proof $[B]$ of Brenner-Butler's
theorem $[BB]$, extends to tilting categories over arbitrary skeletally
small pre additive categories. We also have the corresponding theorems on
the invariance of Grothendieck groups under tilting, and the relations
between the global dimension of a category and the global dimension of the
tilted category. Then we prove, that under mild conditions, tilting functors
restrict to the categories of finitely presented functors. For dualizing
categories we have theorems analogous to classical tilting for finite
dimensional algebras.

The first definition is a natural generalization of the classical notion of
a tilting object.

\begin{definition}
Let $\mathcal{C}$ an annuli variety, a subcategory $\mathcal{T}$ of $\mathrm{%
\mathrm{Mod}}(\mathcal{C})$ is a tilting category, if every object in $%
\mathcal{T}$ is finitely presented and the following conditions are
satisfied:

\begin{itemize}
\item[(i)] $\mathrm{pdim}\mathcal{T}\leq 1.$

\item[(ii)] We have $\mathrm{Ext}_{\mathcal{C}}^{1}(T_{i},T_{j})=0$, for
every pair of objects $T_{i},T_{j}$ in $\mathcal{T}$.

\item[(iii)] For every object $C$ in $\mathcal{C}$, the representable
functor $\mathcal{C}(\;\;,C)$ has a resolution
\begin{equation*}
0\rightarrow \mathcal{C}(\;\;,C)\rightarrow T_{1}\rightarrow
T_{2}\rightarrow 0\text{,}
\end{equation*}%
with $T_{1},T_{2}$ in $\mathcal{T}$.
\end{itemize}
\end{definition}

\begin{definition}
Given a skeletally small pre additive category $\mathcal{C}$, we introduce
the following notion of tensor product: given an abelian group $G$ and an
object $M$ in $\mathrm{Mod}(\mathcal{C})$ the tensor product $G\otimes _{%
\mathbb{Z}}M$, is the functor defined in objects as $(G\otimes _{\mathbb{Z}%
}M)(C)=G\otimes _{\mathbb{Z}}M(C)$.
\end{definition}

Given a subcategory $\mathcal{T}$ of $\mathrm{Mod}(\mathcal{C} )$, we define
for every $M$ in $\mathrm{Mod}(\mathcal{C})$ and every object $T$ of $%
\mathcal{T}$ the evaluation map $e^{(T,M)}:\mathcal{C}(T,M)\otimes _{\mathbb{%
Z}}M\rightarrow M$, $e^{(T,M)}_{C}(f\otimes m)=f_C(m)$, where, $C\in
\mathcal{C}$, $m\in M(C)$ and $f=\{f_{C}\}_{C\in\mathcal{C}}\in \mathcal{C}%
(T,M)$. Define the trace of $\mathcal{T}$ in $M$, as the image of the sum $%
\coprod e^{(T_{i},M)}:\coprod_{i\in I}\mathcal{C}(T_{i},M)\otimes _{\mathbb{Z%
}}T_{i}\rightarrow M$, where $\{T_{i}\}_{i\in I}$ is a set of
representatives of the isomorphism classes of objects in $\mathcal{T}$.

Let $T_{i}$ be an object in $\{T_{i}\}_{i\in I}$, and $X_{i}=\{f_{j}^{i}%
\}_{j\in J_{i}}$ a set of generators of the abelian group $\mathcal{C}%
(T_{i},M)$, let $\mathbb{Z}^{(X_{i})}$ be the free abelian group with basis $%
X_{i}$ and $X=\cup X_{i}$. Then the epimorphism of abelian groups
\begin{equation*}
\varphi :\mathbb{Z}^{(X_{i})}\rightarrow \mathcal{C}(T_{i},M)\rightarrow 0
\end{equation*}

\bigskip Induces an epimorphism

\begin{equation*}
\psi :\coprod_{i\in I}T_{i}^{(X_{i})}\cong \coprod_{i\in I}\mathbb{Z}%
^{(X_{i})}\otimes _{\mathbb{Z}}T_{i}\rightarrow \coprod_{i\in I}\mathcal{C}%
(T_{i},M)\otimes _{\mathbb{Z}}T_{i}\;\;\rightarrow 0
\end{equation*}

composing with the sum of the evaluation maps we obtain a map

\begin{equation*}
\coprod e^{(T_{i},M)}\psi :\coprod_{i\in I}T_{i}^{(X_{i})}\rightarrow M
\end{equation*}

\bigskip Let's denote $\coprod_{i\in I}T_{i}^{(X_{i})}$ by $T^{(X)}$, then
the map: $\varTheta_{M}=$ $\coprod e^{(T_{i},M)}\psi :$ $T^{(X)}\rightarrow
M $ has the following property:

\begin{proposition}
\label{Catilt} Let $M$ be in $\mathrm{Mod}($ $\mathcal{C)}$, and $\varTheta
_{M}:T^{(X)}\rightarrow M$ the map given above. Then for every object $%
T^{\prime }\in \mathcal{T}$ and every map $\eta :T^{\prime }\rightarrow M$,
there exists a map $\gamma :T^{\prime }\rightarrow T^{(X)}$ such that $%
\varTheta_{M}$ $\gamma =\eta $.
\end{proposition}

Given a skeletally small pre additive category $\mathcal{C}$ and $\mathcal{T}
$ a tilting subcategory of $\mathrm{\mathrm{Mod}}(\mathcal{C})$, we denote
by $\mathscr{T}$ the full subcategory of $\mathrm{\mathrm{Mod}}(\mathcal{C})$
whose objects are epimorphic images of objects in $\mathcal{T}$. We want to
prove that $\mathscr T$ is a torsion class of $\mathrm{\mathrm{Mod}}(%
\mathcal{C})$.

The above proposition implies the following:

\begin{proposition}
\label{TORD} Let $M$ be an object in $\mathrm{\mathrm{Mod}}(\mathcal{C})$.
Then $M$ is in $\mathscr T$, if and only if, $\coprod e^{(T_i,M)}$ is an
epimorphism.
\end{proposition}

We can prove now:

\begin{proposition}
\label{TORC} If $M$ is in $\mathscr T$, then, for every object $T\in
\mathcal{T}$, $\mathrm{Ext}_{\mathcal{C}}^{1}(T,M)=0$.
\end{proposition}

\begin{proof}
Let's assume $M$ is in $\mathscr T$. Then there exists a short exact
sequence
\begin{equation}
0\rightarrow \mathrm{Ker}(\alpha )\rightarrow \coprod_{j\in J}T_{j}%
\xrightarrow{\alpha}M\rightarrow 0\text{,}  \label{TORA}
\end{equation}%
where $\{T_{j}\}_{j\in J}$ is a family of objects in $\mathcal{T} $. Since $%
T $ is finitely presented, it follows $\mathrm{Ext}_{\mathcal{C}%
}^{1}(T,\;\;) $ commutes with arbitrary sums. By hypothesis, $\mathrm{Ext}_{%
\mathcal{C}}^{1}(T,T_{j})=0$ for every $j\in J$, and $T$ of projective
dimension one, implies for every object $K,$ $\mathrm{Ext}_{\mathcal{C}%
}^{2}(T,K)=0$. Applying $\mathrm{Hom}_{\mathcal{C}}(T,\;\;)$ to the above
exact sequence, it follows from the long homology sequence, that the
sequence $\mathrm{Ext}_{\mathcal{C}}^{1}(T,\coprod_{j\in J}T_{j})\rightarrow
\mathrm{Ext}_{\mathcal{C}}^{1}(T,M)\rightarrow \mathrm{Ext}_{\mathcal{C}%
}^{2}(T,K)$ is exact. Finally we have $\mathrm{Ext}_{\mathcal{C}}^{1}(T,M)=0$%
.
\end{proof}

\begin{proposition}
\label{TORG} The category $\mathscr{T}$ is closed under extensions, direct
sums and epimorphic images. This is: $\mathscr{T}$ is a torsion class of a
torsion theory of $\mathrm{\mathrm{Mod}}(\mathcal{C})$.
\end{proposition}

\begin{proof}
We only need to see it is closed under extensions. Consider the exact
sequence%
\begin{equation*}
0\rightarrow M_{1}\rightarrow M_{2}\rightarrow M_{3}\rightarrow 0\text{,}
\end{equation*}%
with $M_{1},M_{3}$ in $\mathscr{T}$. Then $\mathrm{Ext}_{\mathcal{C}%
}^{1}(T_{i},M_{1})=\mathrm{Ext}_{\mathcal{C}}^{1}(T_{i},M_{3})=0$, by
Proposition \ref{TORC}.

For every object $T_{i}\in \{T_{i}\}_{i\in I}$ , we apply the functor $%
\mathcal{C}(T_{i},\;\;)$, to the above exact sequence to obtain, by the long
exact sequence, an exact sequence of abelian groups
\begin{equation*}
0\rightarrow \mathcal{C}(T_{i},M_{1})\rightarrow \mathcal{C}%
(T_{i},M_{2})\rightarrow \mathcal{C}(T_{i},M_{3})\rightarrow \mathrm{Ext}%
^{1}_{\mathcal{C}}(T_{i},M_{1})=0\text{.}
\end{equation*}

Applying the tensor product $\otimes _{\mathbb{Z}}T_{i}$ to the sequence,
and adding the exact sequences, we obtain the exact sequence:
\begin{equation*}
\coprod_{i\in I}\mathcal{C}(T_{i},M_{1})\otimes _{\mathbb{Z}%
}T_{i}\rightarrow \coprod_{i\in I}\mathcal{C}(T_{i},M_{2})\otimes _{\mathbb{Z%
}}T_{i}\rightarrow \coprod_{i\in I}\mathcal{C}(T_{i},M_{3})\otimes _{\mathbb{%
Z}}T_{i}\rightarrow 0\text{,}
\end{equation*}

which induces a commutative diagram:
\begin{equation*}
\dgARROWLENGTH=.3em\begin{diagram} \node{} \node{\coprod_{i\in
I}\mathcal{C}(T_i,M_1)\otimes T_i}\arrow{e,t}{}\arrow{s,l}{\coprod
e^{(T_i,M_1)}} \node{\coprod_{i\in I}\mathcal{C}(T_i,M_1)\otimes
T_i}\arrow{e,t}{}\arrow{s,l}{\coprod e^{(M_i,F_2)}} \node{\coprod_{i\in
I}\mathcal{C}(T_i,M_1)\otimes T_i\rightarrow 0}\arrow{s,l}{\coprod
e^{(T_i,M_3)}}\\ \node{0}\arrow{e} \node{M_1}\arrow{e,t}{}
\node{M_2}\arrow{e,t}{} \node{ M_3\rightarrow 0} \end{diagram}
\end{equation*}%
The morphisms $\coprod e^{(T_{i},M_{1})}$ and $\coprod e^{(T_{i},M_{3})}$
are epimorphisms, by Proposition \ref{TORD}. Then it follows, by the Snake
lemma, $\coprod e^{(T_{i},M_{2})}$ is an epimorphism. Again by Proposition %
\ref{TORD}, $M_{2}$ is in $\mathscr T$.
\end{proof}

We will prove next that the trace of $\mathcal{T}$ in $M$, denoted by $\tau_{%
\mathcal{T}}(M)$, is the idempotent radical corresponding to the torsion
theory with torsion class $\mathcal{T}$.

\begin{proposition}
\label{TORF} For any $M$ in $\mathrm{Mod}(\mathcal{C})$ and $T^{\prime }$ in
$\mathcal{T}$, $\mathrm{Hom}_{\mathcal{C}}(T^{\prime },M/\tau_{\mathcal{T}%
}M)=0$. In particular $\tau_{\mathcal{T}}(M/\tau_{\mathcal{T}}(M))=0$.
\end{proposition}

\begin{proof}
Consider the exact sequence
\begin{equation}
0\rightarrow \tau _{\mathcal{T}}(M)\xrightarrow{j}M\xrightarrow{p}M/\tau _{%
\mathcal{T}}(M)\rightarrow 0  \label{trazaart1}
\end{equation}%
and a natural transformation $\eta :T^{\prime }\rightarrow M/\tau _{\mathcal{%
T}}(M)$. Taking the pull back of the maps $\eta $ and $p$ we get the
following commutative exact diagram:

\begin{equation*}
\begin{diagram}\dgARROWLENGTH=1.5em \node{0}\arrow{e,t}{}
\node{\tau_{\mathcal{T}}(M)}\arrow{e,t}{}\arrow{s,=,-}{}
\node{W}\arrow{e,t}{}\arrow{s,t}{} \node{T'}\arrow{e,t}{}\arrow{s,l}{\eta}
\node{0}\\ \node{0}\arrow{e,t}{} \node{\tau_{\mathcal{T}}(M)}\arrow{e,t}{j}
\node{M}\arrow{e,t}{p} \node{M/\tau_{\mathcal{T}}(M)}\arrow{e,t}{} \node{0}
\end{diagram}
\end{equation*}%
By Proposition \ref{TORC}, the top exact sequence splits, hence there exists
a map $f:$ $T^{\prime }\rightarrow M$, such that $pf$=$\eta $. By the
properties of the trace, $f(T^{\prime })\subset \tau _{\mathcal{T}}M$ and $%
\eta =0$.
\end{proof}

We have the following characterization of the torsion class:

\begin{proposition}
\label{TORH} Let $\mathcal{T}$ be a tilting category in $\mathrm{\mathrm{Mod}%
}(\mathcal{C})$. Then:
\begin{equation*}
{\mathscr T}=\{M\in \mathrm{\mathrm{Mod}}(C)|\mathrm{Ext}_{\mathcal{C}%
}^{1}(T,M)=0,\;T\in \mathcal{T}\}.
\end{equation*}
\end{proposition}

\begin{proof}
Let's assume $M\in \mathrm{\mathrm{Mod}}(\mathcal{C})$ and for each $T$ in $%
\mathcal{T}$, $\mathrm{Ext}_{\mathcal{C}}^{1}(T,M)=0$. Then applying the
functor $\mathrm{Hom}_{\mathcal{C}}(T,-)$ to the sequence (\ref{trazaart1}),
it follows from the long homology sequence and the fact $\mathrm{pdim}T=1$
that
\begin{equation*}
0=\mathrm{Ext}_{\mathcal{C}}^{1}(T,M)\rightarrow \mathrm{Ext}_{\mathcal{C}%
}^{1}(T,M/\tau _{\mathcal{T}}(M))\rightarrow \mathrm{Ext}_{\mathcal{C}%
}^{2}(T,\tau _{\mathcal{T}}(M))=0
\end{equation*}%
is exact, and in consequence, $\mathrm{Ext}_{\mathcal{C}}^{1}(T,M/\tau _{%
\mathcal{T}}(M))=0$.

Let $C\in \mathcal{C}$. Then we have an exact sequence $0\rightarrow
\mathcal{C}(\;\;,C)\rightarrow T_{0}\rightarrow T_{1}\rightarrow 0$, with $%
T_{i}\in T$, $i=0,1$. After applying $\mathrm{Hom}_{\mathcal{C}}(\;\;,M/\tau
_{\mathcal{T}}(M))$ to the sequence, we get, by the long homology sequence,
the exact sequence
\begin{equation*}
0=\mathrm{Hom}(T_{0},M/\tau _{\mathcal{T}}(M))\rightarrow \mathrm{Hom}%
((\;\;,C),M/\tau _{\mathcal{T}}(M))\rightarrow \mathrm{Ext}^{1}(T_{1},M/\tau
_{\mathcal{T}}(M))=0\text{.}
\end{equation*}%
By Yoneda's Lemma, $0=\mathrm{Hom}((\;\;,C),M/\tau _{\mathcal{T}%
}(M))=M(C)/\tau _{\mathcal{T}}M(C)$, this is: $M=\tau _{\mathcal{T}}(M)$ is
in $\mathscr{T}$. The other inclusion is already proved in Proposition \ref%
{TORC}.
\end{proof}

The torsion class $\mathscr T$ induces a torsion pair $(\mathscr T,\mathscr %
F)$, where
\begin{equation*}
\mathscr F=\{N\in \mathrm{\mathrm{Mod}}(\mathcal{C})|\mathrm{Hom}%
(M,N)=0,\;M\in \mathscr T\}\text{,}
\end{equation*}%
[see S].

\begin{proposition}
$\mathscr{F}=\{N\in\mathrm{\mathrm{Mod}}(\mathcal{C})\}|\mathrm{Hom}%
(T,N)=0,\; T\in\mathcal{T}\}$.
\end{proposition}

\begin{proof}
Let $M$ be an object in $\mathscr T$. Then there is an epimorphism $%
\coprod_{j\in J}T_{j}\rightarrow M\rightarrow 0$, with $\{T_{j}\}_{i\in J}$
a family of objects in $\mathcal{T}$, and let $N\in \{N\in \mathrm{\mathrm{%
Mod}}(\mathcal{C})|\mathrm{Hom}(T,N)=0,\;T\in \mathcal{T}\}$. Then there is
a monomorphism
\begin{equation*}
0\rightarrow \mathrm{Hom}_{\mathcal{C}}(M,N)\rightarrow\mathrm{Hom}_\mathcal{%
C} (\coprod_{j\in J}T_{j},N)\cong \prod_{j\in J}\mathrm{Hom}_{\mathcal{C}%
}(T_{j},N)=0\text{,}
\end{equation*}%
and $M$ is in $\mathscr F$. The other inclusion is clear.
\end{proof}

Let $\mathcal{T}$ be a tilting subcategory of $\mathrm{\mathrm{Mod}}(%
\mathcal{C})$, we define $\mathrm{Add}\mathcal{T}$ as the full subcategory
of $\mathrm{\mathrm{Mod}}(\mathcal{C})$ whose objects are direct summands of
arbitrary sums of objects in $\mathcal{T}$. Let $\phi :\mathrm{\mathrm{Mod}}(%
\mathcal{C})\rightarrow \mathrm{\mathrm{Mod}}(\mathcal{T})$ be the functor
defined as $\phi (M)=\mathrm{Hom}(\;\;,M)_{\mathcal{T}}$. We adapt Bongartz
's proof of Brenner-Butler's theorem to our situation. We start with a
version of [B Prop. 1.4]:

\begin{proposition}
\label{TORZ}

\begin{itemize}
\item[(i)] If $M_{3}\xrightarrow{h}M_{2}\xrightarrow{g}M_{1}\xrightarrow{f}%
M_{0}$ is an exact sequence in $\mathrm{\mathrm{Mod}}(\mathcal{C})$, such
that $M_{i}$ belongs to $\mathscr T$, then the sequence%
\begin{equation*}
\phi (M_{2})\rightarrow \phi (M_{1})\rightarrow \phi (M_{0})
\end{equation*}%
is exact.

\item[(ii)] For each $M\in \mathscr T$ there is an exact sequence
\begin{equation*}
\cdots \xrightarrow{t_{n+1}}T^{n}\rightarrow \cdots \rightarrow T^{1}%
\xrightarrow{t_1}T^{0}\xrightarrow{t_0}M\rightarrow 0
\end{equation*}%
with $T^{i}$ in $\mathrm{Add}\mathcal{T}$, such that the maps $T^{n}\overset{%
\delta_{n}}{\rightarrow }\mathrm{Im}t_{n}\rightarrow 0$, have the following
property: given $T$ in $\mathcal{T}$ and a map $f:T\rightarrow \mathrm{Im}t
_{n}$, there exist a map $h:T\rightarrow T^{n}$, such that $\delta_{n}h=f.$

\item[(iii)] For each $M$ in $\mathscr{T}$ there exists an isomorphism
\begin{equation*}
\phi (M)\otimes \mathcal{T}\rightarrow M\text{.}
\end{equation*}

\item[(iv)] $\mathrm{Tor}_{1}^{\mathcal{T}}(\phi (M),\mathcal{T})=0$.

\item[(v)] For any pair of objects $M,N$ in $\mathscr{T}$ we have an
isomorphism $\mathrm{Ext}_{\mathcal{C}}^{i}(M,N)=\mathrm{Ext}_{\mathcal{T}%
}^{i}(\phi (M),\phi (N))$.
\end{itemize}
\end{proposition}

\begin{proof}
(i) Since $\mathscr {T}$ is closed under epimorphic images, then $\mathrm{Im}%
(g)$, $\mathrm{Im}(f)$ and $\mathrm{Coker}(f)$ are in $\mathscr T$. From the
exact sequences: $0\rightarrow \mathrm{Im}h\rightarrow M_{2}\rightarrow
\mathrm{Im}g\rightarrow 0$, $0\rightarrow \mathrm{Im}g\rightarrow
M_{1}\rightarrow \mathrm{Im}f\rightarrow 0$, $0\rightarrow \mathrm{Im}%
f\rightarrow M_{0}\rightarrow \mathrm{Coker}f\rightarrow 0$, we obtain exact
sequences in $\mathrm{Mod}($ $\mathcal{T}):$
\begin{eqnarray*}
0 &\rightarrow &(\;\;,\mathrm{Im}h)_{\mathcal{T}}\rightarrow (\;\;,M_{2})_{%
\mathcal{T}}\rightarrow (\;\;,\mathrm{Im}g)_{\mathcal{T}}\rightarrow \mathrm{%
Ext}^{1}(\;\;,\mathrm{Im}h)_{\mathcal{T}}=0\text{,} \\
0 &\rightarrow &(\;\;,\mathrm{Im}g)_{\mathcal{T}}\rightarrow (\;\;,M_{1})_{%
\mathcal{T}}\rightarrow (\;\;,\mathrm{Im}f)_{\mathcal{T}}\rightarrow \mathrm{%
Ext}^{1}(\;\;,\mathrm{Im}g)_{\mathcal{T}}=0\text{,} \\
0 &\rightarrow &(\;\;,\mathrm{Im}f)_{\mathcal{T}}\rightarrow (\;\;,M_{0})_{%
\mathcal{T}}\rightarrow (\;\;,\mathrm{Coker}f)_{\mathcal{T}}\rightarrow
\mathrm{Ext}^{1}(\;\;,\mathrm{Im}f)_{\mathcal{T}}=0\text{.}
\end{eqnarray*}%
Gluing the sequences, the result follows.

(ii) By the Propositions \ref{Catilt} and \ref{TORD}, for each $M\in
\mathscr
T$ there is an exact sequence
\begin{equation}
0\rightarrow K_{0}\rightarrow T^{(X_{0})}\xrightarrow{\eta_0}M\rightarrow 0
\label{Catilt2}
\end{equation}%
such that; every map $\eta :T^{\prime }\rightarrow M$ factors through $%
\eta_0 $.

It follows that $\mathrm{Hom}_{\mathcal{C}}(T^{\prime },T^{(X_{0})})%
\xrightarrow{(T',\eta_0)}\mathrm{Hom}_{\mathcal{C}}(T^{\prime },M)$ is an
epimorphism. Since $\mathrm{Ext}^{1}_{\mathcal{C}}(T^{\prime
},T^{(X_{0})})=0 $, after applying $\mathrm{Hom}_{\mathcal{C}}(T^{\prime
},\;)$ to (\ref{Catilt2}), it follows by the long homology sequence, that $%
\mathrm{Ext}^{1}_{\mathcal{C}}(T^{\prime },K_{0})=0$, and $K_{0}$ is in $%
\mathscr{T}$. The claim follows by induction.

(iii) By part (ii) there exists an exact sequence $\cdots \rightarrow
T^{1}\rightarrow T^{0}\rightarrow M\rightarrow 0$. After applying $\phi $ we
obtain an exact sequence $\cdots \rightarrow \phi (T^{1})\rightarrow \phi
(T^{0})\rightarrow \phi (M)\rightarrow 0$. Since $\phi $ and $-\otimes
\mathcal{T}$ preserve arbitrary sums, for $i\geq 0$, we have, $\phi
(T^{i})\otimes \mathcal{T}=(\;\;,T^{i})_{\mathcal{T}}\otimes \mathcal{T}%
\cong T_{i}$. It follows, the existence of an isomorphism $\eta $, such that
the following diagram commutes:
\begin{equation*}
\begin{diagram}\dgARROWLENGTH=1.1em
\node{\phi(T^1)\otimes\mathcal{T}}\arrow{e,t}{}\arrow{s,l}{\cong}
\node{\phi(T^0)\otimes\mathcal{T}}\arrow{e,t}{}\arrow{s,l}{\cong}
\node{\phi(M)\otimes\mathcal{T}}\arrow{e,t}{}\arrow{s,l}{\eta} \node{0}\\
\node{T^1}\arrow{e,t}{} \node{T^0}\arrow{e,t}{} \node{M}\arrow{e,t}{}
\node{0} \end{diagram}
\end{equation*}

(iv) It is also clear from the above diagram $\mathrm{Tor}_{1}^{\mathcal{T}%
}(\phi (M),\mathcal{T})=0$.

(v) Using the sequence constructed in (ii) and the relation $\mathrm{Ext}_{%
\mathcal{C}}^{j}(T^{n},N)=0$ for all $j\geq 0$ and $n\geq 0$, we get by
dimension-shift $\mathrm{Ext}_{\mathcal{C}}^{j}(M,N)\cong \mathrm{Ext}^{1}_%
\mathcal{C}(\mathrm{Im}(t_{j-1}),N)$, for $j\geq 1$. The exact sequences
\begin{eqnarray*}
&T^{j+1}\rightarrow T^{j}\rightarrow \mathrm{Im}(t_{j})\rightarrow 0\text{,}&
\\
&0\rightarrow \mathrm{Im}(t_{j})\rightarrow T^{j-1}\rightarrow \mathrm{Im}%
(t_{j-1})\rightarrow 0\text{,}&
\end{eqnarray*}%
induce the following commutative diagram:

\xymatrix{
             &
             &0\ar[d]
             &
             &
             &\\
             &\mathrm{Hom}_{\mathcal{C}}(T^{j-1},N)\ar[rd]^{\mathrm{Hom}_{\mathcal{C}}(t_j,N)}\ar[r]^{\mathrm{Hom}_{\mathcal{C}}(q,N)}
             &\mathrm{Hom}_{\mathcal{C}}(\mathrm{Im}(t_j),N)\ar[r]\ar[d]^{\mathrm{Hom}_{\mathcal{C}}(p,N)}
             &\mathrm{Ext}^1_{\mathcal{C}}(\mathrm{Im}(t_{j-1}),N)\ar[r]^{}
             &0\\
             &
             &\mathrm{Hom}_{\mathcal{C}}(T^j,N)\ar[d]^{\mathrm{Hom}_{\mathcal{C}}(t_{j+1},N)}
             &
             &
             &\\
             &
             &\mathrm{Hom}_{\mathcal{C}}(T^{j+1},N)
             &
              &\\
        }

Hence, there are isomorphisms:
\begin{equation}  \label{BrennerB}
\mathrm{Ext}_{\mathcal{C}}^{j}(M,N)\cong \mathrm{Ext}^{1}_{\mathcal{C}}(%
\mathrm{Im}(t_{j-1}),N)\cong H^{j}(\mathrm{Hom}_{\mathcal{C}}(T.,N))\text{.}
\end{equation}

Applying $\phi $ to the exact sequence in (ii), we obtain the following
projective resolution of $\phi (M)$
\begin{equation*}
(\;\;,T.)\rightarrow (\;\;,M)_{\mathcal{T}}:\cdots \rightarrow
(\;\;,T^{1})\rightarrow (\;\;,T^{0})\rightarrow (\;\;,M)_{\mathcal{T}%
}\rightarrow 0\text{.}
\end{equation*}

By Yoneda's Lemma, $\mathrm{Hom}_{\mathcal{T}}((\;\;,T_{.}),(\;\;,N)_%
\mathcal{T})\cong \mathrm{Hom}_{\mathcal{C}}(T_{.},N)$, therefore:
\begin{equation}
\mathrm{Ext}_{\mathcal{T}}^{j}(\phi (M),\phi (N))=H^{j}((\;\;,T.),\phi
(N))=H^{j}(\mathrm{Hom}_\mathcal{C}(T.,N))\text{.}  \label{TORO}
\end{equation}%
The claim follows from (\ref{TORO}) and (\ref{BrennerB}).
\end{proof}

We will see next how to compute $-\otimes \mathcal{T}$, and its left derived
functors.

\begin{proposition}
\label{BB1} For each $C$ in $\mathcal{C}$, and $M\in \mathrm{\mathrm{Mod}}(%
\mathcal{T})$ there is a natural isomorphism in $\mathrm{Mod}($ $\mathcal{C)}
$,
\begin{equation*}
(M\otimes \mathcal{T})(C)\cong (\mathcal{C}(\;\;,C),\;\;)_{\mathcal{T}%
}\otimes _{\mathcal{T}}M\text{.}
\end{equation*}
\end{proposition}

\begin{proof}
Let's consider the following presentation of $M$
\begin{equation}
\coprod_{i\in I}(\;\;,T_{i})\overset{((\;,f_{ij}))}{\rightarrow }%
\coprod_{j\in J}(\;\;,T_{j})\rightarrow M\rightarrow 0.  \label{BB2}
\end{equation}

Applying $\otimes \mathcal{T}$ and evaluating in $C$, we obtain the
following short exact sequence,
\begin{equation}
\coprod_{i\in I}T_{i}(C)\xrightarrow {(f_{ij})_C}\coprod_{j\in
J}T_{j}(C)\rightarrow M\otimes \mathcal{T}\;(C)\rightarrow 0\text{.}
\label{BrennerB7}
\end{equation}

By the properties of the tensor product, and Yoneda's Lemma, we have
isomorphisms:
\begin{equation*}
(\mathcal{C}(\;\;,C),\;\;)\otimes _{\mathcal{T}}\coprod_{i\in
I}(\;\;,T_{i})\cong \coprod_{i\in I}(\mathcal{C}(\;\;,C),\;\;)\otimes _{%
\mathcal{T}}(\;\;,T_{i})\cong \coprod_{i\in I}(\mathcal{C}%
(\;\;,C),T_{i})\cong \coprod_{i\in I}T_{i}(C)
\end{equation*}

Since $(\mathcal{C}(\;\;,C),\;\;)\otimes _{\mathcal{T}}$, is right exact,
when we apply it to the presentation (\ref{BB2}), and compare the exact
sequence we get, with the exact sequence (\ref{BrennerB7}), we obtain the
desired isomorphism. It is easy to verify that this isomorphism is natural.
\end{proof}

\begin{proposition}
\label{BB4} For each $C$ in $\mathcal{C}$, and $M\in \mathrm{\mathrm{Mod}}(%
\mathcal{T})$, and any non negative integer $n$, there exists an isomorphism
\begin{equation}
\mathrm{Tor}_{n}^{\mathcal{T}}((\mathcal{C}(\;\;,C),\;\;)_{\mathcal{T}%
},M)\cong \mathrm{Tor}_{n}^{\mathcal{T}}(M,\mathcal{T})(C)\text{,}
\end{equation}%
which is natural in $C$ and $M.$
\end{proposition}

\begin{proof}
Let's consider an exact sequence
\begin{equation}
0\rightarrow \Omega M\rightarrow (\;\;,T)\rightarrow M\rightarrow 0\text{,}
\label{BB5}
\end{equation}%
where $T$ is a sum of objects in $\mathcal{T}$. Applying $\otimes \mathcal{T}
$ to (\ref{BB5}) we get the exact sequence
\begin{equation*}
0\rightarrow \mathrm{Tor}_{1}^{\mathcal{T}}(M,\mathcal{T})\rightarrow \Omega
M\otimes \mathcal{T}\rightarrow (\;\;,T)\otimes \mathcal{T}{,}
\end{equation*}%
and applying $(\mathcal{C}(\;\;,C),\;\;)\otimes _{\mathcal{T}}-$ to (\ref%
{BB5}), we get the exact sequence:
\begin{equation*}
0\rightarrow \mathrm{Tor}_{1}^{\mathcal{T}}(((\;\;,C),\;\;),M)\rightarrow
((\;\;,C),\;\;)\otimes \Omega M\rightarrow ((\;\;,C),\;\;)\otimes (\;\;,T)%
\text{.}
\end{equation*}

By Proposition \ref{BB1}, there exists an isomorphism $\gamma _{C}$, such
that the following diagram commutes

\xymatrix{
 0\ar[r]&\mathrm{Tor}_1^{\mathcal{T}}(((\;\;,C),\;\;),M)\ar[d]^{\gamma_C} \ar[r]
             &((\;\;,C),\;\;)\otimes\Omega M \ar[r]\ar[d]^{\cong}
             &((\;\;,C),\;\;)\otimes(\;\;,T)\ar[d]^{\cong}\\
 0\ar[r]&\mathrm{Tor}_1^{\mathcal{T}}(M,\mathcal{T})(C)\ar[r]
             &\Omega M\otimes \mathcal{T} (C) \ar[r]
             &T(C)
        }

We can easily verify that $\gamma =\{\gamma _{C}\}_{C\in \mathcal{C}}$ is a
natural transformation.

We leave to the reader the proof of the naturality in $M$.
\end{proof}

\begin{remark}
Let $C$ an $C^{\prime }$ be two objects in $\mathcal{C}$, and $\mathcal{T}$
a tilting subcategory of $\mathrm{\mathrm{Mod}}(\mathcal{C})$. Then we have
two exact sequences: $0\rightarrow \mathcal{C}(\;\;,C)\xrightarrow{g}%
T_{0}\rightarrow T_{1}\rightarrow 0$ and $0\rightarrow \mathcal{C}%
(\;\;,C^{\prime })\xrightarrow{g'}T_{0}^{\prime }\rightarrow T_{1}^{\prime
}\rightarrow 0$, with $T_{i}$ and $T_{i}^{\prime }$ in $\mathcal{T}$, for $%
i=0,1$. Let $f:C\rightarrow C^{\prime }$ be a morphism. Then, taking the
push out, we have the following commutative exact diagram:
\begin{equation*}
\begin{diagram}\dgARROWLENGTH=.7em \node{0}\arrow{e,}{}
\node{\mathcal{C}(\;\;,C)}\arrow{e,t}{g}\arrow{s,l}{g'(\mathcal{C}(\;\;,f))}
\node{T_0}\arrow{e,}{}\arrow{s,}{} \node{T_1}\arrow{e,}{}\arrow{s,l,=}{}
\node{0}\\ \node{0}\arrow{e,}{} \node{T'_0}\arrow{e,t}{}
\node{W}\arrow{e,}{} \node{T_1}\arrow{e,}{} \node{0} \end{diagram}
\end{equation*}%
Since $\mathrm{Ext}_{\mathcal{C}}^{1}(T_{1},T_{0}^{\prime })=0$, the lower
exact sequence in the diagram splits. Hence, there exist morphisms $u,v$,
such that the following exact diagram commutes:
\begin{equation}
\begin{diagram}\dgARROWLENGTH=.5em \node{0}\arrow{e,}{}
\node{\mathcal{C}(\;\;,C)}\arrow{e,t}{g}\arrow{s,l}{\mathcal{C}(\;\;,f)}
\node{T_0}\arrow{e,}{}\arrow{s,l}{u} \node{T_1}\arrow{e,}{}\arrow{s,l}{v}
\node{0}\\ \node{0}\arrow{e,}{} \node{\mathcal{C}(\;\;,C')}\arrow{e,t}{g'}
\node{T'_0}\arrow{e,}{} \node{T'_1}\arrow{e,}{} \node{0} \end{diagram}
\label{remark1}
\end{equation}
\end{remark}

\begin{proposition}
\label{BB9} Let $\mathcal{T}$ be a tilting subcategory of $\mathrm{\mathrm{%
Mod}}(\mathcal{C})$. Then for each $M$ in $\mathrm{\mathrm{Mod}}(\mathcal{C}%
) $, there is an exact sequence
\begin{equation*}
0\rightarrow \phi (M)\otimes \mathcal{T}\rightarrow M\rightarrow \mathrm{Tor}%
_{1}^{\mathcal{T}}(\mathrm{Ext}_{\mathcal{C}}^{1}(\;\;,M)_{\mathcal{T}},%
\mathcal{T})\rightarrow 0\text{.}
\end{equation*}%
Moreover, $\mathrm{Ext}_{\mathcal{C}}^{1}(\;\;,M)_{\mathcal{T}}\otimes
\mathcal{T}=0$.
\end{proposition}

\begin{proof}
For each $C$ in $\mathcal{C}$ we have an exact sequence
\begin{equation}
0\rightarrow \mathcal{C}(\;,C)\rightarrow T_{0}\rightarrow T_{1}\rightarrow 0%
\text{,}  \label{BrennerB2}
\end{equation}%
which induces the following exact sequence
\begin{equation}
0\rightarrow \mathrm{Hom}_{\mathcal{C}}(T_{1},\;\;)_{\mathcal{T}}\rightarrow
\mathrm{Hom}_{\mathcal{C}}(T_{0},\;\;)_{\mathcal{T}}\rightarrow \mathrm{Hom}%
_{\mathcal{C}}(\mathcal{C}(\;\;,C),\;\;)_{\mathcal{T}}\rightarrow 0\text{.}
\label{BrennerB8}
\end{equation}

After applying $-\otimes _{\mathcal{T}}\mathrm{Ext}_{\mathcal{C}%
}^{1}(\;\;,M)_{\mathcal{T}}$ to (\ref{BrennerB8}), we get the following
exact sequence
\begin{eqnarray}  \label{BrennerB3}
0 &\rightarrow &\mathrm{Tor}_{1}^{\mathcal{T}}((\mathcal{C}(\;\;,C),\;)_{%
\mathcal{T}},\mathrm{Ext}_{\mathcal{C}}^{1}(\;\;,M)_{\mathcal{T}%
})\rightarrow \mathrm{Ext}_{\mathcal{C}}^{1}(T_{1},M)\rightarrow \\
&\rightarrow &\mathrm{Ext}_{\mathcal{C}}^{1}(T_{0},M)\rightarrow \mathrm{Hom}%
(\mathcal{C}(\;\;,C),\;\;)_{\mathcal{T}}\otimes _{\mathcal{T}}\mathrm{Ext}_{%
\mathcal{C}}^{1}(\;\;,M)_{\mathcal{T}}\rightarrow 0\text{.}  \notag
\end{eqnarray}

Applying $\mathrm{Hom}_{\mathcal{C}}(\;\;,M)$ to (\ref{BrennerB2}), we get
from the long homology sequence, the exact sequence
\begin{eqnarray}
0 &\rightarrow &\mathrm{Hom}_{\mathcal{C}}(T_{1},M)\rightarrow \mathrm{Hom}_{%
\mathcal{C}}(T_{0},M)\rightarrow \mathrm{Hom}_{\mathcal{C}}(\mathcal{C}%
(\;\;,C),M)\rightarrow  \label{BrennerB5} \\
&\rightarrow &\mathrm{Ext}_{\mathcal{C}}^{1}(T_{1},M)\rightarrow \mathrm{Ext}%
_{\mathcal{C}}^{1}(T_{0},M)\rightarrow \mathrm{Ext}_{\mathcal{C}}^{1}(%
\mathcal{C}(\;\;,C),M)=0\text{.}  \notag
\end{eqnarray}

Then, for each $C$ in $\mathcal{C}$ we have the following isomorphisms:
\begin{eqnarray*}
(\mathrm{Ext}_{\mathcal{C}}^{1}(\;\;,M)_{\mathcal{T}}\otimes \mathcal{T})(C)
&\cong &(\mathcal{C}(\;\;,C),\;\;)_{\mathcal{T}}\otimes _{\mathcal{T}}%
\mathrm{Ext}_{\mathcal{C}}^{1}(\;\;,M)_\mathcal{T},\text{ by Proposition \ref%
{BB1}} \\
&\cong &\mathrm{Ext}_{\mathcal{C}}^{1}(\mathcal{C}(\;\;,C),M)_{\mathcal{T}%
}=0,\text{ by (\ref{BrennerB3}) and (\ref{BrennerB5}),}
\end{eqnarray*}%
which proves the second part of the Proposition.

Now, by Proposition \ref{BB4} we know that
\begin{equation}  \label{BBR01}
\mathrm{Tor}_{1}^{\mathcal{T}}(\mathrm{Ext}^1_{\mathcal{C}}(\;\;,M)_{%
\mathcal{T}},\mathcal{T})(C)\cong \mathrm{Tor}_{1}^{\mathcal{T}}((\mathcal{C}%
(\;\;,C),\;)_\mathcal{T},\mathrm{Ext}_{\mathcal{C}}^{1}(\;\;,M)_{\mathcal{T}%
}).
\end{equation}%
It follows by \ref{BBR01}, and the isomorphism $\mathrm{Hom}_{\mathcal{C}}(%
\mathcal{C}(\;\;,C),M)\cong M(C),$ together with (\ref{BrennerB3}), and (\ref%
{BrennerB5}), that there exist a isomorphisms $\gamma _{C}$ such that the
following diagram commutes:
\begin{equation*}
\begin{diagram}\dgARROWLENGTH=.3em \node{}
\node{M(C)}\arrow{e}\arrow{s,l}{\gamma_C}
\node{\mathrm{Ext}^1_{\mathcal{C}}(T_1,M)}\arrow{e}\arrow{s,l}{1}
\node{\mathrm{Ext}^1_{\mathcal{C}}(T_0,M)\rightarrow 0}\arrow{s,l}{1} \\
\node{0}\arrow{e}
\node{\mathrm{Tor}^{\mathcal{T}}_1(\mathrm{Ext}^1_{\mathcal{C}}(\;\;,M)_{%
\mathcal{T}},\mathcal T)(C)}\arrow{e}
\node{\mathrm{Ext}^1_{\mathcal{C}}(T_1,M)}\arrow{e}
\node{\mathrm{Ext}^1_{\mathcal{C}}(T_0,M)} \end{diagram}
\end{equation*}%
Moreover, the Snake Lemma implies $\gamma _{C}$ is epimorphisms. Using
Remark \ref{remark1}, the reader can check that $\gamma =\{\gamma
_{C}\}_{C\in \mathcal{C}}$ is a natural transformation.

After applying $-\otimes _{\mathcal{T}}\phi (M)$ to the exact sequence (\ref%
{BrennerB8}) we obtain the following exact sequence
\begin{equation*}
\phi (M)(T_{1})\rightarrow \phi (M)(T_{0})\rightarrow \mathrm{Hom}(\mathcal{C%
}(\;\;,C),\;\;)\otimes _{\mathcal{T}}\phi (M)\rightarrow 0\text{.}
\end{equation*}

By the isomorphism $\phi (M)\otimes \mathcal{T}(C)\cong \mathrm{Hom}(%
\mathcal{C}(\;\;,C),\;\;)_\mathcal{T}\otimes _{\mathcal{T}}\phi (M)$, and
the exact sequence (\ref{BrennerB5}), there exist a morphisms $\eta _{C}$
such that the following diagram commutes:

\begin{equation*}
\begin{diagram}\dgARROWLENGTH=.3em \node{}
\node{\phi(M)(T_1)}\arrow{s,l}{\cong}\arrow{e}
\node{\phi(M)(T_0)}\arrow{e}\arrow{s,l}{\cong} \node{\phi(M)\otimes\mathcal
T(C)}\arrow{e}\arrow{s,l}{\eta_C} \node{0}\\ \node{0}\arrow{e}
\node{\mathrm{Hom}(T_1,M)}\arrow{e} \node{\mathrm{Hom}(T_0,M)}\arrow{e}
\node{\mathrm{Hom}(\mathcal{C}(\;C),M)} \end{diagram}
\end{equation*}

Using the Snake Lemma again, the map $\eta _{C}$ is monomorphism. The reader
can check $\eta =\{\eta _{C}\}_{C\in \mathcal{C}}$ is natural
transformation. In this way we have the following exact sequence
\begin{equation*}
0\rightarrow \phi (M)\otimes \mathcal{T}\xrightarrow{\eta}M%
\xrightarrow{\gamma}\mathrm{Tor}_{1}^{\mathcal{T}}(\mathrm{Ext}_{\mathcal{C}%
}^{1}(\;\;,M),\mathcal{T})\rightarrow 0\text{.}
\end{equation*}%
Proving the proposition.
\end{proof}

\begin{proposition}
Let $\mathcal{T}$ be a tilting subcategory of $\mathrm{\mathrm{Mod}}(%
\mathcal{C})$. Then for each $N$ in $\mathrm{Mod}(\mathcal{T)}$ there exists
an exact sequence
\begin{equation*}
0\rightarrow \mathrm{Ext}_{\mathcal{C}}^{1}(-,\mathrm{Tor}^{\mathcal{T}}_1(N,%
\mathcal{T}))\rightarrow N\rightarrow \phi (N\otimes \mathcal{T})\rightarrow
0.
\end{equation*}%
In addition $\phi (\mathrm{Tor}_{1}^{\mathcal{T}}(N,\mathcal{T}))=0$.
\end{proposition}

\begin{proof}
We choose the projective resolutions
\begin{eqnarray}
L_{.}\rightarrow N &:&\cdots \rightarrow (-,T_{2})\xrightarrow{(-,h_2)}%
(-,T_{1})\xrightarrow {(-,h_1)}(-,T_{0})\rightarrow N\rightarrow 0
\label{BB17} \\
&&0\rightarrow P_{1}\rightarrow P_{0}\rightarrow T\rightarrow 0
\end{eqnarray}%
such that for $i\geq 0$, $T_{i}$ is a sum of objects in $\mathcal{T}$.

Applying $-\otimes \mathcal{T}$ to the complex $L.$, we obtain the complex $%
L.\otimes \mathcal{T}$, whose objects are in $\mathrm{Add}\mathcal{T}$, and
we have $\mathrm{Ext}_{\mathcal{C}}^{1}(T,L_{.}\otimes \mathcal{T})=0$. In
this way, we obtain the following exact sequence of complexes:
\begin{equation}  \label{BB18}
0\rightarrow \mathrm{Hom}_{\mathcal{C}}(T,L_{.}\otimes \mathcal{T}%
)\rightarrow \mathrm{Hom}_{\mathcal{C}}(P_{0},L_{.}\otimes \mathcal{T}%
)\rightarrow \mathrm{Hom}_{\mathcal{C}}(P_{1},L_{.}\otimes \mathcal{T}%
)\rightarrow 0\text{.}
\end{equation}

Observe that $L_{.}(T)$ and $\mathrm{Hom}_{\mathcal{C}}(T,L_{.}\otimes
\mathcal{T})$ are isomorphic, hence the exact sequence (\ref{BB18}) becomes%
\begin{equation*}
0\rightarrow L_{.}(T)\rightarrow \mathrm{Hom}_{\mathcal{C}%
}(P_{0},L_{.}\otimes \mathcal{T})\rightarrow \mathrm{Hom}_{\mathcal{C}%
}(P_{1},L_{.}\otimes \mathcal{T})\rightarrow 0\text{.}
\end{equation*}

By the above sequence, and the long homology sequence, we get an exact
sequence:
\begin{eqnarray}  \label{BB20}
0 =H_{1}(L_{.}(T))\rightarrow H_{1}(\mathrm{Hom}_{\mathcal{C}%
}(P_{0},L_{.}\otimes \mathcal{T}))\rightarrow H_{1}(\mathrm{Hom}_{\mathcal{C}%
}(P_{1},L_{.}\otimes \mathcal{T}))\rightarrow \\
\rightarrow H_{0}(\mathrm{Hom}_{\mathcal{C}}(L_{.}(T)))\rightarrow H_{0}(%
\mathrm{Hom}_{\mathcal{C}}(P_{0},L_{.}\otimes \mathcal{T}))\rightarrow H_{0}(%
\mathrm{Hom}_{\mathcal{C}}(P_{1},L_{.}\otimes \mathcal{T}))\rightarrow 0%
\text{.}  \notag
\end{eqnarray}

Since $P_{i}$, $i=0,1$ are projective, there exists an isomorphism
\begin{equation*}
H_{1}(\mathrm{Hom}_{\mathcal{C}}(P_{i},L_{.}\otimes \mathcal{T}))\cong
\mathrm{Hom}_{\mathcal{C}}(P_{i},H_{1}(L_{.}\otimes \mathcal{T}))=\mathrm{Hom%
}_{\mathcal{C}}(P_{i},\mathrm{Tor}_{1}^{\mathcal{T}}(N,\mathcal{T})).
\end{equation*}%
Finally, the exact sequence (\ref{BB20}) can be written as:
\begin{eqnarray}
0 &\rightarrow &\mathrm{Hom}_{\mathcal{C}}(P_{0},\mathrm{Tor}_{1}^{\mathcal{T%
}}(N,\mathcal{T}))\rightarrow \mathrm{Hom}_{\mathcal{C}}(P_{1},\mathrm{Tor}%
_{1}^{\mathcal{T}}(N,\mathcal{T}))\rightarrow  \label{BB21} \\
N(T) &\rightarrow &\mathrm{Hom}_{\mathcal{C}}(P_{0},N\otimes \mathcal{T}%
)\rightarrow \mathrm{Hom}_{\mathcal{C}}(P_{1},N\otimes \mathcal{T}%
)\rightarrow 0\text{.}  \notag
\end{eqnarray}

After applying $\mathrm{Hom}_{\mathcal{C}}(\;\;,\mathrm{Tor}^{\mathcal{T}%
}_1(N,\mathcal{T}))$ to the projective resolution of $T$, we obtain the
following exact sequence:
\begin{eqnarray}
0 &\rightarrow &\mathrm{Hom}_{\mathcal{C}}(T,\mathrm{Tor}_{1}^{\mathcal{T}%
}(N,\mathcal{T}))\rightarrow \mathrm{Hom}_{\mathcal{C}}(P_{0},\mathrm{Tor}%
_{1}^{\mathcal{T}}(N,\mathcal{T}))\rightarrow  \label{BB22} \\
&\rightarrow &\mathrm{Hom}_{\mathcal{C}}(P_{1},\mathrm{Tor}_{1}^{\mathcal{T}%
}(N,\mathcal{T}))\rightarrow \mathrm{Ext}_{\mathcal{C}}^{1}(T,\mathrm{Tor}%
_{1}^{\mathcal{T}}(N,\mathcal{T}))\rightarrow 0\text{.}  \notag
\end{eqnarray}

Comparing (\ref{BB21}) and (\ref{BB22}), we get $\phi (\mathrm{Tor}_{1}^{%
\mathcal{T}}(N,\mathcal{T}))(T)=0$, which proves part of the Proposition. In
addition (\ref{BB21}) and (\ref{BB22}), imply the existence of a morphism $%
\eta _{T}$, such that the following diagram commutes:

\begin{equation*}
\begin{diagram}\dgARROWLENGTH=.3em
\node{(P_{0},\mathrm{Tor}_{1}^{\mathcal{T}}(N,\mathcal{T}))}\arrow{s,l}{%
\cong}\arrow{e}
\node{(P_{1},\mathrm{Tor}_{1}^{\mathcal{T}}(N,\mathcal{T}))}\arrow{e}%
\arrow{s,l}{\cong}
\node{\mathrm{Ext}^1_{\mathcal{C}}(T,\mathrm{Tor}^\mathcal
T_1(N,\mathcal{T})\rightarrow 0}\arrow{s,l}{\eta_T} \\ \node{0\rightarrow
(P_{0},\mathrm{Tor}_{1}^{\mathcal{T}}(N,\mathcal{T}))}\arrow{e}
\node{(P_{1},\mathrm{Tor}_{1}^{\mathcal{T}}(N,\mathcal{T}))}\arrow{e}
\node{N(T)} \end{diagram}
\end{equation*}

By the Snake Lemma, $\eta _{T}$ is mono.

After applying $\mathrm{Hom}_{\mathcal{C}}(\;\;,N\otimes \mathcal{T})$ to
the projective resolution of $T$, we obtain the following exact sequence:
\begin{eqnarray}
0 &\rightarrow &\mathrm{Hom}_{\mathcal{C}}(T,N\otimes \mathcal{T}%
)\rightarrow \mathrm{Hom}_{\mathcal{C}}(P_{0},N\otimes \mathcal{T}%
)\rightarrow  \label{BB24} \\
&\rightarrow &\mathrm{Hom}_{\mathcal{C}}(P_{1},N\otimes \mathcal{T}%
)\rightarrow \mathrm{Ext}_{\mathcal{C}}(T,N\otimes \mathcal{T})\rightarrow 0%
\text{.}  \notag
\end{eqnarray}

The sequences (\ref{BB21}), and (\ref{BB24}), imply the existence of a
morphism $\gamma _{T}:N(T)\rightarrow \mathrm{Hom}_{\mathcal{C}}(T,N\otimes
\mathcal{T})=\phi (N\otimes \mathcal{T})$. It follows by the Snake Lemma
that $\gamma _{T}$ is epimorphism. Moreover,
\begin{equation}  \label{BBI}
\mathrm{Ext}_{\mathcal{C}}^{1}(\;\;,N\otimes \mathcal{T})=0\text{.}
\end{equation}

It's not hard to check that $\eta =\{\eta _{T}\}_{T\in \mathcal{T}}$ and $%
\gamma =\{\gamma _{T}\}_{T\in \mathcal{T}}$ are natural transformations. We
leave the details to the reader. In this way the following sequence of
functors
\begin{equation*}
0\rightarrow \mathrm{Ext}_{\mathcal{C}}^{1}(-,\mathrm{Tor}_{1}^{\mathcal{T}%
}(N,\mathcal{T}))\xrightarrow{\eta}N\xrightarrow {\gamma}\phi (N\otimes
\mathcal{T})\rightarrow 0
\end{equation*}%
is exact, proving the Proposition.
\end{proof}

We call $F,F^{\prime }:\mathrm{\mathrm{Mod}}(\mathcal{C})\rightarrow \mathrm{%
\mathrm{Mod}}(\mathcal{T})$ to the functors defined by:
\begin{eqnarray*}
F(M) &=&\phi (M)\text{,} \\
F^{\prime }(M) &=&\mathrm{Ext}_{\mathcal{C}}^{1}(\;\;,M)_{\mathcal{T}}\text{,%
}
\end{eqnarray*}

and, $G,G^{\prime }:\mathrm{\mathrm{Mod}}(\mathcal{T})\rightarrow \mathrm{%
\mathrm{Mod}}(\mathcal{C})$ to the functors:%
\begin{eqnarray*}
G(N) &=&N\otimes \mathcal{T}\text{,} \\
G^{\prime }(N) &=&\mathrm{Tor}_{1}^{\mathcal{T}}(N,\mathcal{T})\text{.}
\end{eqnarray*}

Let $\mathcal{T}$ be a tilting subcategory of $\mathrm{\mathrm{Mod}}(%
\mathcal{C})$, and $(\mathscr T,\mathscr F)$ the torsion theory considered
above. We look to the full subcategories $\mathscr{X}$ and $\mathscr{Y}$ of $%
\mathrm{\mathrm{Mod}}(\mathcal{T})$, defined by:
\begin{eqnarray*}
\mathscr{X} &=&\{N\in\mathrm{Mod}(\mathcal{T})|N\otimes \mathcal{T}=0\}\text{%
,} \\
\mathscr{Y} &=&\{N\in\mathrm{Mod}(\mathcal{T})|\mathrm{Tor}_{1}^{\mathcal{T}%
}(N,\mathcal{T})=0\}\text{.}
\end{eqnarray*}

From the previous results, the main theorem of the section follows:

\begin{theorem}[Brenner-Butler]
With the above notation, the following statements are true:

\begin{itemize}
\item[(i)] $F$ and $G$ induce an equivalence between $\mathscr{T}$ and $%
\mathscr{Y}.$

\item[(ii)] $F^{\prime }$ and $G^{\prime }$ induce an equivalence between $%
\mathscr{F}$ and $\mathscr{X}.$

\item[(iii)] The following equations $FG^{\prime }=F^{\prime }G=0$ and $%
G^{\prime }F=GF^{\prime }=0$ hold..
\end{itemize}
\end{theorem}

\begin{corollary}
The pair of subactegories $(\mathscr{X},\mathscr{Y})$ form a torsion theory
in $\mathrm{\mathrm{Mod}}(\mathcal{T}).$
\end{corollary}

\begin{proof}
It is easy to check that for pair of objects $X$ in $\mathscr{X}$, and $Y$
in $\mathscr{Y}$, $\mathrm{Hom}_{\mathcal{T}}(X,Y)=0$.

For every $N$ in $\mathrm{\mathrm{Mod}}(\mathcal{T})$ there is an exact
sequence:

\begin{equation*}
0\rightarrow \mathrm{Ext}_{\mathcal{C}}^{1}(-,\mathrm{Tor}_{1}^{\mathcal{T}%
}(N,\mathcal{T}))\xrightarrow{\eta}N\xrightarrow {\gamma}\phi (N\otimes
\mathcal{T})\rightarrow 0\text{.}
\end{equation*}

By condition (iii) of the Theorem, $\mathrm{Ext}_{\mathcal{C}}^{1}(-,\mathrm{%
Tor}_{1}^{\mathcal{T}}(N,\mathcal{T}))$ is in $\mathscr{X}$ and $\phi
(N\otimes \mathcal{T})$ is in $\mathscr{Y}$, which implies the pair $(%
\mathscr{X},\mathscr{Y})$ is a torsion theory.
\end{proof}

We want next to generalize the following result on tilting of finite
dimensional algebras:

Let $A$ be a finite dimensional $K$-algebra and $_{A}\mathrm{mod}$ the
category of finitely generated left $A$-modules. Let's suppose that $_{A}T$
is a tilting $A$-module and $B=\mathrm{End}_{A}(T)$. Tilting theorem [B]
proves $T_{B}$ is a right tilting $B$-module and $A^{op}$ is isomorphic to $%
\mathrm{End}_{B}(T_{B})$.

\begin{proposition}
\label{TT} Let $\mathcal{T}$ be a tilting subcategory of $\mathrm{Mod}(%
\mathcal{C})$. Let's assume each $T$ in $\mathcal{T}$ has a projective
resolution of finitely generated projectives,
\begin{equation*}
0\rightarrow P_{1}\rightarrow P_{0}\rightarrow T\rightarrow 0\text{.}
\end{equation*}%
Then the following statements hold:

\begin{itemize}
\item[(a)] The full subcategory $\theta $ of $\mathrm{\mathrm{Mod}}(\mathcal{%
T}^{op})$, with objects $\{(\mathcal{C}(\;\;,C),\;\;)_{\mathcal{T}}\}_{C\in
\mathcal{C}}$ is a tilting subcategory in $\mathrm{\mathrm{Mod}}(\mathcal{T}%
^{op})$.

\item[(b)] The category $\theta =\{(\mathcal{C}(\;\;,C),\;\;)_{\mathcal{T}%
}\}_{C\in \mathcal{C}}$ is equivalent to $\mathcal{C}^{op}$.
\end{itemize}
\end{proposition}

\begin{proof}
(a)(i) For each object $C$ in $\mathcal{C}$ there is a resolution
\begin{equation}  \label{Tilop}
0\rightarrow \mathcal{C}(\;\;,C)\rightarrow T_{0}\xrightarrow{f}%
T_{1}\rightarrow 0\text{,}
\end{equation}%
with $T_{i}\in \mathcal{T}$, for $i=0,1$. By the long homology sequence,
there is an exact sequence of objects in $\mathrm{Mod}(\mathcal{C})$,
\begin{equation}
0\rightarrow (T_{1},\;\;)_{\mathcal{T}}\xrightarrow {(f,\;\;)}(T_{0},\;\;)_{%
\mathcal{T}}\rightarrow (\mathcal{C}(\;\;,C),\;\;)_{\mathcal{T}}\rightarrow
\mathrm{Ext}(T_{1},\;\;)_{\mathcal{T}}=0\text{,}  \label{TT1}
\end{equation}%
this is; $\mathrm{pdim}(\mathcal{C}(\;\;,C),\;\;)_{\mathcal{T}}\leq 1$.

(ii) Let $C^{\prime }$ be another object in $\mathcal{C}$. After applying
the functor $\mathrm{Hom}_{\mathcal{T}^{op}}(\;\;,(\mathcal{C}%
(\;\;,C^{\prime }),\;\;)_{\mathcal{T}})$ to (\ref{TT1}), we get an exact
sequence:
\begin{eqnarray}  \label{TT2}
0 \rightarrow ((\mathcal{C}(\;\;,C),\;\;)_{\mathcal{T}},(\mathcal{C}%
(\;\;,C^{\prime }),\;\;)_{\mathcal{T}})\rightarrow ((T_{0},\;\;),(\mathcal{C}%
(\;\;,C^{\prime }),\;\;)_{\mathcal{T}})\rightarrow \\
\rightarrow ((T_{1},\;\;),(\mathcal{C}(\;\;,C^{\prime }),\;\;)_{\mathcal{T}%
})\rightarrow \mathrm{Ext}_{\mathcal{T}^{op}}^{1}((\mathcal{C}%
(\;\;,C),\;\;)_{\mathcal{T}},(\mathcal{C}(\;\;,C^{\prime }),\;\;)_{\mathcal{T%
}})\rightarrow 0 \text{.}  \notag
\end{eqnarray}

By Yoneda's Lemma the following diagram:
\begin{equation*}
\begin{diagram}
\node{((T_0,\;\;),(\mathcal{C}(\;\;,C'),\;\;)_{\mathcal{T}})}\arrow{e,t}{}%
\arrow{s,l}{\cong}
\node{((T_1,\;\;),(\mathcal{C}(\;\;,C'),\;\;)_{\mathcal{T}})}\arrow{s,l}{%
\cong}\\
\node{(\mathcal{C}(\;\;,C'),T_0)}\arrow{e,t}{(\mathcal{C}(\;\;,C'),f)}%
\arrow{s,l}{\cong} \node{(\mathcal{C}(\;\;,C'),T_1)}\arrow{s,l}{\cong}\\
\node{T_0(C')}\arrow{e,t}{f_{C'}} \node{T_1(C')} \end{diagram}
\end{equation*}
commutes.

Evaluating (\ref{Tilop}) in $C^{\prime }$, and using Yoneda's Lemma,
together with (\ref{TT2}), there exist the following isomorphisms of abelian
groups:%
\begin{eqnarray}
((\mathcal{C}(\;\;,C),\;\;)_{\mathcal{T}},(\mathcal{C}(\;\;,C^{\prime
}),\;\;)_{\mathcal{T}}) &\cong &\mathcal{C}(C^{\prime },C)  \label{TT5} \\
\mathrm{Ext}_{\mathcal{T}^{op}}^{1}((\mathcal{C}(\;\;,C),\;\;)_{\mathcal{T}%
},(\mathcal{C}(\;\;,C^{\prime }),\;\;)_{\mathcal{T}}) &\cong &0 \text{.}
\notag
\end{eqnarray}

(iii) The sequence
\begin{equation*}
0\rightarrow (T,\;\;)_{\mathcal{T}}\rightarrow (P_{0},\;\;)_{\mathcal{T}%
}\rightarrow (P_{1},\;\;)_{\mathcal{T}}\rightarrow \mathrm{Ext}_{\mathcal{T}%
^{op}}^{1}(T,\;\;)_{\mathcal{T}}=0
\end{equation*}

is exact, and each $(P_{i},\;\;)_{\mathcal{T}}$ is in $\mathrm{add}\{(%
\mathcal{C}(\;\;,C),\;\;)_{\mathcal{T}}\}_{C\in \mathcal{C}}$.

(b) We define the functor:
\begin{equation*}
\alpha :\mathcal{C}^{op}\rightarrow \theta,\;\alpha (C)=(\mathcal{C}%
(\;\;,C),\;\;)_\mathcal{T}
\end{equation*}%
Which by (\ref{TT5}), is full and faithful, giving the desired equivalence.
\end{proof}

\subsection{ The Grothendieck Groups $K_{0}(\mathcal{C})$ and $K_{0}(%
\mathcal{T}$).}

Given a ring $A$ and a tilting $A$-module $T$ it is a classical theorem [B]
that the Grotendieck groups of $A$ and $B=\mathrm{End}_{A}(T)^{op}$ are
isomorphic. In this subsection we will prove that there is also an
isomorphism between the Grotendieck group, $K_{0}(\mathcal{C})$, of an
arbitrary skeletally small pre additive category $\mathcal{C},$ and the
Grotendieck group, $K_{0}(\mathcal{T}),$ of a tilting subcategory $\mathcal{T%
}$ of $\mathrm{\mathrm{Mod}}(\mathcal{C})$. The proof will follow closely
[CF].

\begin{definition}
Let $\mathcal{C}$ a skeletally small pre additive category $\mathcal{C}$ and
let $\mathcal{T}$ be a tilting subcategory of $\mathrm{\mathrm{Mod}}(%
\mathcal{C})$. Let's denote by $|\mathrm{\mathrm{Mod}}(\mathcal{C})|$ the
set of isomorphism classes of objects in $\mathrm{\mathrm{Mod}}(\mathcal{C})$%
. Let $\mathcal{A}$ be the free abelian group generated by $|\mathrm{\mathrm{%
Mod}}(\mathcal{C})|$ and $\mathcal{R}$ the subgroup of $\mathcal{A}$
generated by relations $M-K-L$ such that $0\rightarrow K\rightarrow
M\rightarrow L\rightarrow 0$ is a short exact sequence in $\mathrm{\mathrm{%
Mod}}(\mathcal{C})$. Then, the Groethendiek group of $\mathcal{C}$ is $K_{0}(%
\mathcal{C})=\mathcal{A}/\mathcal{R}$.
\end{definition}

\begin{proposition}
The Groethendieck groups $K_{0}(\mathcal{C})$ and $K_{0}(\mathcal{T})$ are
isomorphic.
\end{proposition}

\begin{proof}
We define the group homomorphism $\hat{\phi}:\mathcal{A}\rightarrow K_{0}(%
\mathcal{T})$, sending $\hat{\phi}(M)=|F(M)|-|F^{\prime }(M)|$, where $F$
and $F^{\prime }$are the functors given in Brenner-Buttler's theorem. We
claim $\mathcal{R}$ is contained in the kernel of $\hat{\phi}$. In fact, let
$M-K-L$ be a generator of $\mathcal{R}$. Then there exists an exact sequence
in $\mathrm{\mathrm{Mod}}(\mathcal{T})$,
\begin{eqnarray}
0 &\rightarrow &(\;\;,K)_{\mathcal{T}}\rightarrow (\;\;,M)_{\mathcal{T}%
}\rightarrow (\;\;,L)_{\mathcal{T}}\rightarrow \mathrm{Ext}_{\mathcal{C}%
}^{1}(\;\;,K)_{\mathcal{T}}\rightarrow  \notag  \label{GG1} \\
&\rightarrow &\mathrm{Ext}_{\mathcal{C}}^{1}(\;\;,M)_{\mathcal{T}%
}\rightarrow \mathrm{Ext}_{\mathcal{C}}^{1}(\;\;,L)_{\mathcal{T}}\rightarrow
\mathrm{Ext}_{\mathcal{C}}^{2}(\;\;,K)_{\mathcal{T}}=0\text{,}
\end{eqnarray}

hence, the alternating sum,
\begin{equation*}
-(|F(K)|-|F^{\prime }(K)|)-(|F(L)|-|F^{\prime }(L)|)+(|F(M)|+|F^{\prime
}(M)|)=0.
\end{equation*}%
Which means $\hat{\phi}(M-K-L)=0$. Then, there is a unique map $\phi :K_{0}(%
\mathcal{C})\rightarrow K_{0}(\mathcal{T})$, given by $\phi
(|M|)=|F(M)|-|F^{\prime }(M)|$.

For each object $C$ in $\mathcal{C}$, there is a short exact sequence $%
0\rightarrow \mathcal{C}(\;\;,C)\rightarrow T^{0}\rightarrow
T^{1}\rightarrow 0$ with $T^{i}$ in $\mathcal{T}$, for $i=0,1$, which
induces the exact sequence in $\mathrm{\mathrm{Mod}}(\mathcal{T}^{op})$:
\begin{equation*}
0\rightarrow (T_{1},\;\;)_{\mathcal{T}}\rightarrow (T_{0},\;\;)_{\mathcal{T}%
}\rightarrow (\mathcal{C}(\;\;,C),\;\;)_{\mathcal{T}}\rightarrow \mathrm{Ext}%
_{\mathcal{C}}^{1}(T_{1},\;\;)=0\text{.}
\end{equation*}%
Therefore: for $n>1$, and any object $N$ in $\mathrm{Mod}(\mathcal{T})$, $%
\mathrm{pdim}(\mathcal{C}(\;\;,C),\;\;)_{\mathcal{T}}\leq 1$ implies $%
\mathrm{Tor}_{n}^{\mathcal{T}}((\mathcal{C}(\;\;,C),\;\;)_{\mathcal{T}},N)=0$%
. For an exact sequence $0\rightarrow K\rightarrow N\rightarrow L\rightarrow
0$ in $\mathrm{\mathrm{Mod}}(\mathcal{T})$, and any object $C$ in $\mathcal{C%
}$, there is an exact sequence:
\begin{eqnarray*}
0 &\rightarrow &\mathrm{Tor}_{1}^{\mathcal{T}}(\mathcal{C}%
(\;\;,C),\;\;),K)\rightarrow \mathrm{Tor}_{1}^{\mathcal{T}}(\mathcal{C}%
(\;\;,C),\;\;),N)\rightarrow \mathrm{Tor}_{1}^{\mathcal{T}}(\mathcal{C}%
(\;\;,C),\;\;),L)\rightarrow \\
&\rightarrow &(\mathcal{C}(\;\;,C),\;\;)\otimes K\rightarrow (\mathcal{C}%
(\;\;,C),\;\;)\otimes N\rightarrow (\mathcal{C}(\;\;,C),\;\;)\otimes
L\rightarrow 0.
\end{eqnarray*}%
Which can be re written as follows:
\begin{eqnarray}
0 &\rightarrow &\mathrm{Tor}_{1}^{\mathcal{T}}(K,\mathcal{T})\rightarrow
\mathrm{Tor}_{1}^{\mathcal{T}}(N,\mathcal{T})\rightarrow \mathrm{Tor}_{1}^{%
\mathcal{T}}(L,\mathcal{T})\rightarrow  \notag  \label{GG2} \\
&\rightarrow &K\otimes \mathcal{T}\rightarrow N\otimes \mathcal{T}%
\rightarrow L\otimes \mathcal{T}\rightarrow 0
\end{eqnarray}

In an analogous way, there is a group homomorphism $\psi :K_{0}(\mathcal{T}%
)\rightarrow K_{0}(\mathcal{C})$, given by $\psi (|N|)=|G(N)|-|G^{\prime
}(N)|$.

By Brenner-Butler's theorem, there are isomorphisms:
\begin{eqnarray*}
\psi \phi (|M|) &=&\psi (|F(M)|-|F^{\prime }(M)|) \\
&=&\psi (|F(M)|)-\psi (|F^{\prime }(M)|) \\
&=&(|GF(M)|-|G^{\prime }F(M)|)-(|GF^{\prime }(M)|-|G^{\prime }F^{\prime
}(M)|) \\
&=&|GF(M)|-|G^{\prime }F^{\prime }(M)|\text{,}
\end{eqnarray*}%
and from the exact sequence
\begin{equation*}
0\rightarrow GF(M)\rightarrow M\rightarrow G^{\prime }F^{\prime
}(M)\rightarrow 0
\end{equation*}%
it follows: $|M|=|GF(M)|-|G^{\prime }F^{\prime }(M)|$, this is: $\psi \phi
=1_{K_{0}(\mathcal{C})}$. With a similar argument we prove $\phi \psi
=1_{K_{0}(\mathcal{T})}$.
\end{proof}

\subsection{ Global Dimension and Tilting}

In this subsection we will compare the global dimensions of a category $%
\mathcal{C}$ and its tilting category $\mathcal{T}$, obtaining results
similar to the ring situation. The proof given here uses the same line of
arguments as in [ASS].

Let $\mathcal{C}$ be a skeletally small pre additive category and $\mathcal{T%
}$ a tilting subcategory of $\mathrm{\mathrm{Mod}}(\mathcal{C})$. Let $%
\mathscr T$ be the torsion class of $\mathrm{\mathrm{Mod}}(\mathcal{C})$,
whose objects are epimorphic images of arbitrary sums of objects in $%
\mathcal{T}$. We proved $\mathscr T=\{M\in \mathrm{\mathrm{Mod}}(\mathcal{C}%
)|\mathrm{Ext}^{1}(\mathcal{T},M)=0\}$. We use this fact in the following:

\begin{lemma}
\label{GD1} Let $M$ be an object in $\mathscr{T}$ and assume $\mathrm{Ext}_{%
\mathcal{C}}^{1}(M,\;\;)_{\mathscr{T}}=0$. Then $M$ is in $\mathrm{Add}%
\mathcal{T}$.
\end{lemma}

\begin{proof}
Let $M$ be an object in $\mathscr{T}$. By Proposition \ref{TORZ} there is a
short exact sequence
\begin{equation*}
0\rightarrow \mathrm{Ker}(\alpha )\rightarrow \coprod_{i\in I}T_{i}%
\xrightarrow{\alpha}M\rightarrow 0\text{,}
\end{equation*}%
with $\mathrm{Ker}(\alpha )$ in $\mathscr{T}$. Therefore: the sequence
splits and $M$ is in $\mathrm{Add}\mathcal{T}$.
\end{proof}

\begin{proposition}
If $M$ is in $\mathscr{T}$, then $\mathrm{pdim}\mathrm{Hom}_{\mathcal{C}%
}(\;\;,M)_{\mathcal{T}}\leq \mathrm{pdim}M$.
\end{proposition}

\begin{proof}
By induction in $\mathrm{pdim}M$. If $\mathrm{pdim}M=0$, then $M$ is
projective, and since $M$ is in $\mathscr{T}$ there is an epimorphism $%
f:\coprod_{i\in I}T_{i}\rightarrow M\rightarrow 0$, with $T_{i}$ in $%
\mathcal{T}$, which splits and $M$ is a summand of $\coprod_{i\in I}T_{i}$.
It follows $M$ is in $\mathrm{Add}\mathcal{T}$, and $\mathrm{Hom}_{\mathcal{C%
}}(\;\;,M)_{\mathcal{T}}$ is summand of $\coprod_{i\in I} (\;\;,T_{i})$,
this is, $\mathrm{Hom}_{\mathcal{C}}(\;\;,M)_{\mathcal{T}}$ is projective
and $\mathrm{pdim}\mathrm{Hom}_{\mathcal{C}}(\;\;,M)=0$.

Let's assume $\mathrm{pdim}M=1$. There is an exact sequence
\begin{equation}  \label{GD2}
0\rightarrow L\rightarrow T_{0}\rightarrow M\rightarrow 0\text{,}
\end{equation}%
with $T_{0}$ in $\mathrm{Add}\mathcal{T}$ and $L$ in $\mathscr{T}$. The
sequence (\ref{GD2}) induces an exact sequence:
\begin{equation*}
0\rightarrow (\;\;,L)_{\mathcal{T}}\rightarrow (\;\;,T_{0})_{\mathcal{T}%
}\rightarrow (\;\;,M)_{\mathcal{T}}\rightarrow \mathrm{Ext}_{\mathcal{C}%
}^{1}(\;\;,L)_{\mathcal{T}}=0\text{.}
\end{equation*}

Since $\mathrm{pd}M=1$, then $\mathrm{Ext}_{\mathcal{C}}^{2}(M,\;\;)_{%
\mathscr{T}}=0$, and from (\ref{GD2}) and the long homology sequence, it
follows the sequence
\begin{equation*}
0=\mathrm{Ext}_{\mathcal{C}}^{1}(T_{0},\;\;)_{\mathscr{T}}\rightarrow
\mathrm{Ext}_{\mathcal{C}}^{1}(L,\;\;)_{\mathscr{T}}\rightarrow \mathrm{Ext}%
_{\mathcal{C}}^{2}(M,\;\;)_{\mathscr{T}}=0
\end{equation*}%
is exact. In consequence, $\mathrm{Ext}_{\mathcal{C}}^{1}(L,\;\;)_{\mathscr %
T}=0$, and by Lemma \ref{GD1}, $L$ is in $\mathrm{Add}\mathcal{T}$.
Therefore $\mathrm{Hom}_{\mathcal{C}}(\;\;,L)_{\mathcal{T}}$ is projective
and $\mathrm{Hom}_{\mathcal{C}}(\;\;,M)_{\mathcal{T}}$, has projective
dimension less or equal to one.

Suppose $n\geq 2$ and the claim is true for all objects in $\mathscr{T}$
with projective dimension less than $n$. Let $M$ be an object in $\mathscr{T}
$ with $\mathrm{pdim}M=n$. Then from (\ref{GD2}) and the long homology
sequence we get an exact sequence
\begin{equation*}
0=\mathrm{Ext}_{\mathcal{C}}^{n}(T_{0},\;\;)\rightarrow \mathrm{Ext}_{%
\mathcal{C}}^{n}(L,\;\;)\rightarrow \mathrm{Ext}_{\mathcal{C}%
}^{n+1}(M,\;\;)=0\text{.}
\end{equation*}%
Then $\mathrm{pdim}L\leq (n-1)$. By induction hypothesis $\mathrm{pd}\mathrm{%
Hom}_{C}(\;\;,L)_{\mathcal{T}}\leq n-1$ and $\mathrm{Hom}_{\mathcal{C}%
}(,\;\;T_{0})_{\mathcal{T}}$ is projective. Applying the contravariant
functor $\mathrm{Hom}_{\mathcal{T}}(-,\;)$ to the exact sequence (\ref{GD2}%
), we obtain by the long homology sequence the inequalities:
\begin{equation*}
\mathrm{pdim}\mathrm{Hom}_{\mathcal{C}}(\;\;,M)_{\mathcal{T}}\leq \mathrm{%
pdim}\mathrm{Hom}_{\mathcal{C}}(\;\;,L)_{\mathcal{T}}+1\leq 1+(n-1)
\end{equation*}
\end{proof}

\begin{theorem}
With the same assumptions as above, we have the inequality:
\begin{equation*}
\mathrm{gdim}(\mathcal{T})\leq 1+\mathrm{gdim}(\mathcal{C}).
\end{equation*}
\end{theorem}

\begin{proof}
Let $X$ be an object in $\mathrm{\mathrm{Mod}}(\mathcal{T})$, and cover it
with a projective object, to get an exact sequence
\begin{equation}  \label{GD4}
0\rightarrow Y\rightarrow \coprod_{i\in I}(\;\;,T_{i})\rightarrow
X\rightarrow 0\text{.}
\end{equation}

The functor $\phi :\mathscr{T}\rightarrow \mathscr{Y}$ is an equivalence,
and $\coprod_{i\in I}(\;\;,T_{i})$ is in $\mathscr{Y}$, and $\mathscr{Y}$ is
closed under sub-objects, hence $Y$ is in $\mathscr{Y}$. Since $\phi $ is
dense, there exists an object $M$ in $\mathscr{T}$, such that $\phi (M)=%
\mathrm{Hom}_{\mathcal{C}}(\;\;,M)_{\mathcal{T}}\cong Y$, and $\mathrm{pdim}%
Y\leq \mathrm{pdim}M$. From the exact sequence (\ref{GD4}) we have the
following inequalities:
\begin{equation*}
\mathrm{pdim}X\leq 1+\mathrm{pdim}Y\leq 1+\mathrm{pdim}M\leq 1+\mathrm{gdim}(%
\mathcal{C})
\end{equation*}%
and $\mathrm{gdim}(\mathcal{T})\leq 1+\mathrm{gdim}(\mathcal{C})$.
\end{proof}

\subsection{Brenner-Butler's theorem for categories of finitely presented
functors.}

In this subsection we will prove that, under mild assumptions on the
categories $\mathcal{C}$ and $\mathcal{T}$, Brenner-Butler's theorem holds
in the categories of finitely presented functors. To prove it we need to see
under which conditions the functor $\phi :\mathrm{\mathrm{Mod}}(\mathcal{C}%
)\rightarrow \mathrm{\mathrm{Mod}}(\mathcal{T})$ restricts to the categories
of finitely presented functors, $\phi :\mathrm{mod}(\mathcal{C})\rightarrow
\mathrm{mod}(\mathcal{T})$.

It was recalled in Section 1, that the category of finitely presented
functors $\mathrm{mod}(\mathcal{C})$ is abelian, if and only if, $\mathcal{C}
$ has pseudokerneles [AR]. Hence; it is natural to assume $\mathcal{C}$ and $%
\mathcal{T}$ have pseudokerneles. Under these conditions we have the
following.

\begin{proposition}
\label{R2} Let's assume $\mathcal{C}$ and $\mathcal{T}$ have pseudokerneles.
Then the functor
\begin{equation*}
\phi |\mathrm{mod}(\mathcal{C}):\mathrm{mod}(\mathcal{C})\rightarrow \mathrm{%
\mathrm{Mod}}(\mathcal{T})\text{,}
\end{equation*}%
has image in $\mathrm{mod}(\mathcal{T})$.
\end{proposition}

\begin{proof}
(a) For each object $\mathcal{C}$, the functors $(\;\;,\mathcal{C}(\;\;,C))_{%
\mathcal{T}}$ and $\mathrm{Ext}^{1}_{\mathcal{C}}(\;\;,\mathcal{C}(\;\;,C))_{%
\mathcal{T}}$ are in $\mathrm{mod}(\mathcal{T})$.

To see this, consider the exact sequence
\begin{equation*}
0\rightarrow \mathcal{C}(\;\;,C)\rightarrow T_{0}\rightarrow
T_{1}\rightarrow 0\text{,}
\end{equation*}%
with $T_{0},T_{1}$ in $\mathcal{T}$. From the above exact sequence, and the
long homology sequence, we get an exact sequence:
\begin{eqnarray*}
0 &\rightarrow &(\;\;,\mathcal{C}(\;\;,C))_{\mathcal{T}}\rightarrow
(\;\;,T_{0})_{\mathcal{T}}\rightarrow (\;\;,T_{1})_{\mathcal{T}}\rightarrow
\\
&\rightarrow &\mathrm{Ext}^{1}_{\mathcal{C}}(\;\;,\mathcal{C}(\;\;,C))_{%
\mathcal{T}}\rightarrow \mathrm{Ext}^{1}_{\mathcal{C}}(\;\;,T_{0})_{\mathcal{%
T}}=0\text{.}
\end{eqnarray*}

The claim follows from the fact $\mathrm{mod}(\mathcal{T})$ is abelian.

(b) Let $M$ be in $\mathrm{mod}(\mathcal{C})$. Since $\mathcal{C}$ has
pseudokernels, then $M$ has a projective resolution
\begin{equation*}
\cdots \rightarrow \mathcal{C}(\;\;,C_{3})\xrightarrow{(\;\;,f_2)}\mathcal{C}%
(\;\;,C_{2})\xrightarrow{(\;\;,f_1)}\mathcal{C}(\;\;,C_{1})%
\xrightarrow{(\;\;,f_0)}\mathcal{C}(\;\;,C_{0})\rightarrow M\rightarrow 0
\end{equation*}%
Let $K_{i}$ be $\mathrm{Im}(\;\;,f_{i})$. Then for all $i\geq 0,$ the
functors $\mathrm{Ext}^{1}_{\mathcal{C}}(\;\;,K_{i})_{\mathcal{T}}$ and $%
(\;\;,K_{i})_{\mathcal{T}}$ are finitely presented. Indeed, from the exact
sequences,
\begin{equation*}
0\rightarrow K_{i+1}\xrightarrow{k_{i+1}}\mathcal{C}(\;\;,C_{i})%
\xrightarrow{p_i}K_{i}\rightarrow 0\text{,}
\end{equation*}%
the long homology sequence, and the fact $\mathrm{pdim}\mathcal{T}\leq 1$,
we obtain for each $i$ an exact sequence:
\begin{eqnarray*}
0\rightarrow (\;\;,K_{i+1})_{\mathcal{T}}\xrightarrow{(\;\;,k_{i+1})}(\;\;,%
\mathcal{C}(\;\;,C_{i}))_{\mathcal{T}}\xrightarrow{(\;\;,p_i)}(\;\;,K_{i})_{%
\mathcal{T}}\xrightarrow{\partial_i}\mathrm{Ext}_{\mathcal{C}%
}^{1}(\;\;,K_{i+1})_{\mathcal{T}}\rightarrow && \\
\xrightarrow{\mathrm{Ext}^1(\;\;,k_{i+1})}\mathrm{Ext}^{1}(\;\;,\mathcal{C}%
(\;\;,C_{i}))_{\mathcal{T}}\xrightarrow{\mathrm{Ext}^1(\;\;,p_{i})}\mathrm{%
Ext}^{1}(\;\;,K_{i})_{\mathcal{T}}\rightarrow \mathrm{Ext}%
^{2}(\;\;,K_{i+1})_{\mathcal{T}}=0\text{.} &&
\end{eqnarray*}

By (a), for all $i\geq 0,$ the functor $\mathrm{Ext}^{1}_{\mathcal{C}}(\;\;,%
\mathcal{C}(\;\;,C_{i}))_{\mathcal{T}}$ is finitely presented. Hence; each $%
\mathrm{Ext}^{1}(\;\;,K_{i})_{\mathcal{T}}$ is finitely generated .

From the exact sequence
\begin{equation*}
\mathrm{Ext}_{\mathcal{C}}^{1}(\;\;,K_{i+1})_{\mathcal{T}}%
\xrightarrow{\mathrm{Ext}^1_{\mathcal{C}}(\;\;,k_{i+1})}\mathrm{Ext}^{1}_{%
\mathcal{C}}(\;\;,\mathcal{C}(\;\;,C_{i}))_{\mathcal{T}}\xrightarrow{%
\mathrm{Ext}^1_{\mathcal{C}}(\;\;,p_{i})}\mathrm{Ext}^{1}_{\mathcal{C}%
}(\;\;,K_{i})_{\mathcal{T}}\rightarrow 0\text{,}
\end{equation*}
it follows $\mathrm{Ext}^{1}_{\mathcal{C}}(\;\;,K_{i})_{\mathcal{T}}$ is
actually finitely presented. Since $\mathrm{mod}(\mathcal{T})$ is abelian,
the kernel of $\mathrm{Ext}_{\mathcal{C}}^{1}(\;\;,K_{i+1})_\mathcal{T}%
\xrightarrow{\mathrm{Ext}^1_{\mathcal{C}}(\;\;,k_{i+1})}\mathrm{Ext}^{1}_{%
\mathcal{C}}(\;\;,\mathcal{C}(\;\;,C_{i}))_\mathcal{\ T} $ is finitely
presented.

In a similar way, each $(\;\;,K_{i})_{\mathcal{T}}$ is finitely generated,
and it follows that the cokernel of the map $0\rightarrow (\;\;,K_{i+1})_%
\mathcal{T}\xrightarrow{(\;\;,k_{i+1})}(\;\;,\mathcal{C}(\;\;,C_{i}))_%
\mathcal{T}$ is finitely presented.

We have proved each $(\;\;,K_{i})_{\mathcal{T}}$ is an extension of two
finitely presented functors, therefore: it is finitely presented.

(c) From the exact sequence $0\rightarrow K_{0}\xrightarrow{k_0}\mathcal{C}%
(\;\;,C_{0})\rightarrow M\rightarrow 0$, and the long homology sequence, we
have an exact sequence
\begin{equation*}
0\rightarrow (\;\;,K_{0})_{\mathcal{T}}\rightarrow (\;\;,\mathcal{C}%
(\;\;,C_{0}))_{\mathcal{T}}\rightarrow (\;\;,M)_{\mathcal{T}}\rightarrow
\mathrm{Ext}^{1}_{\mathcal{C}}(\;\;,K_{0})_{\mathcal{T}}\rightarrow \mathrm{%
Ext}^{1}_{\mathcal{C}}(\;\;,\mathcal{C}(\;\;,C_{0})_{\mathcal{T}}\text{.}
\end{equation*}%
Using again $\mathrm{mod}(\mathcal{T})$ is abelian, it follows $\phi
(M)=(\;\;,M)_{\mathcal{T}}$ is finitely presented.
\end{proof}

\begin{corollary}
Assume $\mathcal{C}$ and $\mathcal{T}$ have pseudokernels. Then $\mathcal{T}$
is contravariantly finite in $\mathrm{mod}(\mathcal{C)}.$
\end{corollary}

\begin{proposition}
\label{BrennerB10} \bigskip Assume $\mathcal{C}$ and $\mathcal{T}$ have
pseudokernels. Then the following statements hold:

\begin{itemize}
\item[(i)] The functors $F,F^{\prime} ,G,G^{\prime} $ in Brenner-Buter's
theorem restrict to the subcategories of finitely presented functors.

\item[(ii)] Given a functor $M$ in $\mathscr T\cap \mathrm{mod}(\mathcal{C})$%
, there exists a resolution
\begin{equation*}
\rightarrow T_{n}\overset{t _{n}}{\rightarrow }\cdots\rightarrow T_{2}%
\overset{t_{2}}{\rightarrow }T_{1}\overset{t _{1}}{\rightarrow }T_{0}\overset%
{t_{0}}{\rightarrow }M\rightarrow 0
\end{equation*}

such that, each $T_{i}$ is in $\mathrm{add}\mathcal{T}$, and $T_{n}\overset{%
\delta _{n}}{\rightarrow }$ $\mathrm{Im}t _{n}$ is a $\mathcal{T}$%
-approximation of $\mathrm{Im}t _{n}.$

\item[(iii)] If $M$ is a functor in $\mathrm{mod}(\mathcal{C)}$, then the
the trace $\tau _{\mathcal{T}}(M)$ of $\mathcal{T}$ in $M$, and $M/\tau _{%
\mathcal{T}}(M)$ are finitely presented.

\item[(iv)] Denote by $t_{\mathscr{X}}$ the radical of the torsion theory $(%
\mathscr{X},\mathscr{Y})$ of $\mathrm{Mod}(\mathcal{T})$. Then for any
functor $N$ in $\mathrm{mod}(\mathcal{T}),$ $t_{\mathscr{X}}(N)$ and $N/t_{%
\mathscr{X}}(N)$ are finitely presented.

\item[(v)] For any pair of finitely presented functors $M,N$ in $\mathscr{T}$%
, we have an isomorphism $\mathrm{Ext}_{\mathcal{C}}^{i}(M,N)=\mathrm{Ext}_{%
\mathcal{T}}^{i}(\phi (M),\phi (N)).$
\end{itemize}
\end{proposition}

\begin{proof}
(i) We proved $F$ preserves finitely presented functors. If $%
(\;,T_{1})\rightarrow (\;,T_{0})\rightarrow N\rightarrow 0$ is a
presentation of $N\in \mathrm{mod}(\mathcal{T})$ , $T_{0},T_{1}\in \mathrm{%
add}\mathcal{T}$, then tensoring the exact sequence with $\mathcal{T}$ , we
obtain an exact sequence $T_{0}\rightarrow T_{1}\rightarrow \mathcal{T}%
\otimes N\rightarrow 0.$ Since $T_{0},T_{1}$ are finitely presented, and $%
\mathrm{mod}(\mathcal{C})$ is abelian, $G(N)=\mathcal{T}\otimes N$ is
finitely presented.

We left to the reader to prove that the functors $F^{\prime }$ and $%
G^{\prime}$, preserve finitely presented functors.

(ii) Let $M$ be in $\mathrm{mod}(\mathcal{C})$ and a map $\delta
:T^{0}\rightarrow M$ with $T^{0}$ in $\mathrm{Add}(\mathcal{T})$, as in
Proposition \ref{Catilt}, its image is $\tau _{\mathcal{T}}(M)$. If $%
f:T_{0}\rightarrow M$ is a $\mathcal{T}$-approximation, then $f$ factors
through $\delta $, and $\delta $ factors through $f$. In consequence, $%
\mathrm{Im}f=\tau _{\mathcal{T}}(M).$ In particular, if $M\in $ $\mathscr T$%
, then $f$ is an epimorphism.

Using again the fact $\mathrm{mod}(\mathcal{T})$ is abelian, the kernel of $%
f $, $K_{0}$, is finitely presented. From the exact sequence

\begin{equation*}
0\rightarrow K_{0}\rightarrow T_{0}\rightarrow M\rightarrow 0\text{,}
\end{equation*}
the long homology sequence, and the fact $f:T_{0}\rightarrow M$ is a $%
\mathcal{T}$-approximation, it follows $\mathrm{Ext}_\mathcal{C}%
^{1}(\;,K_{0})_{\mathcal{T}}=0$. Hence, $K_{0}$ is in $\mathscr T$, and the
claim follows by induction.

(iii) Let $M$ be in $\mathrm{mod}(\mathcal{C}).$ Then by the proof of (ii), $%
\tau_{\mathcal{T}}(M)$ finitely generated, and $M$ finitely presented,
implies $M/\tau _{\mathcal{T}}(M)$ is finitely presented, and $\mathrm{mod}(%
\mathcal{C})$ abelian, implies $\tau _{\mathcal{T}}(M)$ is finitely
presented.

(iv) Assume $N$ in $\mathrm{mod}(\mathcal{T})$. Since the functors: $%
F,F^{\prime} ,G,G^{\prime}$ preserve finitely presented functors, all terms
in the exact sequence:

\begin{equation*}
0\rightarrow \mathrm{Ext}_{\mathcal{C}}^{1}(-,\mathrm{Tor}_{1}^{\mathcal{T}%
}(N,\mathcal{T}))\xrightarrow{\eta}N\xrightarrow {\gamma}\phi (N\otimes
\mathcal{T})\rightarrow 0
\end{equation*}%
are finitely presented. The claim follows by observing the isomorphisms: $t_{%
\mathscr{X}}(N)\cong $ $\mathrm{Ext}_{\mathcal{C}}^{1}(-,\mathrm{Tor}_{1}^{%
\mathcal{T}}(N,\mathcal{T}))$ and $N/$ $t_{\mathscr{X}}(N)\cong \phi
(N\otimes \mathcal{T})$.

(v) Follows as in Proposition \ref{TORZ}.
\end{proof}

We denote by $(\tilde{\mathscr{T}},\tilde{\mathscr{F}})$ and $(\tilde{%
\mathscr{X}},\tilde{\mathscr{Y}})$ the intersection of the torsion theories $%
(\mathscr T,\mathscr F)$ and ($\mathscr{X},\mathscr{Y})$ with the categories
of finitely presented functors, $\mathrm{mod}$($\mathcal{C})$ and $\mathrm{%
mod}(\mathcal{T)}$, respectively. From the previous proposition we obtain
the following:

\begin{theorem}[Brenner-Butler]
\label{bbre1} Let $\mathcal{T}$ be a tilting subcategory of $\mathrm{mod}(%
\mathcal{C})$ and assume $\mathcal{C}$ and $\mathcal{T}$ have pseudokernls.
With the above notation the following statements hold:

\begin{itemize}
\item[(i)] The functors $F$ and $G$ induce an equivalence between $\tilde{%
\mathscr{T}}$ and $\tilde{\mathscr{Y}}.$

\item[(ii)] The functors $F^{\prime}$ and $G^{\prime}$ induce an equivalence
between $\tilde{\mathscr{F}}$ and $\tilde{\mathscr{X}}.$

\item[(iii)] We also have: $FG^{\prime}=F^{\prime}G=0$ and $%
G^{\prime}F=GF^{\prime}=0$.
\end{itemize}
\end{theorem}

\begin{proof}
The proof is clear from (i) in the above proposition.
\end{proof}

\subsection{Classical tilting for dualizing varieties.}

In order to have a complete analogy with tilting theory for finite
dimensional algebras, we need to add more restrictions in our categories, in
particular, we need the existence of duality. We will assume in this
subsection that $\mathcal{C}$ and $\mathcal{T}$ are dualizing varieties.

It was proved above that the category $\theta =\{\theta _{C}\}_{C\in
\mathcal{C}}$, where $\theta _{C}=(\mathcal{C}(\;\;,C),\;\;)_{\mathcal{T}}$,
is a tilting subcategory of $\mathrm{mod}(\mathcal{T}^{op}).$ Then by
Brenner-Butler's theorem, there are torsion pairs $(\mathscr{T}(\theta ),%
\mathscr{F}(\theta ))$ and $(\mathscr{X}(\theta ),\mathscr{Y}(\theta ))$ in $%
\mathrm{mod}(\mathcal{T}^{op})$ and $\mathrm{mod}(\theta )$, respectively,
and equivalence of categories

\begin{equation*}
\begin{diagram} \node{\mathscr{T}(\theta)}\arrow[3]{se,b}{\phi_{\theta}}
\node{} \node{} \node{} \node{} \node{} \node{}
\node{\mathscr{F}(\theta)}\arrow[3]{sw,b}{\mathrm{Ext}^1_{\mathcal{T}^{op}}(%
\;\;,-)_\theta}\\ \\ \\ \\ \\ \\ \\
\node{\mathscr{X}(\theta)}\arrow[3]{ne,t}{\mathrm{Tor}^\theta_1(\;\;,%
\theta)} \node{} \node{} \node{} \node{} \node{} \node{}
\node{\mathscr{Y}(\theta)}\arrow[3]{nw,t}{\otimes\theta} \end{diagram}
\end{equation*}

By Proposition \ref{TT}, there is an equivalence of categories $\alpha (C):%
\mathcal{C}^{op}\rightarrow \theta $, $\alpha (C)=\theta _{C}=(\mathcal{C}%
(\;\;,C),\;\;)_{\mathcal{T}}$ , which induces an equivalence $\alpha \ast :%
\mathrm{mod}(\theta )\rightarrow \mathrm{mod}(\mathcal{C}^{op})$ given by:
\begin{equation*}
\alpha \ast (H)(C)=H(\alpha (C))=H(\theta _{C})=H((\mathcal{C}(\;\;,C),\;\;)|%
\mathcal{T}),
\end{equation*}%
for each $C$ in $\mathcal{C}$, and $H$ in $\mathrm{mod}(\theta )$.

\begin{lemma}
\label{BrennerB11} Let $N$ be an object in $\mathrm{mod}(\mathcal{T})$ and $%
C $ one object in $\mathcal{C}$. Then the following statements hold:

\begin{itemize}
\item[(a)] There is an isomorphism $((\mathcal{C}(\;\;,C),\;\;)_{\mathcal{T}%
},DN)\cong D(N\otimes \mathcal{T})(C),$ such that the following square
\begin{equation*}
\begin{diagram} \node{\mathrm{mod}(\mathcal{C})}\arrow{s,l}{D}
\node{\mathrm{mod}(\mathcal{T})}\arrow{w,t}{\otimes\mathcal{T}}%
\arrow{s,l}{D}\\ \node{\mathrm{mod}(\mathcal{C}^{op})}
\node{\mathrm{mod}(\mathcal{T}^{op})}\arrow{w,t}{\alpha*\phi_{\theta}}
\end{diagram}
\end{equation*}
commutes.

\item[(b)] There is an isomorphism $\mathrm{Ext}_{\mathcal{T}^{op}}^{1}((%
\mathcal{C}(\;\;,C),\;\;)_{\mathcal{T}},DN)\cong D(\mathrm{Tor}_{1}^{%
\mathcal{T}}(N,\mathcal{T}))(C),$ such that the following square
\begin{equation*}
\begin{diagram} \node{\mathrm{mod}(\mathcal{C})}\arrow{s,l}{D} \node{}
\node{\mathrm{mod}(\mathcal{T})}\arrow[2]{w,t}{\mathrm{Tor}^\mathcal{T}_1(\;%
\;,\mathcal{T})}\arrow{s,l}{D}\\ \node{\mathrm{mod}(\mathcal{C}^{op})}
\node{}
\node{\mathrm{mod}(\mathcal{T}^{op})}\arrow[2]{w,t}{\alpha*\mathrm{Ext}^1(\;%
\;,-)|\theta} \end{diagram}
\end{equation*}

commutes. \newline

\item[(c)] We have the following equivalences of categories:

\begin{itemize}
\item[(i)] $D(\mathscr{X}(\mathcal{T}))\cong \mathscr{F}(\theta )$, $D(%
\mathscr{Y}(\mathcal{T}))\cong \mathscr{T}(\theta )$

\item[(ii)] $D(\mathscr{T}(\mathcal{T}))\cong \mathscr{Y}(\theta )$, $D(%
\mathscr{F}(\mathcal{T}))\cong \mathscr{X}(\theta )$
\end{itemize}
\end{itemize}
\end{lemma}

\begin{proof}
Applying the functor $(\;\;,DN)$ to the exact sequence
\begin{equation}
0\rightarrow (T^{1},\;\;)\rightarrow (T^{0},\;\;)\rightarrow (\mathcal{C}%
(\;\;,C),\;\;)_{\mathcal{T}}\rightarrow 0\text{.}  \label{restor1}
\end{equation}%
We obtain by the long homology sequence, and Yoneda's Lemma, the following
exact sequence:
\begin{equation}
0\rightarrow ((\;\;,C),\;\;)_{\mathcal{T}},DN)\rightarrow
DN(T^{0})\rightarrow DN(T^{1})\rightarrow \mathrm{Ext}^{1}((\;\;,C),\;\;)_{%
\mathcal{T}},DN)\rightarrow 0\text{,}  \label{restor2}
\end{equation}%
and applying $\otimes N$ to (\ref{restor1}), by the long homology sequence,
we get an exact sequence
\begin{equation}
0\rightarrow \mathrm{Tor}_{1}^{\mathcal{T}}(N,\mathcal{T})(C)\rightarrow
N(T^{1})\rightarrow N(T^{0})\rightarrow (N\otimes \mathcal{T})(C)\rightarrow
0\text{.}  \label{restor3}
\end{equation}%
Dualizing (\ref{restor3}), and comparing it with (\ref{restor2}), we obtain
the isomorphisms in (a) and (b).

To see that the first square commutes, let $N$ be in $\mathrm{mod}(\mathcal{T%
})$. Then there are equalities:
\begin{eqnarray*}
\alpha \ast (\phi _{\theta }(DN))(C) &=&\alpha \ast ((\;\;,DN)|\theta )(C) \\
&=&(((\;\;,C),\;\;),DN)=D(N\otimes \mathcal{T})(C)\text{.}
\end{eqnarray*}

The equalities

\begin{eqnarray*}
\alpha \ast (\mathrm{Ext}^{1}(\;\;,-)|\theta (DN))(C) &=&\alpha \ast (%
\mathrm{Ext}^{1}(\;\;,DN)|\theta )(C) \\
&=&\mathrm{Ext}_{1}^{\mathcal{T}}(((\;\;,C),\;\;),DN)=D(\mathrm{Tor}_{1}^{%
\mathcal{T}}(N\otimes \mathcal{T}))(C)\text{,}
\end{eqnarray*}%
imply, the second square commutes.

It only remains to prove (c). By (a) it follows
\begin{equation*}
\mathscr{T}(\theta )=\{N\in \mathrm{mod}(\mathcal{T}^{op})|\mathrm{Ext}%
^{1}(\theta _{C},N)=0\}=D(\mathscr{Y}(\mathcal{T})).
\end{equation*}%
By Brenner-Butler's Theorem, there are equivalences of categories:
\begin{eqnarray*}
\phi _{\theta } &:&\mathscr F(\theta )\rightarrow \mathscr{X}(\theta ) \\
-\otimes \mathcal{T} &:&\mathscr Y(\mathcal{T})\rightarrow \mathscr{T}(%
\mathcal{T})\text{.}
\end{eqnarray*}

Then we have a commutative square%
\begin{equation*}
\begin{diagram} \node{\mathscr T(\mathcal{T})}\arrow{s,l}{D}
\node{\mathscr{Y}(\mathcal{T})}\arrow{w,t}{\otimes\mathcal{T}}\arrow{s,l}{D}%
\\ \node{\mathscr{Y}(\theta)}
\node{\mathscr{T}(\theta)}\arrow{w,t}{\alpha*\phi_{\theta}} \end{diagram}
\end{equation*}%
By part (b), it follows
\begin{equation*}
\mathscr{F}(\theta )=\{N\in \mathrm{mod}(\mathcal{T}^{op})|\mathrm{Hom}%
(\theta _{C},N)=0\}=D(\mathscr{X}(\mathcal{T}))
\end{equation*}

From the equivalence of categories given in Brenner-Butler's Theorem:
\begin{eqnarray*}
\mathrm{Tor}_{1}^{\mathcal{T}}(\;\;,\mathcal{T}) &:&\mathscr X(\mathcal{T}%
)\rightarrow \mathscr{F}(\mathcal{T}) \\
\mathrm{Ext}^{1}_{\mathcal{T}^{op}}(\;\;,-)_{\theta } &:&\mathscr T(\theta
)\rightarrow \mathscr{Y}(\theta )
\end{eqnarray*}

We have a commutative square:
\begin{equation*}
\begin{diagram} \node{\mathscr F(\mathcal{T})}\arrow{s,l}{D}
\node{\mathscr{X}(\mathcal{T})}\arrow{w,t}{\mathrm{Tor}_1^\mathcal{T}(\;\;,%
\mathcal{T})}\arrow{s,l}{D}\\ \node{\mathscr{X}(\theta)}
\node{\mathscr{F}(\theta)}\arrow{w,t}{\alpha* \mathrm{Ext}^1(\;\;,-)|\theta}
\end{diagram}
\end{equation*}
\end{proof}

\begin{definition}
\emph{[ASS]} Let $\mathcal{C}$ be Krull-Schmidt. A torsion theory $(\mathscr %
T,\mathscr F) $ in $\mathrm{mod}(\mathcal{C})$ \textbf{splits,} if every
indecomposable $M\in \mathrm{mod}(\mathcal{C})$ is, either in $\mathscr T$
or in $\mathscr F $.
\end{definition}

\begin{proposition}[ASS Prop. 1.7]
\label{Brenner-B} Let $\mathcal{C}$ be dualizing category and $(\mathscr T,%
\mathscr F)$ a torsion pair in $\mathrm{mod}(\mathcal{C})$. Then the
following conditions are equivalent:

\begin{itemize}
\item[(a)] The torsion theory $(\mathscr T,\mathscr F)$ splits.

\item[(b)] Let $\tau$ be the radical of the torsion theory. Then for any $M$
en $\mathrm{mod}(\mathcal{C})$, the exact sequence: 0$\rightarrow
\tau(M)\rightarrow M\rightarrow M/\tau(M)\rightarrow 0$, splits.

\item[(c)] For any $N$ in $\mathscr F$ and any $M$ in $\mathscr T$, $\mathrm{%
Ext}^{1}_{\mathcal{C}}(N,M)=0$.

\item[(c)] If $M\in \mathscr T$, then $\mathrm{TrD}M\in \mathscr T$.

\item[(d)] If $N\in \mathscr F$, then $\mathrm{DTr}N\in \mathscr F$.
\end{itemize}
\end{proposition}

We say that a tilting subcategory $\mathcal{T}$ of $\mathrm{mod}(\mathcal{C}%
) $ \textbf{separates,} if the torsion theory $(\mathscr{T}(\mathcal{T}),%
\mathscr{F}(\mathcal{T}))$ in $\mathrm{mod}(\mathcal{C})$ splits, and we say
it \textbf{splits,} if the torsion theory $(\mathscr{X}(\mathcal{T}),%
\mathscr{Y}(\mathcal{T}))$ in $\mathrm{mod}(\mathcal{T})$ splits.

\begin{lemma}
\label{BrennerB9} Let $\mathcal{T}$ be a tilting subcategory of $\mathrm{mod}%
(\mathcal{C})$ that splits. Then the following statements hold:

\begin{itemize}
\item[(a)] Let $0\rightarrow M\xrightarrow{f}E\xrightarrow{g}\mathrm{TrD}%
(M)\rightarrow 0$ be an almost split sequence, with $M$ in $\mathscr{T}(%
\mathcal{T})$. Then the three terms are in $\mathscr{T}(\mathcal{T})$, and $%
0\rightarrow \phi (M)\rightarrow \phi (E)\rightarrow \phi (\mathrm{TrD}%
(M))\rightarrow 0$ is an almost split sequence, whose terms are in $%
\mathscr{Y}(\mathcal{T}).$

\item[(b)] Let $0\rightarrow \mathrm{DTr}(M)\rightarrow E\rightarrow
M\rightarrow 0$ be an almost split sequence, with $M$ in $\mathscr{F}(%
\mathcal{T})$. Then the three terms are in $\mathscr{F}(\mathcal{T})$, and $%
0\rightarrow \mathrm{Ext}^1_\mathcal{C}(\;\;,\mathrm{Dtr}(M))_{\mathcal{T}%
}\rightarrow \mathrm{Ext}^1_\mathcal{C}(\;\;,E)_{\mathcal{T}}\rightarrow
\mathrm{Ext}^1_\mathcal{C}(\;\;,M)_{\mathcal{T}}\rightarrow 0$ is an almost
split sequence, whose terms are in $\mathscr{X}(\mathcal{T})$.
\end{itemize}
\end{lemma}

\begin{proof}
We will prove only (a), being (b) similar. By previous Lemma, $\mathrm{TrD}%
(M)\in \mathscr{T}(\mathcal{T})$, therefore: $E$ $\in \mathscr{T}(\mathcal{T}%
)$. By the long homology sequence, we get the exact sequence
\begin{equation}
0\rightarrow (\;\;,M)_{\mathcal{T}}\xrightarrow{(\;\;,f)}(\;\;,E)_{\mathcal{T%
}}\xrightarrow{(\;\;,g)}(\;\;,\mathrm{TrD}(M))_{\mathcal{T}}\rightarrow
\mathrm{Ext}^1_\mathcal{C}(\;\;,M)_{\mathcal{T}}=0\text{,}
\end{equation}%
whose terms are in $\mathscr{Y}(\mathcal{T})$.

The morphism $(\;\;,g)_{\mathcal{T}}$ is right minimal. Let $\overline{\eta }%
:(\;\;,E)_{\mathcal{T}}\rightarrow (\;\;,E)_{\mathcal{T}}$ be an
endomorphism, such that $(\;\;,g)_{\mathcal{T}}\tilde{\eta}=(\;\;,g)_{%
\mathcal{T}}$. Since $\phi:\mathscr{T}(\mathcal{T})\rightarrow \mathscr{Y}(%
\mathcal{T})$ is an equivalence of categories we have $((\;\;,E)_{\mathcal{T}%
},(\;\;,E)_{\mathcal{T}})\cong (E,E)$ and $\tilde{\eta}=(\;\;,\eta )_{%
\mathcal{T}}$, with $\eta :E\rightarrow E$. Then $g\eta =g$, and since $g$
is right minimal, it follows $\eta $ is an isomorphism, hence, $\tilde{\eta}$
is an isomorphism.

Let $N$ be indecomposable and $\tilde{\gamma}:N\rightarrow (\;\;,\mathrm{TrD}%
(M))$ a non isomorphism and non zero map. Then, either $N\in \mathscr{X}(%
\mathcal{T})$ or $N\in \mathscr{Y}(\mathcal{T})$. If $N\in \mathscr X(%
\mathcal{T})$, then $\tilde{\gamma}\in \mathrm{Hom}(\mathscr X(\mathcal{T}),%
\mathscr Y(\mathcal{T}))=0$, hence, $N=(\;\;,H)_\mathcal{T}$, with $H\in %
\mathscr T(\mathcal{T})$ and $\tilde{\gamma}=(\;\;,\gamma )_\mathcal{T}%
:(\;\;,H)_\mathcal{T}\rightarrow (\;\;,\mathrm{TrD}M)_\mathcal{T}$, where $%
\gamma $ is non isomorphism and non zero. Then $\eta $ factors through $E$,
and $\tilde{\eta}$ factors through $(\;\;,E)_\mathcal{T}$.
\end{proof}

As a consequence of Lemma \ref{BrennerB9}, and Proposition \ref{BrennerB10},
we have:

\begin{lemma}[ ASS Lemma 5.5]
Let $\mathcal{T}$ be a tilting subcategory of $\mathrm{mod}(\mathcal{C})$.
If $M\in \mathscr T(\mathcal{T})$ and $N\in \mathscr F(\mathcal{T})$, then
for any $j\geq 1$, there is an isomorphism
\begin{equation*}
\mathrm{Ext}_{\mathcal{C}}^{j}(M,N)\cong \mathrm{Ext}_{\mathcal{T}%
}^{j-1}(\phi (M),\mathrm{Ext}^{1}_{\mathcal{C}}(\;\;,N)_\mathcal{T})\text{.}
\end{equation*}
\end{lemma}

From the fact $\theta =\{\theta _{C}=((\;\;,C),\;\;)_\mathcal{T}\}_{C\in
\mathcal{C}}$ is a tilting subcategory in $\mathrm{mod}(\mathcal{T}^{op})$,
we obtain the following:

\begin{theorem}
Let $\mathcal{T}$ be a tilting subcategory of $\mathrm{mod}(\mathcal{C}).$%
Then:

\begin{itemize}
\item[(a)] $\mathcal{T}$ separates, if and only if, for any $Y\in \mathscr{Y}%
(\mathcal{T})$, $\mathrm{pdim}Y=1$.

\item[(b)] $\mathcal{T}$ splits, if and only if, for any $N\in \mathscr{F}(%
\mathcal{T})$, $\mathrm{idim}N=1$.
\end{itemize}
\end{theorem}

\begin{proof}
(b) Is proved in [ASS Theo. 5.6], but, for he benefit of the reader, we
repeat the proof here.

First, the sufficiency of the condition. Assume that for every $N\in %
\mathscr{F}(\mathcal{T})$, we have $\mathrm{idim}N=1$. Let $X\in \mathscr{X}(%
\mathcal{T})$ and $Y\in \mathscr{Y}(\mathcal{T})$. Then there exist $M\in %
\mathscr{T}(\mathcal{T})$ and $N\in \mathscr{F}(\mathcal{T})$ such that $%
X\cong \mathrm{Ext}_{\mathcal{C}}^{1}(\;,N)_{\mathcal{T}}$ and $Y\cong \phi
(M)$, by Brenner-Buttler theorem. Hence, by the above Lemma
\begin{equation*}
\mathrm{Ext}_{\mathcal{T}}^{1}(Y,X)\cong \mathrm{Ext}_{\mathcal{T}}^{1}(\phi
(M),\mathrm{Ext}_{\mathcal{C}}^{1}(\;,N)_{\mathcal{T}})\cong \mathrm{Ext}_{%
\mathcal{C}}^{2}(M,N)=0,
\end{equation*}
Conversely, assume that $(\mathscr{X}(\mathcal{T}),\mathscr{Y}(\mathcal{T}))$
is spplitting and let $N\in \mathscr{F}(\mathcal{T})$. Take an injective
resolution of $N$
\begin{equation*}
0\rightarrow N\xrightarrow{d^0}I^{0}\xrightarrow{d^1}I^{1}\xrightarrow{d^2}%
I^{2}\rightarrow \cdots
\end{equation*}

Let $L^{0}=\mathrm{Im}d^{1}$ and $L^{1}=\mathrm{Im}d^{2}$. Since $\mathscr{T}%
(\mathcal{T})=\mathrm{Ker}\mathrm{Ext}_{\mathcal{C}}^{1}(\mathcal{T},\;)$
contais the injective objects and it is closed under epimorphic images, and $%
N\in \mathscr{F}(\mathcal{T})$, it follows $L^{1}\in \mathscr{T}(\mathcal{T}%
) $. Then, by the above Lemma, we have:

\begin{equation*}
\mathrm{Ext}_{\mathcal{C}}^{1}(L^{1},L^{0})\cong \mathrm{Ext}_{\mathcal{C}%
}^{2}(L^{1},N)\cong \mathrm{Ext}_{\mathcal{T}}^{1}(\phi (L^{1}),\mathrm{Ext}%
_{\mathcal{C}}^{1}(\;,N)_{\mathcal{T}})
\end{equation*}%
But, $\phi (L^{1})\in \mathscr{Y}(\mathcal{T})$, $\mathrm{Ext}_{\mathcal{C}%
}^{1}(\;,N)_{\mathcal{T}}\in \mathscr{X}(\mathcal{T})$ and $(\mathscr{X}(%
\mathcal{T}),\mathscr{Y}(\mathcal{T}))$ is splitting, by hypothesis. By
Proposition \ref{Brenner-B}, this implies $\mathrm{Ext}_{\mathcal{T}%
}^{1}(\phi (L^{1}),\mathrm{Ext}_{\mathcal{C}}^{1}(\;,N)_{\mathcal{T}})=0$,
proving that the exact sequence $0\rightarrow L^{0}\rightarrow
I^{1}\rightarrow L^{1}\rightarrow 0$ splits . Therefore, $L^{0}$ is
injective and $\mathrm{idim}N\leq 1$. Finally, $N\in \mathscr{F}(\mathcal{T}%
) $, implies $N$ is not injective, thus $\mathrm{idim}N=1$.

We prove that (a) follows from (b):

By definition, $\mathcal{T}$ separates, if and only if, the torsion pair $(%
\mathscr{T}(\mathcal{T}),\mathscr{F}(\mathcal{T}))$ in $\mathrm{mod}(%
\mathcal{C})$ splits. By duality, this occurs if and only if, the pair $(%
\mathscr{X}(\theta ),\mathscr{Y}(\theta ))=(D\mathscr{F}(\mathcal{T}),D%
\mathscr{T}(\mathcal{T}))$, is a splitting torsion theory in $\mathrm{mod}(%
\mathcal{C}^{op})\cong \mathrm{mod}(\theta )$. By (b), $\theta $ is a
tilting subcategory that splits in $\mathrm{mod}(\mathcal{T}^{op})$, if and
only if for all $N\in \mathscr F(\theta )$, $\mathrm{idim}N=1$. But $%
\mathscr{F}(\theta )=D\mathscr{Y}(\mathcal{T})$, and (a) follows.
\end{proof}

\begin{corollary}
If $\mathrm{gdim}\mathcal{C}\leq 1$, then any tilting subcategory $\mathcal{T%
}\subset \mathrm{mod}(\mathcal{C})$ splits.
\end{corollary}

As a consequence of the previous theorem and Lemma \ref{BrennerB11}, we have
the following:

\begin{corollary}
If $\mathcal{C}$ and $\mathcal{T}$ are hereditary, then $\mathcal{T}$ splits
and separates
\end{corollary}

The last theorems in this subsection will have important implications in the
hereditary case, that we will consider in next section.

\section{Infinite quivers}

We begin this section showing that there exist natural examples of the
notions discussed in the previous section.

We prove first that for a finite dimensional algebra $\Lambda $, having a
finitely generated generator $M$, such that there exists a tilting $\mathrm{%
End}_{\Lambda }(M)^{op}$-module $T$, the module $T$ can be extended to a
tilting category of $\mathrm{mod}_{\Lambda }.$

We study next locally finite infinite quivers $Q$, and consider the quiver
algebra 
-category. We prove that for locally finite infinite quivers, a section on
the preprojective component produces a tilting category. To apply the
tilting functor is analogous to apply an infinite sequence of partial
Coxeter functors, to change the orientation of the quiver. We next use these
results to computate the Auslander-Reiten components of the locally finite
infinite Dynkin quivers.

To describe the shape of all Auslander-Reiten components, it is enough to
compute them for a fixed orientation, and then to apply tilting. We choose
quivers with only sinks and sources.

A simple modification of the arguments given by [ABPRS], proves that regular
Auslander-Reiten components of locally finite infinite quivers are of type $%
A_{\infty }$.

\subsection{Extending tilting}

In this subsection we will prove that there exist examples of tilting
subcategories. The first source of examples is produced in the following way:

Let $\mathcal{C}$ be a $K$-category, $K$ a field, which is $\mathrm{Hom}$%
-finite, with an additive generator $\varLambda$, such that there is a
tilting $\mathrm{End}(\varLambda)^{op}$-module $T$, $\mathcal{C}^{\prime }=\{%
\mathrm{add}\varLambda\}$ is a subcategory of $\mathcal{C}$ containing the
generator of $\mathcal{C}$, $\mathrm{\mathrm{Mod}}(\mathcal{C}^{\prime })$
is equivalent to $\mathrm{\mathrm{Mod}}(\mathrm{End}(\varLambda)^{op})$ and
it has tilting subcategory which corresponds to $T$ under the equivalence.
We will extend first the tilting subcategory of $\mathrm{\mathrm{Mod}}(%
\mathcal{C}^{\prime })$ to a partial tilting subcategory of $\mathrm{\mathrm{%
Mod}}(\mathcal{C})$ and then using Bongartz' argument [B], [CF] we will
extend it to a tilting subcategory of $\mathrm{\mathrm{Mod}}(\mathcal{C})$.

\begin{lemma}
\label{ET1} Let $\mathcal{C}^{\prime }$ be a subcategory of $\mathcal{C}$.
If $\mathcal{T}$ is a self orthogonal subcategory of $\mathrm{\mathrm{Mod}}(%
\mathcal{C}^{\prime })$, consisting of finitely presented objects, then $%
\mathcal{C}\otimes _{\mathcal{C}^{\prime }}\mathcal{T}$ is a self orthogonal
subcategory of $\mathrm{\mathrm{Mod}}(\mathcal{C})$.
\end{lemma}

\begin{proof}
Let $T_{1},T_{2}$ be objects of $\mathcal{T}$, with presentations $\mathcal{C%
}^{\prime }(\;\;,C_{1}^{\prime })\rightarrow \mathcal{C}^{\prime
}(\;\;,C_{0}^{\prime })\rightarrow T_{1}\rightarrow 0$, $\mathcal{C}^{\prime
}(\;\;,C_{1}^{\prime \prime })\rightarrow \mathcal{C}^{\prime
}(\;\;,C_{0}^{\prime \prime })\rightarrow T_{2}\rightarrow 0$. Let's
consider the functors: $\mathrm{res}:\mathrm{\mathrm{Mod}}(\mathcal{C}%
)\rightarrow \mathrm{\mathrm{Mod}}(\mathcal{C}^{\prime })$, $\mathcal{C}%
\otimes _{\mathcal{C}^{\prime }}:\mathrm{\mathrm{Mod}}(\mathcal{C}^{\prime
})\rightarrow \mathrm{\mathrm{Mod}}(\mathcal{C})$, as in [A].

We claim $\mathrm{Ext}_{\mathcal{C}}^{1}(\mathcal{C}\otimes _{\mathcal{C}%
^{\prime }}T_{1},\mathcal{C}\otimes _{\mathcal{C}^{\prime }}T_{2})=0$.
Indeed, let $\mathbf{e}$ be an element of $\mathrm{Ext}_{\mathcal{C}}^{1}(%
\mathcal{C}\otimes _{\mathcal{C}^{\prime }}T_{1},\mathcal{C}\otimes _{%
\mathcal{C}^{\prime }}T_{2})$, $\mathbf{e}:0\rightarrow \mathcal{C}\otimes _{%
\mathcal{C}^{\prime }}T_{2}\rightarrow F\rightarrow \mathcal{C}\otimes _{%
\mathcal{C}^{\prime }}T_{1}\rightarrow 0$. We apply $\mathcal{C}\otimes _{%
\mathcal{C}^{\prime }}$ to the projective presentations of $T_{1}$ and $%
T_{2} $, to get projective presentations of the corresponding extension. By
the Horse Shoe Lemma, we obtain the following commutative exact diagram:
\begin{equation*}
\begin{diagram} \dgARROWLENGTH=.1em \node{0}\arrow{s}{}\arrow{e,!}{}
\node{0}\arrow{s}{}\arrow{e,!}{} \node{0}\arrow{s}{}\arrow{e,!}{} \node{}\\
\node{\mathcal{C}(\;\;,C'_1)}\arrow{s}{}\arrow{e}{}
\node{\mathcal{C}(\;\;,C'_0)}\arrow{s}{}\arrow{e,}{}
\node{\mathcal{C}\otimes_{\mathcal{C}'} T_1}\arrow{s}{}\arrow{e}{}
\node{0}\\ \node{\mathcal{C}(\;\;,C'_1\coprod C''_1)}\arrow{s}{}\arrow{e}{}
\node{\mathcal{C}(\;\;,C'_0\coprod C''_0)}\arrow{s}{}\arrow{e,}{}
\node{F}\arrow{s}{}\arrow{e}{} \node{0}\\
\node{\mathcal{C}(\;\;,C''_1)}\arrow{s}{}\arrow{e}{}
\node{\mathcal{C}(\;\;,C''_0)}\arrow{s}{}\arrow{e,}{}
\node{\mathcal{C}\otimes_{\mathcal{C}'} T_2}\arrow{s}{}\arrow{e}{}
\node{0}\\ \node{0}\arrow{e,!}{} \node{0}\arrow{e,!}{} \node{0}\arrow{e,!}{}
\node{} \end{diagram}
\end{equation*}%
By [A Prop. 3.2], there exists an isomorphism $\mathcal{C}\otimes _{\mathcal{%
C}^{\prime }}\mathrm{res}(F)\cong \mathrm{Id}_{\mathrm{\mathrm{Mod}}(%
\mathcal{C})}(F)=F$.

Since the functor $\mathrm{res}$ is exact, we have the following commutative
exact diagram:
\begin{equation*}
\begin{diagram}\dgARROWLENGTH=.2em \node{0}\arrow{s}{}\arrow{e,!}{}
\node{0}\arrow{s}{}\arrow{e,!}{} \node{0}\arrow{s}{}\arrow{e,!}{} \node{}\\
\node{\mathcal{C}'(\;\;,C'_1)}\arrow{s}{}\arrow{e}{}
\node{\mathcal{C}'(\;\;,C'_0)}\arrow{s}{}\arrow{e,}{}
\node{T_1}\arrow{s}{}\arrow{e}{} \node{0}\\
\node{\mathcal{C}'(\;\;,C'_1\coprod C''_1)}\arrow{s}{}\arrow{e}{}
\node{\mathcal{C}'(\;\;,C'_0\coprod C''_0)}\arrow{s}{}\arrow{e,}{}
\node{\mathrm{res}F}\arrow{s}{}\arrow{e}{} \node{0}\\
\node{\mathcal{C}'(\;\;,C''_1)}\arrow{s}{}\arrow{e}{}
\node{\mathcal{C}'(\;\;,C''_0)}\arrow{s}{}\arrow{e,}{}
\node{T_2}\arrow{s}{}\arrow{e}{} \node{0}\\ \node{0}\arrow{e,!}{}
\node{0}\arrow{e,!}{} \node{0}\arrow{e,!}{} \node{} \end{diagram}
\end{equation*}

By assumption, the column at the right of the diagram splits, hence
tensoring with $\mathcal{C}\otimes _{\mathcal{C}^{\prime }}$ and using $%
\mathcal{C}\otimes _{\mathcal{C}^{\prime }}\mathrm{res}(F)\cong F$, the
columns in the first diagram split, hence $\mathbf{e}$ splits.
\end{proof}

\begin{definition}
A full subcategory $\mathcal{T}$ of $\mathrm{\mathrm{Mod}}(C)$ is a \textbf{%
partial tilting}, if its objects satisfy the following two conditions:

\begin{itemize}
\item[(i)] For each object $T$ in $\mathcal{T}$ there is an exact sequence $%
0\rightarrow P_{1}\rightarrow P_{0}\rightarrow T\rightarrow 0$, with $P_{i}$
finitely generated projective.

\item[(ii)] For every pair of objects $T_{i}$ and $T_{j}$ in $\mathcal{T}$, $%
\mathrm{Ext}_{\mathcal{C}}^{1}(T_{i},T_{j})=0$.
\end{itemize}
\end{definition}

\begin{proposition}
\label{ET2} Let $\mathcal{C}^{\prime }$ be a subcategory of $\mathcal{C}$
containing an additive generator $\Lambda $ of $\mathcal{C}$. If $\mathcal{T}
$ is a partial tilting subcategory of $\mathrm{\mathrm{Mod}}(\mathcal{C}%
^{\prime })$, then $\mathcal{C}\otimes _{\mathcal{C}^{\prime }}\mathcal{T}$
is a partial tilting subcategory of $\mathrm{\mathrm{Mod}}(\mathcal{C})$.
\end{proposition}

\begin{proof}
By Lemma \ref{ET1}, $\mathcal{C}\otimes _{\mathcal{C}^{\prime }}\mathcal{T}$
is self orthogonal. (ii) Let $T$ be an object in $\mathcal{T}$, and $%
0\rightarrow \mathcal{C}^{\prime }(\;\;,C_{1}^{\prime })\xrightarrow{%
\mathcal{C}'(\;\;,f)}\mathcal{C}(\;\;,C_{0}^{\prime })\rightarrow
T\rightarrow 0$, a projective resolution with $C_{0}^{\prime }$ and $%
C_{1}^{\prime }$ in $\mathcal{C}^{\prime }$. Applying $\mathcal{C}\otimes -$
, we get the projective presentation $\mathcal{C}(\;\;,C_{1}^{\prime })%
\xrightarrow{\mathcal{C}(\;\;,f)}\mathcal{C}(\;\;,C_{0}^{\prime
})\rightarrow \mathcal{C}\otimes _{\mathcal{C}^{\prime }}T\rightarrow 0$.
Since $\varLambda$ is in $\mathcal{C}^{\prime }$, then $\mathcal{C}^{\prime
}(\varLambda,C_{0}^{\prime })=\mathcal{C}(\varLambda,C_{0}^{\prime })$ and $%
\mathcal{C}^{\prime }(\varLambda,C_{1}^{\prime })=\mathcal{C}(\varLambda%
,C_{1}^{\prime })$. In consequence, $\mathcal{C}\otimes _{\mathcal{C}%
^{\prime }}T(\varLambda)=T(\varLambda)$. We need to see $\mathcal{C}(\;\;,f)$
is monomorphism. Let $C$ be an object in $\mathcal{C}$. Since $\varLambda$
is an additive generator of $\mathcal{C}$, then there is an epimorphism $g:%
\varLambda^{n}\rightarrow C\rightarrow 0$, and hence a monomorphisms $%
0\rightarrow \mathcal{C}(C,C_{i}^{\prime })\rightarrow \mathcal{C}(\varLambda%
,C_{i}^{\prime })$, for $i=0,1$.

In this way, we obtain a commutative exact diagram:%
\begin{equation*}
\divide\dgARROWLENGTH by2\begin{diagram} \node{}\arrow{s,!}{}\arrow{e,!}{}
\node{0}\arrow{s}{}\arrow{e,!}{} \node{0}\arrow{s}{}\arrow{e,!}{} \node{}
\node{}\\ \node{}\arrow{s,!}{}\arrow{e,!}{}
\node{\mathcal{C}(C,C'_1)}\arrow{s,r}{\mathcal{C}(g,C_1')}\arrow{e,t}{%
\mathcal{C}(C,f)}
\node{\mathcal{C}(C,C'_0)}\arrow{s,r}{\mathcal{C}(g,C_0')}\arrow{e,t}{}
\node{\mathcal{C}\otimes_{\mathcal{C}'}
T(C)}\arrow{s,r}{\mathcal{C}\otimes_{\mathcal{C}'} T(g)}\arrow{e,t}{}
\node{0}\\ \node{}\arrow{s,!}{}\arrow{e,!}{} \node{\mathcal{C}(\varLambda^
n,C'_1)}\arrow{s,=,-}{}\arrow{e,t}{\mathcal{C}(\varLambda^n,f)}
\node{\mathcal{C}(\varLambda^n,C'_0)}\arrow{s,=,-}{}\arrow{e,}{}
\node{\mathcal{C}\otimes_{\mathcal{C}'}
T(\varLambda^n)}\arrow{s,=,-}{}\arrow{e}{} \node{0}\\
\node{0}\arrow{s,!}{}\arrow{e,t}{}
\node{\mathcal{C}'(\varLambda^n,C_1')}\arrow{e,t}{\mathcal{C}'(%
\varLambda^n,f)} \node{\mathcal{C}'(\varLambda^n,C_0')}\arrow{e,}{}
\node{T(\varLambda^ n)}\arrow{e}{} \node{0} \end{diagram}
\end{equation*}%
and for any object $C$ in $\mathcal{C}$ the map $\mathcal{C}(C,f)$ is a
monomorphism.
\end{proof}

\begin{theorem}
Let $\mathcal{C}$ be an $\mathrm{Hom}$-finite, $K$-category with additive
generator $\varLambda$ and let $R_{\varLambda}$ be the endomorphism ring $R_{%
\varLambda}=\mathrm{End}(\varLambda)^{op}$. Assume $\mathrm{\mathrm{Mod}}(R_{%
\varLambda}) $ has a partial tilting module. Then $\mathrm{\mathrm{Mod}}(%
\mathcal{C})$ has a tilting subcategory.
\end{theorem}

\begin{proof}
From the equivalence $\mathrm{\mathrm{Mod}}(R_{\varLambda})$ and $\mathrm{%
\mathrm{Mod}}(\mathrm{add}\{\varLambda\})$, it follows the category $\mathrm{%
\mathrm{Mod}}(\mathrm{add}\{\varLambda\})$ has a partial tilting object. The
category $\mathcal{C}^{\prime }=\mathrm{add}\{\varLambda\}$, is a
subcategory of $\mathcal{C}$ containing $\varLambda$. By Proposition \ref%
{ET2}, there exists an object $T $ in $\mathrm{\mathrm{Mod}}(\mathcal{C})$
which is a partial tilting in $\mathrm{\mathrm{Mod}}(\mathcal{C})$. Since $T$
is finitely presented, we have an exact sequence
\begin{equation}
0\rightarrow L\rightarrow \mathcal{C}(\;\;,C_{0})\rightarrow T\rightarrow 0%
\text{,}  \label{ET3}
\end{equation}%
with $L$ finitely generated.

Let $C$ be an object in $\mathcal{C}$. Applying $\mathrm{Hom}_{\mathcal{C}%
}(\;\;,\mathcal{C}(\;\;,C))$ to (\ref{ET3}), and using the long homology
sequence, we obtain the following exact sequence:
\begin{equation*}
\mathrm{Hom}_{\mathcal{C}}(\mathcal{C}(\;\;,C_{0}),\mathcal{C}%
(\;\;,C))\rightarrow \mathrm{Hom}_{\mathcal{C}}(L,\mathcal{C}%
(\;\;,C))\rightarrow \mathrm{Ext}_{\mathcal{C}}(T,\mathcal{C}%
(\;\;,C))\rightarrow 0\text{.}
\end{equation*}

Since $L$ is finitely generated, there is an epimorphism $\mathcal{C}%
(\;\;,C_{1})\rightarrow L\rightarrow 0$, and hence a monomorphism $%
0\rightarrow \mathrm{Hom}_{\mathcal{C}}(L,\mathcal{C}(\;\;,\mathcal{C}%
))\rightarrow \mathrm{Hom}_{\mathcal{C}}(\mathcal{C}(\;\;,C_{1}),\mathcal{C}%
(\;\;,C))\cong \mathrm{Hom}_{\mathcal{C}}(C_{1},C)$. Since $\mathcal{C}$ is $%
\mathrm{Hom}$-finite, then $\mathrm{Hom}_{\mathcal{C}}(C_{1},C)$ is a finite
dimensional $K$-vector space. It follows $\mathrm{Hom}_{\mathcal{C}}(L,%
\mathcal{C}(\;\;,C))$ and $\mathrm{Ext}_{\mathcal{C}}^{1}(T,\mathcal{C}%
(\;\;,C))$ are finite dimensional $K$-vector spaces.

We proceed now as in [B]:

Let $\mathbf{e}_{1},\cdots ,\mathbf{e}_{d}$ be a $K$-base of $\mathrm{Ext}_{%
\mathcal{C}}^{1}(T,\mathcal{C}(\;\;,C))$. Represent each $\mathbf{e_{i}}$ as
a short exact sequence
\begin{equation*}
0\rightarrow \mathcal{C}(\;\;,C)\xrightarrow{f_i}E_{i}\xrightarrow{g_i}%
T\rightarrow 0\text{.}
\end{equation*}%
Consider the following diagram with exact raws:
\begin{equation*}
\begin{diagram}\dgARROWLENGTH=.5em \node{0}\arrow{e}{}
\node{\mathcal{C}(\;\;,C)^d}\arrow{s,l}{\nabla}\arrow{e,t}{f}
\node{\coprod_{i=1}^d E_i}\arrow{s,l}{u}\arrow{e,t}{g}
\node{T^d}\arrow{s,l}{1}\arrow{e}{} \node{0}\\ \node{0}\arrow{e}{}
\node{\mathcal{C}(\;\;,C)}\arrow{e,t}{v} \node{E_C}\arrow{e,t}{w}
\node{T^d}\arrow{e}{} \node{0} \end{diagram}
\end{equation*}%
where $f$ and $g$ are sums of morphisms $f_{i}$, and $g_{i}$ respectively
and $\nabla =[1,\ldots ,1]$. Let's denote by $\mathbf{e}_{C}$ the element of
$\mathrm{Ext}_{\mathcal{C}}^{1}(T^{d},\mathcal{C}(\;\;,C))$ represented by
the exact sequence
\begin{equation*}
\mathbf{e}_{C}:0\rightarrow \mathcal{C}(\;\;,C)\xrightarrow{v}E_{C}%
\xrightarrow{w}T^{d}\rightarrow 0\text{.}
\end{equation*}%
Let $u_{i}:T\rightarrow T^{d}$ be the $i$-th inclusions. We claim for each $%
i=1,\ldots ,d$, $\mathbf{e}_{i}=\mathrm{Ext}_{\mathcal{C}}^{1}(u_{i},%
\mathcal{C}(\;\;,C))\mathbf{e}_{C}$. Indeed, let's consider the following
commutative diagram with exact raws:
\begin{equation*}
\begin{diagram}\dgARROWLENGTH=.3em \node{\;\;\;\;0}\arrow{e}{}
\node{\mathcal{C}(\;\;,C)}\arrow{s,l}{u_i^{''}}\arrow{e,t}{f_i}
\node{E_i}\arrow{s,l}{u_i'}\arrow{e,t}{g_i}
\node{T}\arrow{s,l}{u_i}\arrow{e}{} \node{0}\\ \node{\;\;\;\;0}\arrow{e}{}
\node{\mathcal{C}(\;\;,C)^d}\arrow{s,l}{\nabla}\arrow{e,t}{f}
\node{\coprod_{i=1}^d E_i}\arrow{s,l}{u}\arrow{e,t}{g}
\node{T^d}\arrow{s,l}{1}\arrow{e}{} \node{0}\\ \node{0}\arrow{e}{}
\node{\mathcal{C}(\;\;,C)}\arrow{e,t}{v} \node{E_C}\arrow{e,t}{w}
\node{T^d}\arrow{e}{} \node{0} \end{diagram}
\end{equation*}%
where $u_{i}^{\prime },u_{i}^{\prime \prime }$ denote the corresponding
inclusions. Since $\nabla u_{i}^{\prime \prime }=1_{\mathcal{C}(\;\;,C)}$,
we get the following commutative diagram with exact raws:
\begin{equation*}
\begin{diagram}\dgARROWLENGTH=.5em \node{0}\arrow{e}{}
\node{\mathcal{C}(\;\;,C)}\arrow{s,l}{1}\arrow{e,t}{f_i}
\node{E_i}\arrow{s,l}{uu_i'}\arrow{e,t}{g_i}
\node{T}\arrow{s,l}{u_i}\arrow{e}{} \node{0}\\ \node{0}\arrow{e}{}
\node{\mathcal{C}(\;\;,C)}\arrow{e,t}{v} \node{E_C}\arrow{e,t}{w}
\node{T^d}\arrow{e}{} \node{0} \end{diagram}
\end{equation*}%
hence, the claim follows.

Applying $\mathrm{Hom}_{\mathcal{C}}(T,\;\;)$ to the sequence $\mathbf{e}%
_{C} $, we obtain by he long homology sequence, an exact sequence
\begin{equation*}
\mathrm{Hom}_{\mathcal{C}}(T,T^{d})\xrightarrow{\delta}\mathrm{Ext}_{%
\mathcal{C}}^{1}(T,\mathcal{C}(\;\;,C))\rightarrow \mathrm{Ext}_{\mathcal{C}%
}^{1}(T,E_C)\rightarrow \mathrm{Ext}_{\mathcal{C}}^{1}(T,T^{d})=0\text{.}
\end{equation*}

Since $\mathbf{e_{i}}=\mathrm{Ext}_{\mathcal{C}}^{1}(u_{i},\mathcal{C}%
(\;\;,C))\mathbf{e}_{C}=\delta (u_{i})$, it follows each basic element of
the vector space $\mathrm{Ext}_{\mathcal{C}}^{1}(T,\mathcal{C}(\;\;,C))$ is
in the image of the connecting morphism $\delta $, which is then surjective.
Therefore: $\mathrm{Ext}_{\mathcal{C}}^{1}(T,E_{C})=0$.

We will see that the full subcategory $\mathcal{T}$ of $\mathrm{\mathrm{Mod}}%
(\mathcal{C})$, with $\mathcal{T}=\{T\coprod E_{C}\}_{C\in \mathcal{C}}$, is
a tilting category.

Since $E_{C}$ is an extension of finitely presented functors, it is finitely
presented. Hence, for any object $\mathcal{C}$ in $\mathrm{\mathrm{Mod}}(%
\mathcal{C)}, $ the sum $T\coprod E_{C}$ is finitely presented.

(i) $\mathrm{pdim}T\coprod E_{C}\leq 1$. Since $T$ is a partial tilting
object object in $\mathrm{\mathrm{Mod}}(\mathcal{C})$, it has $\mathrm{pdim}%
T\leq 1$, and from the exact sequence $\mathbf{e}_{C}$ it follows $\mathrm{%
pdim}E_{C}\leq 1$. Therefore: $\mathrm{pdim}T\coprod E_{C}\leq 1$.

(ii) For any pair of objects $C$ and $C^{\prime }$ in $\mathcal{C}$, $%
\mathrm{Ext}_{\mathcal{C}}^1(T\coprod E_{C},T\coprod E_{C^{\prime }})=0$.
Applying $\mathrm{Hom}_{\mathcal{C}}(\;\;,T)$ and $\mathrm{Hom}_{\mathcal{C}%
}(\;\;,E_{C^{\prime}})$ to $\mathbf{e}_{C},$ we obtain by the long homology
sequence, exact sequences:
\begin{eqnarray*}
0 &=&\mathrm{Ext}_{\mathcal{C}}^{1}(T^{d},T)\rightarrow \mathrm{Ext}_{%
\mathcal{C}}^{1}(E_{C},T)\rightarrow \mathrm{Ext}_{\mathcal{C}}^{1}(\mathcal{%
C}(\;\;,C),T)=0, \\
0 &=&\mathrm{Ext}_{\mathcal{C}}^{1}(T^{d},E_{C^{\prime}})\rightarrow \mathrm{%
Ext}_{\mathcal{C}}^{1}(E_{C},E_{C^{\prime}})\rightarrow \mathrm{Ext}_{%
\mathcal{C}}^{1}(\mathcal{C}(\;\;,C),E_{C^{\prime}})=0.
\end{eqnarray*}%
Hence, for any pair of objects $C$ and $C^{\prime }$ in $\mathrm{\mathrm{Mod}%
}(\mathcal{C})$, we have:
\begin{equation*}
\mathrm{Ext}_{\mathcal{C}}^{1}(T,E_{C^{\prime}})=\mathrm{Ext}_{\mathcal{C}%
}^{1}(E_{C},T)=\mathrm{Ext}_{\mathcal{C}}^{1}(E_{C},E_{C^{\prime}})=0\text{.}
\end{equation*}

Using the fact $\mathrm{Ext}^1_{\mathcal{C}}(T,T)=0$ and the above
equalities, the condition (ii) follows. (iii) is immediate.
\end{proof}

\subsection{The Hereditary Case}

In this subsection $Q=(Q_{0},Q_{1})$ will be a \textbf{locally finite, }%
infinite quiver. We will consider the quiver algebra $KQ$ as a subadditive $%
K $-category $\mathcal{C},$ and the finite dimensional representations%
\textbf{\ }of the quiver, will be identified with the category of finitely
presented contravariant functors $\mathrm{mod}(KQ)$ from $KQ$ to he category
of finite dimensional $K$-vector spaces. The category $\mathcal{C}$ is $%
\mathrm{Hom}$-finite, dualizing and Krull-Schmidt. By [AR] $\mathrm{mod}(%
\mathcal{C})$ has almost split sequences. We will describe the
Auslander-Reiten components, beginning with he preprojective components.

The preprojective component $\mathcal{K}$ of the Auslander-Reiten $\Gamma
(KQ)$, of $KQ$, can be computed as in the finite quiver situation. It is
easy to prove that it is a translation quiver of the form $(-\mathbb{N}%
\Delta ,\tau )$, $\Delta =Q^{op}$. The category $\mathcal{K}$ is the quiver
algebra $K(-\mathbb{N}\Delta )$ module the mesh relations $\mathcal{K=}K(-%
\mathbb{N}\Delta )/<m_{x}>$ [R, Lemma 3, Section 2.3].

We will need the following combinatorial lemma, which can be proved as in
[BGP].

\begin{lemma}
\label{Dinher} Let $Q$ be a locally finite infinite quiver that is not of
type $A_{\infty }$, $A_{\infty }^{\infty }$, $D_{\infty }$. Given a finite
subquiver $Q^{\prime \prime }$of $Q$, there exists a finite full connected
subquiver $Q^{\prime }\subset Q$ such that $Q^{\prime \prime }$ is a
subquiver of $Q^{\prime }$ and $Q^{\prime }$ is not Dynkin.
\end{lemma}

\begin{definition}
\emph{[ASS]} Let $(\Gamma ,\tau )$ be a connected translation quiver. A
connected full subquiver $\Sigma $ is a \textbf{section} of $\Gamma $ if the
following conditions are satisfied:

\begin{itemize}
\item[(S1)] If $X_{0}\rightarrow X_{1}\rightarrow \cdots \rightarrow X_{t}$
is a path in $\Sigma $ of lenght $l\geq 1$, then $X_{0}\neq X_{t}$.

\item[(S2)] For each $X\in\Gamma_0$, there exist a unique $n\in\mathbb{Z}$
such that $\tau^n X\in\Sigma_0$.

\item[(S3)] If $X_0\rightarrow X_1\rightarrow \cdots\rightarrow X_t$ is a
path in $\Gamma$ with $X_0,X_t\in\Sigma$, then $X_i\in\Sigma_0$ for all $i$
such that $0\le i\le t$.
\end{itemize}
\end{definition}

\begin{definition}
\emph{[ARS]} Let $I$ be the integers in one of the intervals $(-\infty,n]$, $%
[n,\infty)$, $[m,n]$ for $m<n$ or $\{1,\ldots,n\}$ modulo $n$. Let $%
\cdots\rightarrow X_i\xrightarrow{f_i}X_{i+1}\rightarrow \cdots$ be a path
in the Auslander-Reiten quiver $(\Gamma,\tau)$ with each index in $I$. This
path is said to be \textbf{sectional} if $\tau X_{i+2}\ncong X_i$.
\end{definition}

\begin{lemma}
Let $\mathcal{C}$ be a dualizing variety, and $X_{1}\overset{f_{1}}{%
\rightarrow }X_{2}\overset{f_{2}}{\rightarrow }\cdots X_{n-1}\overset{f_{n-1}%
}{\rightarrow }X_{n}$ sectional path in $\mathrm{mod}(\mathcal{C)}$. Then
the composition $f_{n-1}f_{n-2}...f_{1}$ is not zero.
\end{lemma}

\begin{proof}
The proof is as in [ARS Theo. 2.4].
\end{proof}

\begin{remark}
Let $Q$ be a locally finite infinite quiver, $\mathcal{K}$ the preprojective
component of the Auslander-Reiten quiver of $KQ$, that we identify with the
quiver algebra with quiver $(-\mathbb{N}Q,\tau )$ and the mesh relations.
Then given an object $X$ in $\mathcal{K}$ and a positive integer $n$, there
exists a finite number of directed paths in $(-\mathbb{N}Q,\tau )$ with
starting vertex $X$ and length $\leq $ $n$.
\end{remark}

\begin{theorem}
\label{directedpaths1} Let $Q$ be a locally finite infinite quiver and $%
\Sigma $ a section of $(-\mathbb{N}Q,\tau )$ with no infinite directed
paths, then the following is true

\begin{itemize}
\item[(a)] For any vertex $X$ of $(-\mathbb{N}Q,\tau )$ the number of
directed paths from $X$ to the section is finite and there is a non zero
path from $X$ to the section.

\item[(b)] For any vertex $X$ of $(-\mathbb{N}Q,\tau )$ the number of
directed paths from the section to $X$ is finite and there is a non zero
path from the section to $X$.
\end{itemize}
\end{theorem}

\begin{proof}
(a) The proof will be by induction on the least $n$ $\geq 0$ such that $\tau
^{-n}X$ is in $\Sigma $.

In the case there is not such an $n$, then $X$ is a successor of the section
and the number of paths to the section is zero.

If $n=0$, then $X$ is on the section and the claim is true, since the number
of directed paths on the section is finite.

Assume $n=1$ and consider $X$ an element of the preprojective component with
almost split sequence:

\begin{equation*}
0\rightarrow X\overset{(\sigma \alpha _{i})}{\rightarrow }\overset{n}{%
\underset{i=1}{\coprod }}E_{i}\overset{(\alpha _{i})}{\rightarrow }\tau
^{-1}X\rightarrow 0
\end{equation*}

If all $E_{i}$ are in the section, then there is nothing to prove.

Let $\alpha _{i}:E_{i}\rightarrow \tau ^{-1}X$ be a map that it is not in
the section. Then $\sigma ^{-1}\alpha _{i}:\tau ^{-1}X\rightarrow \tau
^{-1}E_{i}$ is in the section, assume there is a irreducible map $\beta $ : $%
E_{i}\rightarrow Y$, different from $\alpha _{i}$. Then there is an
irreducible map $\sigma ^{-1}\beta :Y\rightarrow \tau ^{-1}E_{i}.$ Let $n$
be the maximum of the lengths of the paths in $\Sigma $ starting at $\tau
^{-1}X.$ By the above remark, there exists only a finite number of paths of
length $\leq n$ starting at $Y$ and assume $Y\overset{\beta }{\rightarrow }%
Y_{1}\overset{\beta _{1}}{\rightarrow }Y_{2}\overset{\beta _{2}}{\rightarrow
}...\overset{\beta _{n-1}}{\rightarrow }Y_{n}$ is a path that does not meet
any of the paths starting at $\tau ^{-1}X$ and ending at $\Sigma .$ Then for
$k\leq n$ we have the following diagram of irreducible maps:

\begin{equation}  \label{directedfinite}
\begin{diagram}\dgARROWLENGTH=1em \node{\tau^{-1}Y_{k-1}}
\node{\tau^{-1}Y_{k-2}}\arrow{w,t}{\sigma^{-2}\beta_{k-1}}
\node{\tau^{-1}Y_{k}}\arrow{w,..,-}
\node{\tau^{-1}E_1}\arrow{w,t}{\sigma^{-2}\beta_{k-1}}
\node{\tau^{-1}X}\arrow{w,t}{\sigma^{-1}\alpha}
\node{E_2}\arrow{w,t}{\alpha_2}\\ \node{Y_{k}}\arrow{n,l}{\sigma^{1}\beta_k}
\node{Y_{k-1}}\arrow{w,t}{\beta_{k}}\arrow{n,l}{\sigma^{-1}\beta_{k-1}}
\node{Y_{1}}\arrow{w,..,-}\arrow{n,l}{\sigma^{-1}\beta_1}
\node{Y}\arrow{w,t}{\beta_{1}}\arrow{n,l}{\sigma^{-1}\beta}
\node{E_1}\arrow{w,t}{\beta}\arrow{n,l}{\alpha}
\node{X}\arrow{w,t}{\sigma\alpha}\arrow{n,l}{\sigma\alpha_2} \end{diagram}
\end{equation}

with $\tau ^{-1}Y_{k-1}$ in $\Sigma $, and there is no irreducible map $\tau
^{-1}$Y$_{k-1}\rightarrow $Z with Z on $\Sigma $. By the definition of
section, $\sigma ^{-1}\beta _{k}$ :Y$_{k}\rightarrow \tau ^{-1}$Y$_{k-1}$ is
in $\Sigma $.

Assume there is only a finite number of paths for any $X$ in $(-\mathbb{N}%
Q,\tau )$ with $\tau ^{-m+1}X$ $\in \Sigma _{0}.$

As before, we consider an almost split sequence starting at $X$ and a map $%
\alpha _{i}:E_{i}\rightarrow \tau ^{-1}X$ that is not in the section. Then,
by induction hypothesis, there is a finite number of paths from $\tau ^{-1}X
$ to the section. Let $n$ be the length of the largest such path. Assume
there is a irreducible map $\beta $ : $E_{i}\rightarrow Y$, different from $%
\alpha _{i}$. Arguing as above, we consider a path $Y\overset{\beta }{%
\rightarrow }Y_{1}\overset{\beta _{1}}{\rightarrow }Y_{2}\overset{\beta _{2}}%
{\rightarrow }...\overset{\beta _{n-1}}{\rightarrow }Y_{n}$ that does not
meet any of the paths starting at $\tau ^{-1}X$. Then we obtain a diagram of
irreducible maps as in (\ref{directedfinite}) and for some integer $k\leq n$%
, $\tau ^{-1}$Y$_{k-1}$ in $\Sigma .$ Then either $\sigma ^{-1}\beta _{k}$ :Y%
$_{k}\rightarrow \tau ^{-1}$Y$_{k-1}$ is in $\Sigma $, or $\tau ^{-1}Y_{k}$
is in $\Sigma _{0}$ . In any case, it follows from the case $n=1$, there is
only a finite number of paths from $X$ to $\Sigma .$ Moreover, since
sectional paths have non zero composition, there is a non zero path from $X$
to the section.

(b) Is proved using dual arguments or going to the opposite category.
\end{proof}

\begin{proposition}
\label{herin1} Let $\mathcal{K}$ be a preprojective component $\Gamma (KQ)$
and $\Sigma $ a section of $\mathcal{K}$ without infinite directed paths.
Let $P$ be an indecomposable projective. Then there exist a short exact
sequence, $0\rightarrow P\rightarrow T^{0}\rightarrow T^{1}\rightarrow 0$,
with $T^{0}$, $T^{1}\in \mathrm{add}\Sigma _{0}$.
\end{proposition}

\begin{proof}
(1) We will separate the proof in two cases, assuming first $Q$ is not of
type $A_{\infty }$, $A_{\infty }^{\infty }$ or $D_{\infty }$. We use the
isomorphism $\mathcal{K}$ $\cong K(-\mathbb{N}\Delta )/<m_{X}>$, where $%
m_{X} $ denotes the set of mesh relations. The projective $P$ is identified
with a vertex $v_{0}$ in $\Delta _{0}$. Let $\mathcal{V}_{0}^{-}=\{v_{i}^{-}%
\in Q_{0}\mid $there is a path $v_{i}^{-}\rightarrow v_{0}\}$, let $\mathcal{%
V}_{0}=\{v_{1},v_{2}...,v_{n}\}$ be the vertices of $\Sigma $ that are the
ending vertex of a path starting at some $v_{i}^{-}$. Since $\Sigma $ has no
infinite directed paths, the set $\mathcal{V}_{1}=\{v_{0},v_{1}...,v_{m}\}$
of all vertices of $\Sigma $ that are connected with a length zero or
directed path to vertices of $\mathcal{V}_{0}$, is also finite and contains $%
\mathcal{V}_{0}$, $\Sigma ^{\prime \prime }$ is the full subquiver of $%
\Sigma $ with vertices $\mathcal{V}_{1}$. Denote by $\mathcal{V}_{P}$ the
collection of all paths starting at some vertex of $\mathcal{V}_{0}^{-}$ and
ending at $\Sigma $.

Let $\{T_{0},T_{1}...,T_{m}\}$ be the objects of $\mathcal{K}$ corresponding
to the vertices $\mathcal{V}_{1}$ and denote by $\mathcal{P}$ denote the
finite set of indecomposable projective that appear as summands in the
minimal projective presentations of any of the $T_{i}$ for $1\leq i\leq m$.
The objects in $\mathcal{P}$ correspond to vertices of a finite subquiver $%
Q^{\prime \prime }$ of $Q$, which by Lemma \ref{Dinher} is contained in a
finite, full, non Dynkin subquiver $Q^{\prime }$of $Q$.

Let $\Delta ^{\prime }=Q^{\prime op}$, since $\Delta ^{\prime }$ is not
Dynkin, $(-\mathbb{N}\Delta ^{\prime })$ is a connected, full subquiver of $%
(-\mathbb{N}\Delta )$ and $\mathcal{V}_{P}$ is completely contained in $(-%
\mathbb{N}\Delta ^{\prime })$. We identify the preprojective component $%
\mathcal{K}^{\prime }$ of the Auslander-Reiten quiver $\Gamma (KQ^{\prime })$%
, with $K(-\mathbb{N}\Delta ^{\prime })/<m_{X}^{\prime }>$, where $%
m_{X}^{\prime }$ is the set of mesh relations in $K(-\mathbb{N}\Delta
^{\prime })$. Let's consider the ideal $I$ of $K(-\mathbb{N}\Delta )$
defined by:
\begin{equation*}
f\in I(X,Y),\text{ if and only if, $f$ factors through $\tau ^{-i}Z$, for
some $Z\in \Delta _{0}/\Delta _{0}^{\prime }$, $i\in \mathbb{N}$}
\end{equation*}

By the universal properly of quiver algebras, there exists a functor:

\begin{equation*}
\overset{-}{F}:K(-\mathbb{N}\Delta ^{\prime })\rightarrow \frac{K(-\mathbb{N}%
\Delta )}{I+<m_{X}>}
\end{equation*}

The set of meshes can be separated in three kinds: (i) Meshes $%
m_{X_{j}}^{\prime }$ which are also meshes in $K(-\mathbb{N}\Delta )$, (ii)
The mesh $m_{X_{k}}$ of $K(-\mathbb{N}\Delta )$ that is of the form $\Sigma
_{\{\alpha \in \Gamma (\varLambda)|F(\alpha )=X_{k}\}}\alpha \sigma (\alpha
) $, where $\alpha \sigma (\alpha )$ factors through an object in $(-\mathbb{%
N}\Delta )\smallsetminus(-\mathbb{N}\Delta ^{\prime })$ and (iii) The meshes
$m_{X_{l}}$ of $K(-\mathbb{N}\Delta )$ of the form $m_{X_{l}}=\rho
_{j}^{1}+\rho _{l}^{2} $, where $\rho _{l}^{1}$ is a sum of morphisms
factoring through some object in $(-\mathbb{N}\Delta )\smallsetminus(-%
\mathbb{N}\Delta ^{\prime })$ and $\rho _{l}^{2}$ consists of morphims which
do not factor through $(-\mathbb{N}\Delta )\smallsetminus(-\mathbb{N}\Delta
^{\prime })$ . Then it is clear that the kernel of $\overset{-}{F}$ is $%
<m_{X}^{\prime }>$ and there exists a full, faithful and dense functor:

\begin{equation*}
F:\frac{K(-\mathbb{N}\Delta ^{\prime })}{<m_{X}^{\prime }>}\rightarrow \frac{%
K(-\mathbb{N}\Delta )}{I+<m_{X}>}
\end{equation*}

Let $\Sigma ^{\prime }$ be the subquiver of $\Sigma $ consisting of all
vertices which correspond to orbits under the inverse Auslander-Reiten
translation of projective corresponding to the vertices of $Q^{\prime }$,
since $Q^{\prime }$ is non Dynkin, $\Sigma ^{\prime }$ is a section of $%
\mathcal{K}^{\prime }$.

(2) According to [HR Theo. 7.2], there exists a short exact sequence of $%
KQ^{\prime }$-modules
\begin{equation*}
0\rightarrow P\xrightarrow{(f_i)_i}\coprod_{i}T_{i}%
\xrightarrow{(g_{ji})_{ji}}\coprod_{j}T_{j}\rightarrow 0
\end{equation*}%
with $T_{i}$ and $T_{j}$ corresponding to vertices in $\Sigma _{0}^{\prime }$%
, and $f_{i},g_{j_{i}}$ are $K$-linear combinations of paths in $K(-\mathbb{N%
}\Delta ^{\prime })$.

By the equivalence in part (1), $(f_{i})_{i}$ is a monomorphism and $%
\sum_{i} g_{ji}f_{i}=0$ in $\mathcal{K}$.

We have the exact sequence of $KQ$-modules
\begin{equation*}
0\rightarrow P\xrightarrow{(f_i)_i}\coprod_{i}T_{i}%
\xrightarrow{(h_{ki})_{ki}}\coprod_{k}C_{k}\rightarrow 0
\end{equation*}
where $C=\mathrm{Coker}(g_{ji})_{ji}$, is the cokernel of $(g_{ji})_{ji}$ in
$\mathcal{K}$ and $C=\coprod_{k}C_{k}$ its decomposition in sum of
indecomposable $KQ$-modules.

By the universal property of the cokernel, there exists a morphism $%
(l_{jk})_{jk}:\coprod_{k}C_{k}\rightarrow \coprod_{j}T_{j}$ such that the
map $g_{ij}$ is equal to $\sum_{k}l_{jk}h_{ki}:T_{i}\rightarrow T_{j}$. By
condition (S3) of the definition of section, each $C_{k}\in \Sigma _{0}$.

If $Q$ is of type $A_{\infty }$, $A_{\infty }^{\infty }$ or $D_{\infty }$,
then we can choose $\Delta ^{\prime }=Q^{\prime op}$ large enough in order
to the injective $KQ^{\prime }$-modules of the Auslander-Reiten quiver of $%
\Gamma (KQ^{\prime })$ appear as successor of $\Sigma ^{\prime \prime }$ and
then apply an argument similar to the first case.
\end{proof}

We next prove he main result in this subsection.

\begin{theorem}
\label{tilthr} Let $Q$ be a locally finite infinite quiver, $K$ a field and $%
KQ$ the quiver algebra consider as a preadditive category. Let $\mathcal{K}$
be a preprojective component of the Auslander-Reiten quiver of $KQ$ and $%
\Sigma $ a section of $\mathcal{K}$ without infinite oriented paths. Then, $%
\mathrm{add}\Sigma _{0}$ is a tilting subcategory of $\mathrm{mod}(KQ)$.
\end{theorem}

\begin{proof}
Since $Q$ is locally finite, $KQ$ is hereditary and (iii) was proved in
Proposition \ref{herin1}, we only need to prove condition (ii). Let $T_{1}$
and $T_{2}$ be non projective objects in $\Sigma _{0}$, since $KQ$ is a
dualizing variety, Auslander-Reiten formula $\mathrm{Ext}%
_{KQ}^{1}(T_{1},T_{2})$=D($\mathrm{Hom}_{KQ}(T_{2},\tau T_{1}))$ holds. Let $%
\varphi :T_{2}\rightarrow \tau T_{1}$ be a non zero morphism, then using the
Ausalnder-Reiten sequence we have morphisms between indecomposable objects:%
\begin{equation*}
T_{2}\rightarrow \tau T_{1}\rightarrow E_{i}\rightarrow T_{1}
\end{equation*}%
From condition (S3), $\tau T_{1}\in \Sigma $, a contradiction.
\end{proof}

\begin{corollary}
\label{Sectilt} If $\mathcal{T}=\mathrm{add}\Sigma _{0}$ is a tilting
subcategory, then $\mathrm{\mathrm{Mod}}(\mathcal{T})$ has global dimension
one.
\end{corollary}

\begin{proof}
Let $P$ be an indecomposable projective in $\mathrm{mod}(\mathcal{T})$.
Since $\mathcal{T}$ is Krull-Schmidt $P\cong (\;\;,T)$ with $T$ in $\Sigma
_{0}$. By Brenner-Butler's theorem, the subcategory $\mathscr Y=\{N|\mathrm{%
Tor}_{\mathcal{T}}^{1}(N,\mathcal{T})=0\}\subset \mathrm{mod}(\mathcal{T})$
is a torsion free class containing the projective, and it is equivalent to $%
\mathscr T=\{M|\mathrm{Ext}_{\mathcal{C}}^{1}(T_{i},M)=0,T_{i}\in \Sigma
_{0}\}\subset \mathrm{mod}(\mathcal{C})$, via the functor $\phi $.

Let $Y$ be a non zero sub-object of $P$. Then there is a monomorphism $%
\alpha :Y\rightarrow P$. Since $P$ is contained in the subcategory $\mathscr %
Y\subset \mathrm{mod}(\mathcal{T})$, then $Y\in \mathscr Y$, since it is
closed under sub-objects. By the equivalence $\phi :\mathscr T\rightarrow %
\mathscr Y$, there exists an indecomposable object $M$ in $\mathscr T$ such
that $\phi (M)=(\;\;M)_{\mathcal{T}}\cong Y$. Moreover, the inclusion $%
Y\hookrightarrow P$ is of the form $(\;\;,f):(\;\;,M)_{\mathcal{T}%
}\rightarrow (\;\;,T)$, with $f:M\rightarrow T$ a non zero morphism. Since $%
Y $ is a non zero functor, there exists an object $T^{\prime }\in \Sigma
_{0} $ such that $0\neq Y(T^{\prime })=\mathrm{Hom}(T^{\prime },M)$. Let $%
g\in \mathrm{Hom}(T^{\prime },M)$ be a non zero morphism, then there is a
chain of morphisms
\begin{equation*}
T^{\prime }\xrightarrow{g}M\xrightarrow{f}T
\end{equation*}%
Then by Property (S3) of a section, $M$ is in $\Sigma _{0}$ and we conclude $%
Y$ is projective.
\end{proof}

\subsubsection{Representations of Infinite Dynkin Diagrams.}

\bigskip In this sub section we will apply the results of the previous subsection
 to compute the Auslander-Reiten quivers of the infinite Dynkin
quivers $A_{\infty }$, $A_{\infty }^{\infty }$ or $D_{\infty }$ without
infinite paths. 

The Auslander Reiten quivers of the infinite Dynkin quivers  were computed
in [ReVan III. 3]  for a fixed orientation. We will apply here the tilting
theory so far developed to compute the Auslander Reiten quivers for
arbitrary orientations.

We give first the computation of the Auslander Reiten for an infinite Dynkin
quiver with only sinks and sources,  then we change the orientation by tilting with
the objects in a section of the preprojective component with no infinite
paths, and we prove that tilting with objects in a section behaves as in the
finite dimensional case, it removes a portion from the preprojecive
component and glues it on the preinjective component leaving the other
components invariant.

\begin{proposition}
Let $Q$ be an infinite Dynkin quiver with only sinks and sources, $\Gamma
(Q) $ the Auslander-Reiten quiver of $Q$ then:

\begin{itemize}
\item[(a)] $\Gamma (A_{\infty })$ have only two
components: the preprojective and the preinjective components.
\item[(b)]  $\Gamma (D_{\infty })$ have three
components: the preprojective,  the preinjective and a regular component, 
the regular component are of type $A_{\infty }$.

\item[(c)] $\Gamma (A_{\infty }^{\infty })$ has 4 components: the
preprojective, the preinjective and two regular components, the regular
components are of type $A_{\infty }$.
\end{itemize}
\end{proposition}

\begin{proof}
In the case of locally finite infinite quivers, we have almost split
sequences and we can compute them, and the preprojective component, as in
the finite dimensional case; proceeding by induction starting with the
indecomposable projective.

Every indecomposable representation has finite support, hence it can be
considered as a representation of a finite Dynkin quiver.

(a) Consider the quiver $A_{\infty }:0\leftarrow 1\rightarrow 2\leftarrow
3\rightarrow \cdots $, for each pair of integers $b\geq a\geq 0$, let $%
M_{a}^{b}\in \mathrm{rep}(A_{\infty })$, be the representation defined as $%
(M_{a}^{b})_{i}=K$ if $a\leq i\leq b$ and $0$ in the remaining vertices. The
simple projective objects are $P(2n)=M_{2n}^{2n}$, and the non simple
projective are $P(2n+1)=M_{2n}^{2(n+1)}$, with $n\geq 0$. The non projective
representation $M_{a}^{b}$ with $b$ even, can be written as $M_{2k}^{2m}$, $%
M_{2k+1}^{2m}$, with $m>k+1$, $k\geq 0$ and $M_{1}^{2m}$, with $m\geq 1$. A
computation shows $\mathrm{DTr}M_{2k}^{2m}=M_{2(k+1)}^{2(m-1)}$, $\mathrm{DTr%
}(M_{2k+1}^{2m})=M_{2k-1}^{2(m-1)}$, and $\mathrm{DTr}%
(M_{1}^{2m})=M_{0}^{2(m-1)}$, hence each $M_{a}^{b}$ with $b$ even, is in
the preprojective component. By a similar computation $M_{a}^{b}$ with $b$
odd is in the preinjective component.

We look now to the quiver $D_{\infty }$
\begin{equation*}
\begin{diagram}\dgARROWLENGTH=.3em \node{0}
\node{2}\arrow{w}\arrow{e}\arrow{s} \node{3} \node{4}\arrow{w}\arrow{e}
\node{\cdots}\\ \node{} \node{1} \node{} \node{} \node{} \end{diagram}
\end{equation*}

We use the fact that every indecomposable representation has support in a
finite Dynkin quiver whose representations we know, (See [R]).

For each pair of integers $m$,$n$, with $m\ge n\ge 2$ we have indecomposable
representations, $M_{n}^{m}$, defined as follows: $(M_{n}^{m})_{i}=K$, if $%
n\leq i\leq n$, and $0$ in any other vertex. For $m\geq 2$ we define the
indecomposable representations
\begin{equation*}
(N_{0}^{m})_{i}=%
\begin{cases}
K & \text{if } 1\le i\leq m , \\
0 & \text{if } i=0\text{ or } i>m.%
\end{cases}%
,\;(N_{1}^{m})_{i}=%
\begin{cases}
K & \text{if } i=0\text{ or } 1<i\leq m, \\
0 & \text{ in any  other vertex.}%
\end{cases}%
\end{equation*}%
For integers $m$ and  $l $,  with $m>l \geq 2$, let's define
\begin{equation*}
(L_{l}^{m})_{i}=%
\begin{cases}
K^{2} & \text{if } 2\le i\le l, \\
K & \text{if } i\in \{0,1\} \text{ or } l+1\le i\le m, \\
0 & \text{in the other vertices}.%
\end{cases}%
\end{equation*}%
For each integer  $m\ge 0$,  let's define
\begin{equation*}
(L^{m})_{i}=%
\begin{cases}
K & \text{if } 0\le i\le m, \\
0 & \text{in the other vertices}.%
\end{cases}%
\end{equation*}%

To compute the preprojective components of the Auslander-Reiten quiver $%
K(D_{\infty })$ we compute the orbits of $P(1)$ and $P(0)$ under $\mathrm{Tr%
}D$ and to compute the preinjective component we take the orbits of $I(1)$
and $I(0)$ under $D\mathrm{Tr}$, to obtain
\begin{eqnarray*}
&&\{P(1),N_{1}^{3},N_{0}^{5},N_{1}^{7},N_{0}^{9},\ldots
\},\;\{P(0),N_{0}^{3},N_{1}^{5},N_{0}^{7},N_{1}^{9},\ldots \} \\
&&\{I(1),N_{1}^{4},N_{0}^{6},N_{1}^{8},N_{0}^{10},\ldots
\},\;\{I(0),N_{0}^{4},N_{1}^{6},N_{0}^{8},N_{1}^{10},\ldots \}
\end{eqnarray*}%
We see that the representations that lies in the preproyective component are:
$N_1^m$, $N_0^m$, $L^m$ and $M_n^m$ with $n$ and $m$ odd or $0$, and
 $L_l^m$ with $l$ and $m$ odd. 
The representations that lies in the preinyective  component are:
$N_1^m$, $N_0^m$,  $M_n^m$ with $n$ and $m$ odd or $0$, and $N_l^m$ with $l$ and $m$ even. 

Finally, the representations $M_n^m$, $L_n^m$ with $m+n$ and $l+m$ odd, and $L^m$, $M_0^m$ 
with $m$ even, lies in a component of type $\mathbb{Z}A_{\infty}$.

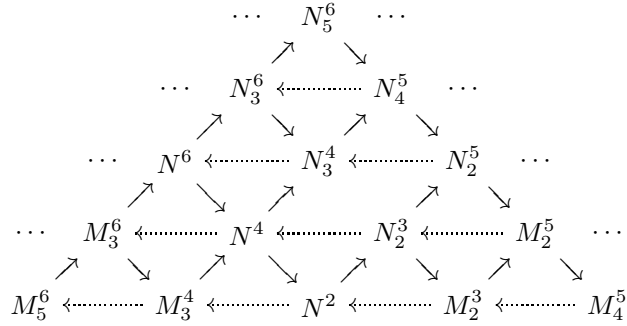
\[
 \begin{diagram}\dgARROWLENGTH=.4em
  \node{}
   \node{}
    \node{}
     \node{\cdots}
      \node{N_{5}^6}\arrow{se}
       \node{\cdots}
        \node{}
         \node{}
          \node{}\\
\node{}
   \node{}
    \node{\cdots}
     \node{N_{3}^6}\arrow{se}\arrow{ne}
      \node{}
       \node{N_{4}^5}\arrow{se}\arrow[2]{w,..}
        \node{\cdots}
         \node{}
          \node{}\\
\node{}
   \node{\cdots}
    \node{N^6}\arrow{se}\arrow{ne}
     \node{}
     \node{N_{3}^4}\arrow[2]{w,..}\arrow{ne}
      \node{}
       \node{N_{2}^5}\arrow{se}\arrow[2]{w,..}
        \node{\cdots}
         \node{}\\
\node{\cdots}
    \node{M_{3}^6}\arrow{se}\arrow{ne}
     \node{}
      \node{N^4}\arrow{se}\arrow{ne}\arrow[2]{w,..}
       \node{}
        \node{N_2^3}\arrow{se}\arrow{ne}\arrow[2]{w,..}
         \node{}
          \node{M_2^5}\arrow{se}\arrow[2]{w,..}
          \node{\cdots}\\
   \node{M_{5}^6}\arrow{ne}
    \node{}
     \node{M_{3}^4}\arrow{ne}\arrow[2]{w,..}
      \node{}
       \node{N^2}\arrow{ne}\arrow[2]{w,..}
        \node{}
         \node{M_2^3}\arrow{ne}\arrow[2]{w,..}
          \node{}
           \node{M_4^5}\arrow[2]{w,..}
 \end{diagram}
\]

\vspace{.5cm}

(b) We consider next the quiver $A_{\infty }^{\infty }:\cdots \rightarrow
-2\leftarrow -1\rightarrow 0\leftarrow 1\rightarrow 2\leftarrow 3\rightarrow
\cdots $ As above, for each pair of integers $b\geq a$, let $M_{a}^{b}\in
\mathrm{rep}(A_{\infty }^{\infty })$, be the representation defined as $%
(M_{a}^{b})_{i}=K$ if $a\leq i\leq b$ and $0$ in the remaining vertices. The
projective simple are $P(2n)=M_{2n}^{2n}$, and the non simple projective are
$P(2n+1)=M_{2n}^{2(n+1)}$, with $n\in \mathbb{Z}$. The non projective
representations $M_{a}^{b}$ with $a$ and $b$ even can be written as $%
M_{2k}^{2m}$ with $m>k+1$, and $k\geq 0$, and we have $\mathrm{TrD}%
(M_{2k}^{2m})=M_{2(k+1)}^{2(m-1)}$, and by induction $M_{a}^{b}$ with $a$
and $b$ even are in the preprojective component. Using the same argument, we
can see that in case $a$ and $b$ are odd, then $M_{a}^{b}$ is in the
preinjective component.

If $a$ is even and $b$ odd, then $\mathrm{DTr}(M_{a}^{b})=M_{a+2}^{b+2}$,
and all these representations are in a regular component, if $a$ is odd and $%
b$ is even, then $\mathrm{DTr}(M_{a}^{b})=M_{a-2}^{b-2}$ and we obtain the
elements of the second regular component.

If $a$ is even and $b$ odd. The almost split sequences are of the form
\begin{eqnarray*}
0\rightarrow M_{a+2}^{b+2}\rightarrow M_{a}^{b+2}\rightarrow
M_{a}^{b}\rightarrow 0 \\
0\rightarrow M_{a+2}^{b+2}\rightarrow M_{a+2}^{b}\coprod
M_{a}^{b+2}\rightarrow M_{a}^{b}\rightarrow 0
\end{eqnarray*}
the first appears in on the border of the regular component

If $a$ is odd and $b$ even. The almost split sequences are of the form
\begin{eqnarray*}
0\rightarrow M_{a-2}^{b-2}\rightarrow M_{a-2}^b\rightarrow
M_{a}^{b}\rightarrow 0 \\
0\rightarrow M_{a-2}^{b-2}\rightarrow M_{a}^{b-2}\coprod
M_{a-2}^b\rightarrow M_{a}^{b}\rightarrow 0
\end{eqnarray*}
the first appears in on the border of the regular component

Proceeding as in the finite dimensional case, we see that in the three cases
the preprojective components are of the form $(-\mathbb{N}Q,\tau )$ and the
preinjective components of the form $(\mathbb{N}Q,\tau )$.
\end{proof}

If $Q$ is a locally finite quiver then $KQ$ is a dualizing variety, if $%
\Sigma $ is a section without infinite paths, then $\mathcal{T}=\mathrm{add}%
\Sigma _{0}$ of $\mathrm{mod}(KQ)$, the functor $\phi :\mathrm{\mathrm{Mod}}(%
\mathcal{C})\rightarrow \mathrm{\mathrm{Mod}}(\mathcal{T})$ restricts to the
category of finitely presented functors $\phi _{\mathrm{mod}(\mathcal{C})}:%
\mathrm{mod}(\mathcal{C})\rightarrow \mathrm{mod}(\mathcal{T})$. We also
proved $\mathcal{T}$ is a dualizing hereditary category, therefore both $KQ$
and $\mathcal{T}$ splits.

Choosing a section without infinite paths corresponds with changes in the
orientation of the quiver $Q$. We observe next that tilting with the objects
of a section without infinite paths in the preprojective component, behave
as in the finite quiver situation; The Auslander Reiten components of the
tilted category are as follows: we cut all the predecessors of the section
in the preprojective component to get the preprojective component of the
tilted category. We glue what we cut as successors of the injective to build
the preinjective component. The regular components remain without changes.

Let $\mathcal{P}(\mathcal{C})$ be the preprojective component of the
Auslander-Reiten quiver, $\Gamma (KQ)$ and $\mathcal{Q}(\mathcal{T})$ the
preinjective component of the Auslander-Reiten quiver de $\Gamma (\mathcal{T}%
)$. By Auslander-Reiten formula we have
\begin{eqnarray*}
\mathscr{T}(\mathcal{T})&=&\{X\in \mathrm{mod}(KQ)|\mathrm{Ext}^{1}(T,X)=0%
\text{, }T\in \Sigma _{0}\} \\
{} &=&\{X\in \mathrm{mod}(KQ)|\mathrm{Hom}(X,\tau T)=0\text{, }T\in \Sigma
_{0}\}.
\end{eqnarray*}

We also know%
\begin{equation*}
\mathscr{F}(\mathcal{T})=\{X\in \mathrm{mod}(KQ)|\mathrm{Hom}(T,X)=0\text{, }%
T\in \Sigma _{0}\}\text{.}
\end{equation*}

\begin{definition}
The set of predeccessors of $\Sigma $ (successors, $succ\Sigma $), $%
pre\Sigma $, is the set of indecomposables $X$ such that there is a $T\in
\Sigma _{0}$ and an integer $n>0$ with $T=\tau ^{-n}X$ ($T=\tau ^{n}X$).
\end{definition}

\begin{lemma}
Let $X$ be an indecomposable object in $\mathrm{mod}(KQ)$. Then, $X$ is in $%
\mathscr{T}(\mathcal{T})$, if and only if, $X$ is not a predecessor of $%
\Sigma $. Moreover, $\mathscr{F}(\mathcal{T})=pre\Sigma $.
\end{lemma}

\begin{proof}
Assume $X$ is in $\mathscr{T}(\mathcal{T})$. If $X$ is a predeccesssor of $%
\Sigma $, then, by Theorem \ref{directedpaths1}, there is a non zero path
from $X$ to $\tau \Sigma $, a contradiction.

Assume now $X$ is not a predeccessor of $\Sigma $. If there exists, $T\in
\Sigma _{0}$ and a non zero map $f:X\rightarrow \tau T$, then for some
positive integer $k$, $\tau ^{k}T$ is projective, hence $X$ is preprojective
and it has to be a succesor of $\Sigma $. Then there is a $T_{0}\in \Sigma
_{0}$ and a non zero map $g:T_{0}\rightarrow X.$

Hence we have maps: $T_{0}\rightarrow X\rightarrow \tau T\rightarrow
E\rightarrow T$, contradicting the fact, $T_{0},T\in \Sigma _{0}.$

The last claim follows from the fact that the tilting category $\mathcal{T}$
separates.
\end{proof}

We have the following:

\begin{proposition}
The follwing statements hold:

\begin{itemize}
\item[(a)] For each $T\in \Sigma _{0}$, the $\mathcal{T}$-module $\mathrm{Ext%
}^{1}(\;\;,\tau T)_{\mathcal{T}}$ is injective.

\item[(b)] For any positive integer $k$, and any object $T$ in $\Sigma $,
there is an isomorphism $\tau ^{k}\mathrm{Ext}^{1}(\;\;,\tau T)_{\mathcal{T}%
}\cong \mathrm{Ext}^{1}(\;\;,\tau ^{k+1}T)_{\mathcal{T}}$.

\item[(c)] Given an indecomposable projective $(-,C)$ there is a natural
isomorphism in $\mathrm{mod}(\mathcal{T}^{op}):D(\mathrm{Ext}^{1}(-,(-,C))_{%
\mathcal{T}})\cong \tau (((-,C),-)_{\mathcal{T}})$.
\end{itemize}
\end{proposition}

\begin{proof}
(a) For each $X\in \mathrm{mod}(\mathcal{T}),$ and each non projective $T\in
\mathcal{T}$ there is an isomorphism
\begin{equation*}
(X,\mathrm{Ext}^{1}(\;\;,\tau T)_{\mathcal{T}})\cong DX(T)\text{.}
\end{equation*}%
Indeed, let $0\rightarrow (\;\;,T_{1})\rightarrow (\;\;,T_{0})\rightarrow
X\rightarrow 0$ be a projective resolution of $X$, and $T$ in $\mathcal{T}$
non projective. Applying Auslander-Reiten formula, there is an isomorphism $%
\eta :(X,\mathrm{Ext}^{1}(\;\;,\tau T)_{\mathcal{T}})\rightarrow DX(T)$,
such that the following diagram commutes
\begin{equation*}
\begin{diagram}\dgARROWLENGTH=.3em
\node{0\rightarrow(X,\mathrm{Ext}^1(\;\;,\tau T))}\arrow{s,l}{\eta}\arrow{e}
\node{((\;\;,T_0),\mathrm{Ext}^1(\;\;,\tau T))}\arrow{s,l}{\cong}\arrow{e}
\node{((\;\;,T_1),\mathrm{Ext}^1(\;\;,\tau T))}\arrow{s,l}{\cong}\\
\node{0\rightarrow D(X(T))}\arrow{e} \node{D(T,T_0)}\arrow{e} \node{D(T,
T_1)} \end{diagram}
\end{equation*}

(b) Consider an almost split sequence in $\mathscr{F}(\mathcal{T})$:%
\begin{equation*}
0\rightarrow \tau ^{2}T\rightarrow E\rightarrow \tau T\rightarrow 0
\end{equation*}

By Lemma \ref{BrennerB9}, it induces an almost split sequence in $\mathscr{X}%
(\mathcal{T})$:

\begin{equation*}
0\rightarrow \mathrm{Ext}^{1}(-,\tau ^{2}T)_{\mathcal{T}}\rightarrow \mathrm{%
Ext}^{1}(-,E)_{\mathcal{T}}\rightarrow \mathrm{Ext}^{1}(-,\tau T)_{\mathcal{T%
}}\rightarrow 0
\end{equation*}

from which it follows $\tau \mathrm{Ext}^{1}(\;\;,\tau T)_{\mathcal{T}}\cong
\mathrm{Ext}^{1}(\;\;,\tau ^{2}T)_{\mathcal{T}}.$

The rest is by induction.

(c) Let $(-,C)$ be an indecomposable projective in $\mathrm{mod}(KQ)$, then
there is an exact sequence:
\begin{equation*}
0\rightarrow (-,C)\rightarrow T_{1}\rightarrow T_{0}\rightarrow 0
\end{equation*}

which induces by the long homology sequence exact sequences:%
\begin{equation*}
0\rightarrow (-,(-,C))\rightarrow (-,T_{1})\rightarrow (-,T_{0})\rightarrow
\mathrm{Ext}^{1}(-,(-,C))_\mathcal{T}\rightarrow \;\;\text{0}
\end{equation*}%
\begin{equation*}
0\rightarrow (T_{0}-)\rightarrow (T_{1},-)\rightarrow ((-,C),-)_\mathcal{T}%
\rightarrow 0
\end{equation*}

The second exact sequence gives a projective presentation of $((-,C),-)_%
\mathcal{T}$, taking the transpose and dualizing we obtain the exact
sequence:%
\begin{equation*}
0\rightarrow \tau (((-,C),-)_{\mathcal{T}})\rightarrow
D((-,T_{0}))\rightarrow D((-,T_{1}))\rightarrow 0
\end{equation*}

Dualizing the first exact sequence we obtain an exact sequence:%
\begin{equation*}
0\rightarrow D(\mathrm{Ext}^{1}(-,(-,C))_{\mathcal{T}}\mathcal{)}\rightarrow
D((-,T_{0}))\rightarrow D((-,T_{1}))\rightarrow 0
\end{equation*}

which implies:

\begin{equation*}
D(\mathrm{Ext}^{1}(-,(-,C))_{\mathcal{T}}\mathcal{)}\cong \tau (((-,C),-)_{%
\mathcal{T}}\mathcal{)}
\end{equation*}
\end{proof}

The results in the proposition can be interpreted as the construction of the
Auslander-Reiten components of $\mathrm{mod}(\mathcal{T})$ by gluing the
predecessors of $\Sigma $ as successors of the injectives and leaving the
remaining components unchanged, as illustrated in the following diagram:

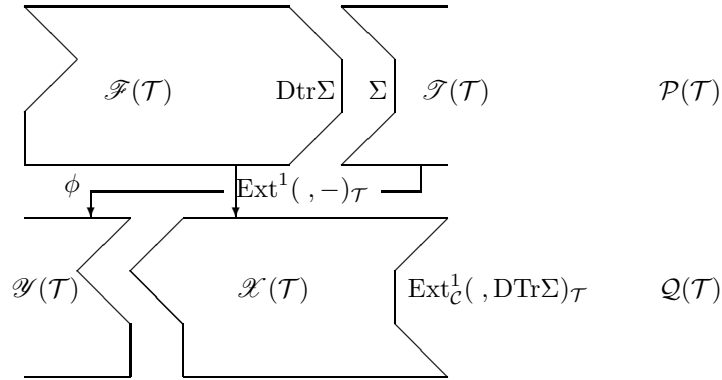
\begin{center}
\begin{picture}(250,150)

  \put(0,0){\line(1,0){40}}
   \put(40,0){\line(0,1){20}}
    \put(20,40){\line(1,-1){20}}
      \put(20,40){\line(1,1){20}}
        \put(0,60){\line(1,0){40}}

\put(60,0){\line(1,0){100}}
 \put(60,0){\line(0,1){20}}
   \put(60,20){\line(-1,1){20}}
       \put(60,20){\line(-1,1){20}}
        \put(40,40){\line(1,1){20}}
         \put(60,60){\line(1,0){100}}

\put(160,0){\line(-1,1){20}}
 \put(140,20){\line(0,1){20}}
   \put(140,40){\line(1,1){20}}

  \put(0,80){\line(1,0){100}} \put(80,80){\vector(0,-1){20}} \put(80,68){$\mathrm{Ext}^1(\;,-)_\mathcal{T}$}
   \put(150,80){\line(0,-1){10}}
    \put(150,70){\line(-1,0){15}}
     \put(75,70){\line(-1,0){50}}
      \put(25,70){\vector(0,-1){10}}
       \put(15,70){$\phi$}

   \put(100,80){\line(1,1){20}}
       \put(120,100){\line(0,1){20}}
          \put(120,120){\line(-1,1){20}}
           \put(0,140){\line(1,0){100}}
    \put(0,80){\line(0,1){20}}
     \put(0,100){\line(1,1){20}}
      \put(0,140){\line(1,-1){20}}

\put(120,80){\line(1,0){40}}
 \put(120,80){\line(1,1){20}}
   \put(140,100){\line(0,1){20}}
    \put(140,120){\line(-1,1){20}}
         \put(120,140){\line(1,0){40}}

\put(30,105){$\mathscr{F}(\mathcal{T})$}
 \put(95,105){$\mathrm{Dtr}\Sigma$}
  \put(130,105){$\Sigma$}
   \put(150,105){$\mathscr{T}(\mathcal T)$}
   \put(-5,30){$\mathscr{Y}(\mathcal{T})$}
       \put(80,30){$\mathscr{X}(\mathcal{T})$}
               \put(145,30){$\mathrm{Ext}^1_\mathcal{C}(\;,\mathrm{DTr}\Sigma)_\mathcal{T}$}
\put(240,30){$\mathcal{Q}(\mathcal{T})$}
   \put(240,105){$\mathcal{P}(\mathcal{T})$}
\end{picture}
\end{center}
\vspace{.5cm}
As a corollary we obtain the main theorem of the subsection.

\begin{theorem}
Let $Q$ be a locally finite infinite Dynkin quiver, $\Gamma (Q)$ the
Auslander-Reiten quiver of $Q$ then:

\begin{itemize}
\item[(a)] $\Gamma (A_{\infty })$ have only two
components: the preprojective and the preinjective components.
\item[(b)]  $\Gamma (D_{\infty })$ have three
components: the preprojective,  the preinjective and a regular component, 
the regular component are of type $A_{\infty }$.

\item[(c)] $\Gamma (A_{\infty }^{\infty })$ has 4 components: the
preprojective, the preinjective and two regular components, the regular
components are of type $A_{\infty }$.
\end{itemize}
\end{theorem}

\subsubsection{The regular components}

\bigskip In the previous subsection we were concerned with the preprojective
components and with all the components of a locally finite infinite quiver
of type: $A_{\infty }$, $A_{\infty }^{\infty }$ or $D_{\infty }$. In this
subsection we study the regular components of an arbitrary locally finite
infinite quiver and prove that the regular components are of type $A_{\infty
}$. To prove this, we will follow very closely the proof given by ABPRS,
[see ARS]. Since $KQ$ is hereditary and the radical of $\mathrm{mod}(KQ)$
has properties very similar to the finite dimensional case, we can also use
length arguments. We follow the first part of ABPRS's proof to conclude that
the number of indecomposable summands, $\alpha (M)$, of a $KQ$-module $M$ in
a regular component, is at most three, and in case $\alpha (M)=3$, there
exist chains of irreducible monomorphisms $B_{i,t_{i}}\rightarrow
B_{i,t_{i-1}}\rightarrow \cdots \rightarrow B_{i,1}=B_{i}\rightarrow M$,
with $\alpha (B_{i,t_{i}})=1$ and $\alpha (B_{i,j})=2$ for $j<t_{i}$, where
\begin{equation*}
0\rightarrow \mathrm{DTr}M\xrightarrow {\bigl(\begin{smallmatrix}
f_1\\f_2\\f_3\end{smallmatrix}\bigr)}B_{1}\coprod B_{2}\coprod B_{3}%
\xrightarrow{\bigl(\begin{smallmatrix} g_1&g_2&g_3\end{smallmatrix}\bigr)}%
M\rightarrow 0
\end{equation*}%
is an almost split sequence (see [ARS Prop. 4.11]). To exclude the case $%
\alpha (M)=3$, we will reduce to a finite quiver situation to obtain a
contradiction. We will make use of the following lemma:

\begin{lemma}
\label{herescsd4} Let $\mathcal{C}$ be a $\mathrm{Hom}$-finite dualizing
locally finite Krull-Schmidt $K$-category. Let $\mathcal{B}=\{B_{i}\}_{i\in
I}$ be a finite family of objects in $\mathrm{mod}(\mathcal{C})$. For each $%
i\in I$, consider almost split sequences in $\mathrm{mod}(\mathcal{C})$
\begin{equation*}
0\rightarrow \mathrm{DTr}(B_{i})\xrightarrow{f_i}E_{i}\xrightarrow{g_i}%
B_{i}\rightarrow 0\text{.}
\end{equation*}%
Then, there is a finite subcategory $\mathcal{C}^{\prime }\subset \mathcal{C}
$, such that the restriction
\begin{equation*}
0\rightarrow \mathrm{DTr}(B_{i})|\mathcal{C}^{\prime }\xrightarrow{g_i|%
\mathcal{C}'}E_{i}|\mathcal{C}^{\prime }\xrightarrow{f_i|\mathcal{C}'}B_{i}|%
\mathcal{C}^{\prime }\rightarrow 0
\end{equation*}%
is an almost split sequence in $\mathrm{mod}(\mathcal{C}^{\prime })$
\end{lemma}

\begin{proof}
We leave the proof to the reader.
\end{proof}

\begin{theorem}
Let $Q=(Q_{0},Q_{1})$ be a connected \textbf{locally finite }infinite quiver%
\textbf{\ } and $M$ an indecomposable module in the regular component of the
Auslander-Reiten quiver $\Gamma (KQ)$. Then $\alpha (M)\leq 2$.
\end{theorem}

\begin{proof}
Since we already know the shape of the Auslander-Reiten components of the
infinite Dynkin quivers, we may assume $Q$ is non Dynkin.

Let's suppose $\alpha (M)=3$, and let $0\rightarrow \mathrm{DTr}%
(M)\rightarrow \coprod_{i=1}^{3}B_{i}\rightarrow M\rightarrow 0$ be an
almost split sequence and chains of irreducible monomorphisms $%
B_{i,t_{i}}\rightarrow \cdots \rightarrow B_{i,1}=B_{i}\rightarrow
M,\;i=1,2,3$, as above.

(a) We proceed as in the lemma \ref{herescsd4}, defining $\mathcal{B}%
=\{M\}\cup \{B_{i,j}\}\cup \{\mathrm{TrD}(B_{i,j})\}$, $i=1,2,3$, $1\leq
j\leq t_{i}$, to find a subcategory $\mathcal{C}^{\prime }\subset \mathcal{C}
$, such that the almost split sequences of the objects in $\mathcal{B}$,
restrict to almost split sequences in $\mathrm{mod}(\mathcal{C}^{\prime })$
\begin{eqnarray*}
0 &\rightarrow &\mathrm{DTr}(B_{i,j})|\mathcal{C}^{\prime }\rightarrow E_{i}|%
\mathcal{C}^{\prime }\rightarrow B_{i,j}|\mathcal{C}^{\prime }\rightarrow
0,1\leq j\leq t_{i} \\
0 &\rightarrow &B_{i,j}|\mathcal{C}^{\prime }\rightarrow E_{i}^{\prime }|%
\mathcal{C}^{\prime }\rightarrow \mathrm{TrD}B_{i,j}|\mathcal{C}^{\prime
}\rightarrow 0,1\leq j\leq t_{i} \\
0 &\rightarrow &\mathrm{DTr}M|\mathcal{C}^{\prime }\rightarrow
\coprod_{i=1}^{3}B_{i}|\mathcal{C}^{\prime }\rightarrow M|\mathcal{C}%
^{\prime }\rightarrow 0,
\end{eqnarray*}

Adding a finite number of objects, if needed, we can identify $\mathcal{C}%
^{\prime }$ with $KQ^{\prime }$, where $Q^{\prime }$ is a finite connected
non Dynkin full subquiver of $Q$. Observe that these almost split sequences
will be almost split sequences for any quiver algebra of a finite subquiver $%
Q^{\prime \prime }$ of $Q^{\prime }$containing $Q^{\prime \prime },$ since
it will contain the support of the objects in the sequences.

Since $\alpha (M)=3$, the module $M|\mathcal{C}^{\prime }$ is in a
preprojective or in a preinjective component. We assume it is in the
preprojective component, the other case will follow by duality.
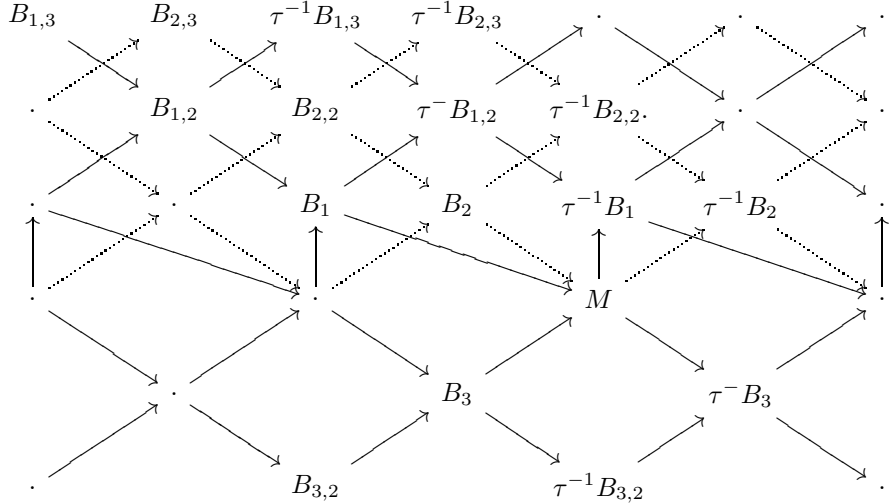
\begin{equation*}
\begin{diagram}\dgARROWLENGTH=1.5em \node{B_{1,3}}\arrow{se}
\node{B_{2,3}}\arrow{se,..} \node{\tau^{-1}B_{1,3}}\arrow{se}
\node{\tau^{-1}B_{2,3}}\arrow{se,..} \node{.}\arrow{se}
\node{.}\arrow{se,..} \node{.}\\ \node{.}\arrow{ne,..}\arrow{se,..}
\node{B_{1,2}}\arrow{ne}\arrow{se} \node{B_{2,2}}\arrow{ne,..}\arrow{se,..}
\node{\tau^{-}B_{1,2}}\arrow{ne}\arrow{se}
\node{\tau^{-1}B_{2,2}.}\arrow{ne,..}\arrow{se,..}
\node{.}\arrow{ne}\arrow{se} \node{.}\\ \node{.}\arrow{ne}\arrow{see}
\node{.}\arrow{ne,..}\arrow{se,..} \node{B_1}\arrow{ne}\arrow{see}
\node{B_2}\arrow{ne,..}\arrow{se,..}
\node{\tau^{-1}B_1}\arrow{ne}\arrow{see}
\node{\tau^{-1}B_2}\arrow{ne,..}\arrow{se,..} \node{.}\\ \node{.}\arrow{n}
\arrow{ne,..}\arrow{se} \node{} \node{.}\arrow{n}\arrow{ne,..}\arrow{se}
\node{} \node{M}\arrow{ne,..}\arrow{n}\arrow{se} \node{} \node{.}\arrow{n}\\
\node{} \node{.}\arrow{ne}\arrow{se} \node{} \node{B_3}\arrow{ne}\arrow{se}
\node{} \node{\tau^{-}B_3}\arrow{ne}\arrow{se} \node{}\\ \node{.}\arrow{ne}
\node{} \node{B_{3,2}}\arrow{ne} \node{} \node{\tau^{-1}B_{3,2}}\arrow{ne}
\node{} \node{.} \end{diagram}
\end{equation*}
Since we are assuming $Q^{\prime}$ is non Dynkin, the preprojective
component is of the from $(-\mathbb{N}\Delta ^{\prime },\tau )$, with $%
Q^{\prime }=\Delta ^{\prime }$, hence, the section consisting of the three
paths $B_{i,t_{i}}\rightarrow \cdots \rightarrow B_{i,1}=B_{i}\rightarrow
M,\;i=1,2,3$ is isomorphic to $Q^{\prime }$ after a change of orientation.
But for any larger subquiver $Q^{\prime \prime }$ the algebra $KQ^{\prime
\prime }$ will have the same section in its preprojective component as $%
KQ^{\prime }$, which implies $Q$ is a finite quiver, contradicting our
hypothesis.
\end{proof}

These results are used in the next section.

The shape of the Auslander-Reiten components of infinite quivers was found
in an independent way by [BSP].

\section{The Auslander Reiten Components of a regular \newline
Auslander Reiten Component}

In the last section we use the results of the previous section to describe
the Auslander Reiten components of a regular Auslander-Reiten component of a
finite dimensional algebra, to do this, we need the concepts and results of
[MVS1], [MVS2], [MVS3], which we briefly recall here.

Let $\Lambda $ be a finite dimensional algebra over an algebraically closed
field $K$. We denote by $\mathrm{mod}_{\Lambda }$ the category of finitely
generated left $\Lambda $-modules and by $\mathcal{A}_{gr}(\mathrm{mod}%
_{\Lambda })$ the category with the same objects as $\mathrm{mod}_{\Lambda }$
and whose morphisms between $A$ and $B$ in $\mathcal{A}_{gr}(\mathrm{mod}%
_{\Lambda })$ are given by:

\begin{center}
$\mathrm{Hom}_{\mathcal{A}_{gr}(\mathrm{mod}_{\Lambda })}(A,B)=
\coprod_{i\geq 0}\mathrm{rad}_{\Lambda }^{i}(A,B)/\mathrm{rad}_{\Lambda
}^{i+1}(A,B)$.
\end{center}

In [MVS1] the definition of Koszul and weakly Koszul categories was given
and it was proved that $\mathrm{mod}_{\Lambda }$ is weakly Koszul and $%
\mathcal{A}_{gr}(\mathrm{mod}_{\Lambda })$ is Koszul.

Koszul categories are a natural generalization of Koszul algebras and the
main results of Koszul theory extend to Koszul categories in particular we
have Koszul duality.

Next, recall the construction of the Ext-category of a full subcategory of
an abelian category. For an abelian category $\mathcal{D}$ we consider the
Ext-category $E(\mathcal{D}^{\prime })$ of the full subcategory $\mathcal{D}%
^{\prime }$of $\mathcal{D}$. The graded category $E(\mathcal{D}^{\prime })$
has the same objects as $\mathcal{D}^{\prime }$and maps given by:

\begin{equation*}
\mathrm{Hom}_{E(\mathcal{D}^{\prime })}(A,B)=\coprod_{i\geq 0}\mathrm{Ext}_{%
\mathcal{D}}^{i}(A,B)\text{,}
\end{equation*}

for all objects $A$ and $B$ in $E(\mathcal{D}^{\prime })$ .

We recall the application of this construction to $\mathrm{mod}(\mathrm{mod}%
_{\Lambda })$. Let's denote by $\mathcal{C}$ the category $\mathrm{mod}(%
\mathrm{mod}_{\Lambda })$ and by $\mathcal{S(C)}$ the full additive
subcategory of $\mathcal{C}$ generated by the simple functors

\begin{center}
S$_{C}$=Hom$_{\Lambda }$(-, $C$)/rad(-,C):(mod$_{\Lambda }$)$^{op}$,
\end{center}

for all indecomposable objects $C$ in $\mathrm{mod}_{\Lambda }$. The category%
$\mathcal{A}_{gr}(\mathrm{mod}_{\Lambda })$ is graded and according to [IT]
the category of graded functors $Gr(\mathcal{A}_{gr}(\mathrm{mod}_{\Lambda
}))$ has global dimension two. Then the Ext-category $E(\mathcal{S(C)})$ has
the same object as $\mathcal{S(C)}$ and maps:

\begin{equation*}
\mathrm{Hom}_{E(\mathcal{A}_{gr}(\mathrm{mod}_{\Lambda }))}(S_{A},S_{B})=
\coprod_{i\ge 0}\mathrm{Ext}_\mathcal{C}^i(S_{A},S_{B})
\end{equation*}

The category $E(\mathcal{S(C)})$ is locally finite of Loewy length 3.
Moreover, the following theorem was proved in [MVS3]:

\begin{theorem}
Let $C$ be an indecomposable $\Lambda $-module, and consider the
indecomposable projective functor P$_{C}=\mathrm{Hom}_{E(\mathcal{A}_{gr}(%
\mathrm{mod}_{\Lambda }))}(-,S_{C})$ in $Gr(E(\mathcal{A}_{gr}(\mathrm{mod}%
_{\Lambda })))$. Then one of the following statements is true:

\begin{itemize}
\item[(i)] P$_{C}$ is a simple projective, when C is a simple injective $%
\Lambda $-module.

\item[(ii)] P$_{C}$ is a projective of Loewy length 2, when C is a non
simple injective module.

\item[(iii)] P$_{C}$ is a projective injective functor of Loewy length 3,
when C is non injective.
\end{itemize}
\end{theorem}

Let $\mathcal{K}$ be a fixed component of the Auslander-Reiten quiver of $%
\Lambda $, $Agr(\mathcal{K}\emph{)}$ is the full subcategory of $\mathcal{A}%
_{gr}(\mathrm{mod}_{\Lambda })$ generated by objects in $\mathcal{K}$\emph{\
}the corresponding Ext-category $E(Agr(\mathcal{K}\emph{))}$ is equivalent
to $E(\mathcal{K}\emph{).}$ The full subcategory gr$_{\mathcal{PI}}(E(%
\mathcal{K}\emph{)}^{op})$ of gr($E(\mathcal{K}\emph{)}^{op})$ consisting of
objects in gr($E(\mathcal{K}\emph{)}^{op})$ without projective-injective
summands it is equivalent to a radical square zero category. Radical square
zero categories are stably equivalent to hereditary\textsl{\ }$K$%
-categories. The stable equivalence is obtained in a way similar to the
radical square zero artin algebras.

\begin{proposition}
\lbrack MVS3]Let $\mathcal{K}$ be the additive closure of a connected
component of the Auslander-Reiten quiver of a finite dimensional algebra $%
\Lambda $ over an algebraically closed field \textrm{K}\textsl{\ }.

\begin{itemize}
\item[(a)] The category gr($E(\mathcal{K}\emph{)}^{op}/$rad$^{2})\cong $gr$%
(Agr(\mathcal{K}\emph{)}$/rad$^{2})$ is stably equivalent to gr($\mathcal{H}%
) $, where $\mathcal{H}$ is a disjoint union of hereditary full
subcategories corresponding to the sections of the Auslander-Reiten
component of $\mathcal{K}$ given by $\mathcal{H}_{C}$ for an indecomposable
module $C$ in $\mathcal{K}$.

\item[(b)] Each subcategory gr($\mathcal{H}_{C})$ is equivalent to the
category of representations of the separated quiver of the component $%
\mathcal{K}$ of the Ausalnder-Reiten quiver at C.
\end{itemize}
\end{proposition}

For a regular component we have he following:

\begin{theorem}
\lbrack MVS3] Let $\mathcal{K}$ be a connected regular component of the
Auslander-Reiten quiver of a finite dimensional algebra $\Lambda $ over an
algebraically closed field \textrm{K}. Then

\begin{itemize}
\item[(a)] The category gr($E(\mathcal{K}\emph{)}^{op}/$rad$^{2})\cong $gr$%
(Agr(\mathcal{K}\emph{)}$/rad$^{2})$ is stably equivalent to gr($\mathcal{H}%
) $, where $\mathcal{H}$ is a disjoint union $\underset{i\in Z}{\cup }%
\mathcal{H}_{i}$ of hereditary full subcategories corresponding to the
sections of the Auslander-Reiten component of $\mathcal{K}$

\item[(b)] The categories $\mathcal{H}_{i}$, $\mathcal{H}_{i+2}$ are
equivalent and the categories $\mathcal{H}_{i}$, $\mathcal{H}_{i+1}$ are
opposite categories for all i.

\item[(c)] If $\mathcal{H}_{i}$ is finite, then it is non-Dynkin.
\end{itemize}
\end{theorem}

It was also proved in [MVS3] (See also [MVS2] ) that a regular component $%
\mathcal{K}$ the category $Agr(\mathcal{K}\emph{)}$ is Artin-Schelter
regular, its Ext-category $E(\mathcal{K}\emph{)}$ is locally finite,
Frobenius of radical cube zero. We will have a situation very similar to the
preprojective algebra and we can use the same line of ideas as in [MV] to
prove the following:

\begin{theorem}
The following statements hold:

\begin{itemize}
\item[(1)] $E(\mathcal{K}\emph{)}$ is selfinjective of radical cube zero,
for any indecomposable object \ \ $M$ in gr$_{\mathcal{PI}}(E(\mathcal{K}%
\emph{)}^{op})$ generated in degree zero $\Omega (M)$ is either simple
generated in degree 2 or it is generated in degree one.

\item[(2)] The indecomposable non Koszul objects in gr$(E(\mathcal{K}\emph{)}%
^{op})$ are the objects $M$ such that for some integer $n$, the n-th syzygy $%
\Omega ^{n}(M)$ is simple.

\item[(3)] For any indecomposable non projective object $M$ the
Ausalnder-Reiten translation is $\Omega ^{2}(\mathcal{N(}M))$, with $%
\mathcal{N}$ the Nakayama equivalence. Hence for $n=2k$ the n-th syzygy $%
\Omega ^{n}(M)$ is simple, if and only if, Dtr$^{k}$M is simple and for $%
n=2k+1$, the syzygy $\Omega ^{n}(M)$ is simple, if and only if , Dtr$^{k-1}$%
M is P/socP, with $P$ an indecomposable projective.

\item[(4)] gr$_{\mathcal{PI}}(E(\mathcal{K}\emph{)}^{op})$ is radical square
zero. It is stably equivalent to the disjoint union$\underset{i\in Z}{\cup }%
\mathcal{H}_{i}$ of hereditary full subcategories corresponding to the
sections of the Auslander-Reiten component of $\mathcal{K}$ and the
categories $\mathcal{H}_{i}$, $\mathcal{H}_{i+2}$ are equivalent and the
categories $\mathcal{H}_{i}$, $\mathcal{H}_{i+1}$ are opposite categories
for all i, the functor producing the stable equivalence sends all the almost
split sequences $0\rightarrow DtrM\rightarrow E\rightarrow M\rightarrow 0$
with $DtrM$ non simple, to almost split sequences and nodes (in this case
all simple are nodes) correspond to two simple objects, one projective and
one injective.

\item[(5)] The Auslander-Reiten components of $E(\mathcal{K}\emph{)}$
containing the projective objects are built by gluing for each i, a
preprojective component of $\mathcal{H}_{i}$ with a preinjective component
of $\mathcal{H}_{i}$ by identifying the simple projective with the
corresponding simple injective objects and adding projective injective
objects to recover the almost split sequences with a projective injective in
the middle term.

\item[(6)] All the regular components are of type A$_{\infty }$ .

\item[(7)] The objects in the preinjective and regular components are Koszul.

\item[(8)] The objects in the preprojective and regular components are co
Koszul.

\item[(9)] The category of indecomposable non projective Koszul object in $E(%
\mathcal{K}\emph{)}$ has almost split sequences to the left and they are
almost split in the whole category gr$(E(\mathcal{K}\emph{)}^{op})$.
\end{itemize}
\end{theorem}

\begin{corollary} The following statements hold:
\begin{itemize}

\item[(a)] If $\mathcal{K}$ is of type A$_{\infty }$ it is a tube, then $E(%
\mathcal{K}\emph{)}$ has a finite number of preprojective components and it
has no regular component.

\item[(b)] If $\mathcal{K}$ is of type A$_{\infty }$ and it is not a tube,
then $E(\mathcal{K}\emph{)}$ has a countable number of preprojective
components and it has no regular component.

\item[(c)] If $\mathcal{K}$ is of type D$_{\infty }$ , then $E(\mathcal{K}%
\emph{)}$ has a countable number of preprojective components and it has no
regular component.

\item[(b)] If $\mathcal{K}$ is of type A$_{\infty }^{\infty }$, then it has
a countable number of preprojective components and a countable number of
regular components of type A$_{\infty }$
\end{itemize}
\end{corollary}

If $\mathcal{K}$ is a regular component, then $E(\mathcal{K}\emph{)}$ is
Koszul and $E$($E(\mathcal{K}\emph{))}=$ $Agr(\mathcal{K)}.$ Koszul duality $%
\Psi :K_{E(\mathcal{K}\emph{)}}\rightarrow K_{Agr(\mathcal{K)}^{op}}$ sends
almost split sequences to almost split sequences, from these facts, we have
the following analogous of the preprojective algebra: ([MV] Theorem 2.8 and
[MVS3]).

\begin{theorem}
Let $\mathcal{K}$ be a regular Auslander-Reiten component of a finite
dimensional algebra $\Lambda .$ The Koszul functors of $Agr(\mathcal{K)}$
have the following properties:

\begin{itemize}
\item[(i)] Every indecomposable, non simple, non projective Koszul functor M
has projective dimension one.

\item[(ii)] For every indecomposable, non simple, Koszul functor M, there
exists a non splittable short exact sequence graded objects and maps: $%
0\rightarrow M\rightarrow E\rightarrow r^{2}\sigma M[2]\rightarrow 0$

where $\sigma $ is an auto equivalence of $Agr(\mathcal{K)}$ and $r$ denotes
the radical. The objects $E$ and $r^{2}\sigma M$ are Koszul and the sequence
is almost split in $K_{Agr(\mathcal{K)}^{op}}$.

\item[(iii)] The indecomposable Koszul functors of $Agr(\mathcal{K)}$ are
distributed in components, preprojective components, corresponding to the
preprojective components of the sections of $\mathcal{K}$ and regular
components, corresponding to the regular components of the sections of $%
\mathcal{K}$.

\item[(i')] Every indecomposable, non simple, non injective co Koszul
functor M has injective dimension one.

\item[(ii')] For every indecomposable, non simple, co Koszul functor M,
there exists a non splittable short exact sequence graded objects and maps: $%
0\rightarrow \sigma M/soc^{2}\sigma M\rightarrow E\rightarrow M\rightarrow 0$
where $\sigma $ is an auto equivalence of $Agr(\mathcal{K)}$ and $soc^{2}M$
denotes the second socle. The objects $E$ and $\sigma M/soc^{2}\sigma M$ are
co Koszul and the sequence is almost split in the category of co Koszul
functors co$K_{Agr(\mathcal{K)}^{op}}$.

\item[(iii')] The indecomposable co Koszul functors of $Agr(\mathcal{K)}$
are distributed in components, preinjective components, corresponding to the
preinjective components of the sections of $\mathcal{K}$ and regular
components, corresponding to the regular components of the sections of $%
\mathcal{K}$.
\end{itemize}
\end{theorem}

This theorem has a nicer interpretation in the quotient category module the
functors of finite length [MVS3]. We recall this construction. We will
obtain results similar to the case considered in [MZ].

\bigskip Let $\mathcal{K}$ be a regular Auslander-Reiten component of a
finite dimensional algebra $\Lambda $. To simplify the notation we will
denote by $\mathcal{C}$ the category $Agr(\mathcal{K)}$ and by gr($\mathcal{%
C)}_{0}$ the finitely presented graded functors and degree zero maps, by
[MVS1] gr($\mathcal{C)}_{0}$ is abelian.

Denote by tors($\mathcal{C)}$ the full subcategory of gr($\mathcal{C)}_{0}$
of all functors of finite length. This is a Serre category, we can take the
quotient category Qgr$(\mathcal{C)}$= gr($\mathcal{C)}_{0}/$tors$(\mathcal{C)%
}$.

The category $Qgr(\mathcal{C)}$ has the same objects as gr($\mathcal{C)}_{0}$
and maps:

Hom$_{Qgr(\mathcal{C)}}(\pi M,\pi N)=\underrightarrow{\lim }_{(M^{\prime
},N^{\prime })\in \mathcal{L}}Hom_{gr(C)_{0}}(M%
{\acute{}}%
,N/N^{\prime })$, where $\mathcal{L}$=\{(F,G)$\mid $\newline
F$\subset $M, G$\subseteq $N, and M/F, G in tors($\mathcal{C}$)\}.

Let $\pi :$gr($\mathcal{C)}_{0}\rightarrow Qgr(\mathcal{C)}$ be the
canonical projection. It is known [P], $Qgr(\mathcal{C)}$ is abelian and $%
\pi $ is exact. If we denote by $t(M)=\underset{L\in tors(\mathcal{C)}}{\sum
}L$ and $L$ a subfunctor of $M$, then $t$ is an idempotent radical and we
say that a functor $M$ is torsion if $t(M)=M$ and torsion free if $t(M)=0.$

It was proved im [MVS2] that for a finetely presented functor $M$ the
torsion part $t(M)$ is of finite lenght, in particular finetely presented.

Denote by $M_{\geq k}$ the truncation subfunctor of the graded functor $M$,
this is ($M_{\geq k}$)$_{j}=$0, if $j<k$ and ($M_{\geq k})_{j}=M_{j}$, if $%
j\geq k$. By definition, $M/M_{\geq k}$ is of finite length and the maps in $%
Qgr(\mathcal{C)}$ can be written as follows:

Hom$_{Qgr(\mathcal{C)}}(\pi M,\pi N)=\underrightarrow{\lim }%
_{k}Hom_{gr(C)_{0}}(M_{\geq k},N/t(N)).$

Since $\pi (M)\cong \pi (M/t(M))$ we may always assume $N$ is torsion free
and in such a case the exact sequence: $0\rightarrow M_{\geq k}\rightarrow
M_{\geq k-1}\rightarrow M_{k-1}/M_{\geq k}\rightarrow 0$ induces an exact
sequence: $0\rightarrow Hom_{gr(C)_{0}}(M_{k-1}/M_{\geq k},N)\rightarrow
Hom_{gr(C)_{0}}(M_{\geq k-1},N)\rightarrow Hom_{gr(C)_{0}}(M_{\geq k},N)$
whose first therm is zero and $\underrightarrow{\lim }%
_{k}Hom_{gr(C)_{0}}(M_{\geq k},N)=\underset{k}{\cup }Hom_{gr(C)_{0}}(M_{\geq
k},N)$.

Since $\mathcal{C}$ has global dimension two, any finitely presented functor
$M$ has a truncation $M_{\geq k}$ such that $M_{\geq k}[k]$ is Koszul (see
[MVS3] or [AE]) and in $Qgr(\mathcal{C})$ to objects $\pi M$ and $\pi N$ are
isomorphic, if and only if there are truncations $M_{\geq k}$ , $N_{\geq k}$
such that $M_{\geq k}$ $\cong $ $N_{\geq k}$ in gr($\mathcal{C)}_{0}$.

Hence $Qgr(\mathcal{C})=\underset{i\in Z}{\cup }\overset{\sim }{K}_{\mathcal{%
C}}[i]$ where $\overset{\sim }{K}_{\mathcal{C}}[i]$ is the image under $\pi $
of the Koszul $\mathcal{C}$-functors, shifted by i. The situation is
analogous to the sheaves on projective space studied in [MZ].

From this it follows:

\begin{theorem}
If we denote by $\mathcal{C}$ the associated graded category of a regular
component of a finite dimensional algebra, then the category $Qgr(\mathcal{C}%
)$ has almost split sequences, they are of the form: $0\rightarrow \pi
M[k]\rightarrow \pi E[k]\rightarrow \pi \sigma M[k+2]\rightarrow 0$ with $%
\sigma $an auto equivalence of $\mathcal{C}$. The category $Qgr(\mathcal{C})$
is the union of connected components of the Auslander-Reiten quiver and
these components are of the following kind: Proj[k] and Reg[k], this means
the $k$-th shift of $\pi (\Pr oj_{\mathcal{C}})$ or $\pi (\mathrm{Re}g_{C})$%
, where $\Pr oj_{\mathcal{C}}$ and $Reg$ $_{\mathcal{C}}$ denote a
preprojective component and a regular component of $\mathcal{C}$,
respectively. $\pi (\Pr oj_{\mathcal{C}})$ and $\pi (\Pr oj_{\mathcal{C}}),$
respectively, $\pi (\mathrm{Reg}_{C})$ and $\mathrm{Reg}_{C},$ have
isomorphic translation quivers.
\end{theorem}

\begin{proof}
By the above observation, any indecomposable object in $Qgr(\mathcal{C})$ is
of the form $\pi (M)[\ell ]$ with $M$ an indecomposable Koszul functor. The
endomorphism ring of $\pi (M)[\ell ]$ is End$_{Qgr(\mathcal{C})}(\pi
(M)[\ell ]$ )=$\underset{k}{\cup }Hom_{gr(C)_{0}}(M_{\geq k},M).$

The truncations $M_{\geq k}$ [k] are isomorphic to $r^{k}M[k],$hence they
correspond under Koszul duality to $\Omega ^{k}M[k]$ , since E($\mathcal{C)}$
is selfinjective, they are indecomposable with local endomorphism ring. A
map $f\in $End$_{Qgr(\mathcal{C})}(\pi (M)[\ell ]$ ) is represented by a map
$\overline{f}:M_{\geq k}\rightarrow M$ and restricting to the image a map $%
M_{\geq k}\rightarrow M_{\geq k}$, hence it is either an iso or nilpotent
and it follows End$_{Qgr(\mathcal{C})}(\pi (M)[\ell ]$ ) is local. We have
proved $Qgr(\mathcal{C})$ is Krull-Schmidt.

Consider the almost split sequence: $0\rightarrow M\overset{j}{\rightarrow }E%
\overset{p}{\rightarrow }r^{2}\sigma M[2]\rightarrow 0$ apply $\pi $ and
shift it, to obtain to obtain the exact sequence: $0\rightarrow \pi M[k]%
\overset{\pi (j)[k]}{\rightarrow }\pi E[k]\overset{\pi (p)[k]}{\rightarrow }%
\pi \sigma M[k+2]\rightarrow 0$ if the sequence splits. there exists a map $%
h $:$\pi \sigma M[k+2]\rightarrow \pi E[k]$ such that $\pi (p)[k]h=1.$

The map h is of the form $\pi (t)$ with $t:r^{\ell }\sigma M[k+2]\rightarrow
r^{\ell }E[k]$ . It follows the exact sequence $0\rightarrow r^{\ell }M[\ell
]\overset{j}{\rightarrow }r^{\ell }E[\ell ]\overset{p}{\rightarrow }r^{\ell
+2}\sigma M[\ell +2]\rightarrow 0$ splits. Using Koszul duality $\Phi :K_{%
\mathcal{C}}\rightarrow K_{E(\mathcal{C)}}$ it follows $0\rightarrow \Omega
^{2+\ell }\sigma \Phi (M)\rightarrow \Omega ^{\ell }\Phi (E))\rightarrow
\Omega ^{\ell }\Phi (M)\rightarrow 0$ splits. Therefore $0\rightarrow \Omega
^{2}\sigma \Phi (M)\rightarrow \Phi (E))\rightarrow \Phi (M)\rightarrow 0$
splits, a contradiction.

Let $\pi N[j]$ be an indecomposable object with $N$ Koszul and $f:\pi
N[j]\rightarrow \pi \sigma M[k+2]$ a non isomorphism. As above, there exists
a map $t:N_{\geq \ell }\rightarrow \sigma M[k-j+2]$ such that $\pi t=f$,
which induces a non isomorphism of Koszul objects $t:N_{\geq \ell }[\ell
]\rightarrow \sigma M[$k+$\ell $-j+2$]_{\geq \text{k+}\ell -j+2}.$

The sequence: $0\rightarrow M_{\geq \text{k+}\ell \text{-j}}[$k+$\ell $-j$]%
\overset{j}{\rightarrow }E_{\geq \text{k+}\ell \text{-j}}[$k+$\ell $-j$]%
\overset{p}{\rightarrow }r^{2}\sigma M_{\geq \text{k+}\ell \text{-j}}[$k+$%
\ell $-j+2$]\rightarrow 0$ is almost split. Then there exists a map $%
s:N_{\geq \ell }[\ell ]\rightarrow E_{\geq \text{k+}\ell \text{-j}}[$k+$\ell
$-j$]$ with $ps=t$. It follows $f$ lifts to $\pi E[k]$.

In a similar way we prove the map $\pi M[k]\overset{\pi (j)[k]}{\rightarrow }%
\pi E[k]$ is left almost split.

The remaining claims are clear.
\end{proof}

We will end the paper with the following remark:

\begin{remark}
Consider a locally finite infinite quiver $Q$ and construct the translation
quiver $ZQ$. For each arrow $\alpha $ of $ZQ$ there exits a unique arrow $%
\sigma \alpha $ such that the end of $\sigma \alpha $ coincides with the
starting of $\alpha $, and a unique arrow $\sigma ^{-1}\alpha $ starting at
the end of $\alpha $. We define the following set $\rho $ of relations in $%
ZQ $:

\begin{itemize}
\item[(i)] If $\alpha $ and $\beta $ are arrows with the same end, then $%
\alpha \sigma \alpha -\beta \sigma \beta $ is a relation.

\item[(ii)] If $\gamma $ is an arrow ending at the start of $\alpha $ and
different from $\sigma \alpha $, then $\alpha \gamma $ is a relation.

\item[(iii)] If $\gamma $ is an arrow starting at the end of $\alpha $ and
different from $\sigma ^{-1}\alpha $, then $\gamma \alpha $ is a relation.
\end{itemize}

Consider the category of representations $rep(ZQ,\rho )$ over a field
\textrm{K}\textsl{\ }of the quiver with relations\textsl{\ }$(ZQ,\rho ).$The
category $rep(ZQ,\rho )$ is by construction selfinjective of radical cube
zero. It follows from the arguments in [M] it is also Koszul . The theory
developed in [MVS1], [MVS2], [MVS3] [MVS4] applies to this situation and we
obtain as Ext-category of $E(ZQ,\rho ),$The category of representations of
the quiver $ZQ$ with mesh relations $\eta .$ It follows $rep(ZQ$,$\eta )$ is
an Artin-Shelter regular Koszul category of global dimension two. We do not
know whether or not it corresponds to an Auslander-Reiten quiver of a finite
dimensional algebra, however the above theorems for regular Auslander-Reiten
components hold, in particular they also hold for a connected component of
the stable non regular Ausalnder-Reiten quiver, as considered in [MVS3].
\end{remark}

\end{document}